\documentclass[a4paper, 11pt]{article} %

\usepackage{amsmath} %
\usepackage{amsthm} %
\usepackage{amssymb} %
\usepackage{epsfig} %
\usepackage{psfrag} %

\usepackage[mathscr,mathcal]{eucal} %
\usepackage{enumerate} %
\usepackage[dvipsnames]{xcolor} % Remove this for final version

\usepackage{dsfont}

\newtheorem  {theorem}       {Theorem}[section]
\newtheorem  {lemma}      [theorem]        {Lemma}
\newtheorem  {corollary}  [theorem]    {Corollary}
\newtheorem  {proposition}[theorem]   {Proposition}
\newtheorem  {definition} [theorem]   {Definition}
\newtheorem  {assumption} [theorem]   {Assumption}
\newtheorem* {theorem*}      {Theorem}
\newtheorem* {lemma*}        {Lemma}
\newtheorem* {corollary*}    {Corollary}
\newtheorem* {proposition*}  {Proposition}
\newtheorem* {definition*}   {Definition}
\newtheorem* {remark*}       {Remark}

\newtheorem  {remark}     [theorem]   {Remark}
\newtheorem* {remarks*}      {Remarks}

\newtheorem* {claim*}        {Claim}

\newcommand{\N} {\mathbb N}

\renewcommand{\d} {\,d}  % Ed: this is a time saver, on the basis that we will never want a dot below a symbol and I often accidentally type \d when I mean \,d

\newcommand{\Gage}{\ensuremath{G^{\mathrm{age}} }}
\newcommand{\G}{\ensuremath{\mathcal{G}}}
\newcommand{\E}[1]{\ensuremath{\mathbb{E} \left[#1 \right]}}
\newcommand{\Prob}[1]{\ensuremath{\mathbb{P} \left(#1 \right)}}

\begin{document}

\title{Age evolution in the mean field forest fire model via multitype branching processes}

\author {Edward Crane \thanks{Heilbronn Institute for Mathematical Research and School of Mathematics, University of Bristol. \texttt{Edward.Crane@bristol.ac.uk}}\and Bal\'azs R\'ath \thanks{MTA-BME Stochastics Research Group, Budapest University of Technology and Economics. \texttt{rathb@math.bme.hu}}\and Dominic Yeo \thanks{
Department of Statistics, University of Oxford.
\texttt{dominic.yeo@stats.ox.ac.uk}}}

\maketitle

\begin{abstract}
We study the distribution of ages in the mean field forest fire model introduced by R\'ath and T\'oth. This model is an evolving random graph whose dynamics combine Erd\H{o}s--R\'enyi edge-addition with a Poisson rain of \emph{lightning strikes}. All edges in a connected component are deleted when any of its vertices is struck by lightning.
We consider the asymptotic regime of lightning rates for which the model displays self-organized criticality. The \emph{age} of a vertex increases at unit rate, but it is reset to zero at each burning time. We show that the empirical age distribution converges as a process to a deterministic solution of an autonomous measure-valued differential equation.  The main technique is to observe that, conditioned on the vertex ages, the graph is an inhomogeneous random graph in the sense of Bollob\'as, Janson and Riordan. We then study the evolution of the ages via the multitype Galton--Watson trees that arise as the limit in law of the component of an identified vertex at any fixed time. These trees are critical from the gelation time onwards.
\end{abstract}

\medskip


\newpage
\setcounter{tocdepth}{2}
\tableofcontents

\newpage

\section{Introduction}\label{section_intro}

We study a stochastic model where a net\-work grows steadily, but is subject to occasional destructive events. These can spread widely; but by damaging the connectivity, each destructive event makes it harder for future destruction to propagate through the network. As motivation, consider the effect of fires on a dense forest. Given the right conditions, a fire can destroy all the trees in a region; but afterwards, future fires cannot pass through this area until some trees have regrown.

A probabilistic model of a complex interacting system is considered to be particularly interesting if the effect sizes or the correlations between regions follow a power-law decay. These so-called \emph{critical} phenomena are observed in many complex real-world networks. Seminal work of Bak, Tang and Wiesenfeld \cite{BakTangWiesenfeld} considers models where from a broad range of initial conditions, the dynamics move the system into a class of states where critical phenomena are observed, and then maintain it there. These authors describe this property as \emph{self-organized criticality (SOC)}, and in recent years a wide range of mathematical  models across many contexts have been shown to exhibit such behaviour, see e.g.\ \cite[\S 3]{Jarai} on the SOC of the Abelian sandpile model and \cite{cerf_gorny} for the Curie--Weiss model of SOC.

The \emph{forest fire model} that we study in this paper is a random process taking values in subgraphs of the complete graph $K_n$. The lattice setting, introduced by Drossel and Schwabl \cite{DrosselSchwabl}, where geometry plays a more central role has also been studied: the subcritical forest fire model on $\mathbb{Z}^d$ is constructed by D\"urre \cite{Durre1,Durre2,Durre3} and the critical model on the half-plane is constructed by Graf \cite{Graf}.
However, the rigorous construction of a self-organized critical forest fire model on $\mathbb{Z}^d$ poses
a real mathematical challenge: in fact Kiss, Manolescu and Sidoravicius show in \cite{KMS} (using the earlier results of van den Berg and Brouwer \cite{vdBB}) that it is impossible to construct such a model on $\mathbb{Z}^2$ by starting with standard dynamical percolation and requiring edges to be burned as soon as they belong to an infinite cluster.

 Our focus is the \emph{mean field} setting where edges may appear between any pair of vertices. Without the destructive dynamics of fires, adding edges uniformly at random defines the famous \emph{random graph process} of Erd\H{o}s and R\'enyi \cite{ER}. This process experiences a phase transition between a subcritical regime, where the largest components have logarithmic size relative to the size of the graph, and a supercritical regime, where for large graphs an asymptotically strictly positive proportion of the vertices form a unique \emph{giant component}. During the asymptotically instantaneous transition, the graphs display various \emph{critical} properties, notably a power-law decay in component sizes, which matches behaviour that is observed in many real-world networks not just at phase transition times but at all times.

The \emph{mean field forest fire (MFFF)} process was introduced in \cite{RathToth}. It can be viewed as an adjustment to the Erd\H{o}s--R\'enyi dynamics, which destroys the edges of potential giant components as they are forming, and thus maintains the system in a critical state forever. The following informal definition will be made precise in \S~\ref{SS: MFFF formal}.
\begin{itemize}
\item The model has $n$ vertices, with some (possibly random) initial set of undirected edges at time $0$.
\item Each possible edge joining two vertices appears at rate $1/n$, independently.
\item At rate $\lambda = \lambda(n)$ each vertex is struck by lightning, independently. When a vertex is struck by lightning, the vertex survives but all of the edges in its connected component (or cluster) are instantaneously deleted. Those edges may subsequently reappear.
\end{itemize}
If a lightning hits a vertex $v$ at time $t$, we say that $t$ is a burning time of all the vertices in the connected cluster of $v$.

The most interesting asymptotic regime for the lightning rate is $1/n \ll \lambda(n) \ll 1$. In this regime, clusters of any fixed finite size are destroyed at a negligible rate when $n$ is large. However, the total rate of lightning strikes in the model diverges, so if a cluster of size comparable to $n$ were able to form then it would only survive for time $o(1)$. In \cite{RathToth} it is shown that in this regime the model displays SOC. Subject to some assumptions on the initial conditions, the limiting cluster size distribution in the $n \to \infty$ limit is deterministic and satisfies a coupled system of differential equations called the \emph{critical forest fire equations}; see Propositions \ref{prop_Smol_uniqueness_simplified} and \ref{prop_RathToth_conv} below. The limiting cluster size distribution stays subcritical until a certain \emph{gelation time}. At the gelation time and afterwards, the limiting cluster size distribution is critical in the sense that it has a polynomially decaying tail. No giant component forms; in fact the model is \emph{conservative}, meaning very roughly that at any time $t$, nearly all of the vertices are contained in small connected components.

\medskip

The goal of this paper is to describe the local graph structure of the MFFF at time $t$ as $n \to \infty$ in terms of a multitype branching process, and also to give a simple description of the time evolution of the parameters that govern this multitype branching processes.

The simplest initial condition for the MFFF is the \emph{monodisperse} state, where at time $0$ there are no edges, so the graph consists of $n$ isolated vertices. In this paper we will consider some more general initial conditions, but to state our results informally in this introduction \emph{we will discuss only the monodisperse initial condition}.
At any time $t \ge 0$ each vertex $v$ has an \emph{age} $a^n_t(v)$, which is defined to be the time since it was last burned, or $t$ if it has not yet been burned. Let $\pi^n_t=\frac{1}{n} \sum_{v} \delta_{a^n_t(v)}$ be the empirical measure of these ages.

Our central observation, stated formally as Theorem~\ref{thm_mfff_inhomog}, is that conditional on the ages $a^n_t(v)$ and $a^n_t(w)$ of two vertices $v$ and $w$, the probability that they are joined by an edge at time $t$ is exactly $1-\exp(-a^n_t(v)\wedge a^n_t(w) /n)$. Furthermore, these events are independent for distinct pairs of vertices. So conditional on $\pi^n_t$, the graph seen at time $t$ is an inhomogeneous random graph \emph{(IRG)} in the sense of Bollob\'as, Janson and Riordan \cite{BollobasJansonRiordan}.

Our first main result, Theorem \ref{thm_convergence}, shows that
the empirical age distributions $\pi^n_t$ converge as $n\rightarrow\infty$ to a deterministic limit distribution $\pi_t$. In the course of the proof, we identify $\pi_t$ explicitly as the marginal distribution of $a_t$ in a Markov process $(C_t, a_t)_{t \geq 0}$ that takes values in $E \times [0,\infty)$, where $E$ is the set $\{1,2,3, \dots\}$ equipped with the one-point compactification topology, with $1$ identified with the point at $\infty$. This topology makes the process almost surely continuous from the left at the random times when $C_t$ explodes. The age $a_t$ is simply the time that has elapsed since the last explosion at time $t$, or $a_t=t$ if there is no explosion in $[0,t]$. (See Section \ref{section_cluster_growth_age} for the definitions in the case of general initial conditions.)

$(C_t,a_t)$ is a \emph{McKean--Vlasov process} with jumps, meaning that it is a Markovian Feller process whose infinitesimal generator at time $t$ is defined in terms of the distribution of $(C_t, a_t)$. We show that $(C_t,a_t)$ is the distributional limit of the process $(C^n_t(\rho^n), a^n_t(\rho^n))$, where $\rho^n$ is a uniformly chosen vertex of the graph and $C^n_t(\rho^n)$ is the size of the cluster of $\rho^n$ at time $t$.  The marginal $(C_t)_{t \geq 0}$ is the \emph{cluster growth process} $(C_t)_{t \geq 0}$ introduced in \cite{CraneFreemanToth}, which is a McKean--Vlasov process on its own; it is an explosive pure jump process  with explosion times $\tau_\ell, \ell \in \mathbb{N}$ which returns to $C_{\tau_\ell}=1$ at each explosion time $\tau_\ell$. The jump rate of $C_t$ is $k$ when $C_t=k$ and the jump distribution for jumps at time $t$ is the law of $C_t$ itself.

Remarkably, $(a_t)_{t \geq 0}$ on its own is also a McKean--Vlasov process. It increases at rate $1$ except at a random discrete set of jump times when it jumps down to $0$. We show that the jump rate of $a_t$ at each time $t$ is a function of $a_t$ and the distribution $\pi_t$ of $a_t$. This leads to our second main result, Theorem \ref{thm_age_pde}, which is that $(\pi_t,\,t\ge 0)$ satisfies an autonomous differential equation that we call the \emph{age differential equation}, see \eqref{eq:tempdefnmu}, \eqref{eq:ageDE_intro} and \eqref{eq:phiformula_intro} below. The well-posedness of this differential equation is proved in the companion paper~\cite{Crane2}, so in fact $\pi_t$ is uniquely determined from $\pi_0$ by the differential equation. However, we stress that in this paper we determine $\pi_t$ from $\pi_0$ using the cluster growth process.

Recall from \cite{RathToth} that there exists a so-called \emph{gelation time} $t_{\mathrm{gel}} \ge 0$, at which the model makes a phase transition from subcritical to critical behaviour.
%The results of \cite{RathToth,CraneFreemanToth} were stated only for subcritical initial conditions, when $t_{\mathrm{gel}} \ge 0$. We impose here an extra costraint on the initial condition, that it is a sample of an \emph{age-driven inhomogeneous random graph} (see Definition \ref{def:agedrivenIRG}), determined in the $n \to \infty$ limit by an initial age distribution $\pi_0$. Because of this, we are able to allow a critical limiting initial condition, in which case we define $t_{\mathrm{gel}} = 0$.
For $0 \le t < t_{\mathrm{gel}}$ the age of each vertex simply increases at rate $1$ unless it burns before $t_{\mathrm{gel}}$. However, only an asymptotically negligible proportion of the vertices burn before $t_{\mathrm{gel}}$, so the limiting age distribution satisfies the simple transport equation
\begin{equation}\label{eq:ageDEprecrit}
\frac{\mathrm{d}\pi_t}{\mathrm{d}t} = -\delta_0' * \pi_t\,.
\end{equation}
Here $\delta_0'$ is the derivative of the Dirac delta at $0$, so this statement is an equality of Schwarz distributions. In other words, for $0 \le t \le t_{\mathrm{gel}}$ and for any Borel set $A \subseteq [0,\infty)$,
$$ \pi_t(A) = \pi_0(\{ x - t: x \in A \})\,. $$
The situation is more interesting for $t\ge t_{\mathrm{gel}}$, when the model is critical. Then, for each such $t$, there exists a unique non-negative, continuous and non-decreasing function $s \mapsto \theta_t(s)$ satisfying $\int \theta_t(s) \,d\pi_t(s)  = 1$ and
\begin{equation}\label{eq:tempdefnmu}
\theta_t(s) = \int_0^\infty \theta_t(u)\left(u\wedge s\right) \mathrm{d}\pi_t(u),\;\;\; s\in[0,\infty].
\end{equation}
For a fixed $t \geq t_{\mathrm{gel}}$,  we denote by $\theta_t \pi_t$ the probability measure absolutely continuous with respect to $\pi_t$ with Radon--Nikodym derivative $\theta_t$. We denote by $\varphi(t)$ the limiting total rate of mass of burnt vertices at time $t$ (see Proposition \ref{prop_Smol_uniqueness_simplified} for details). In fact, $\varphi(t)$ is also equal to the explosion rate of $C_t$.
Then for $t > t_{\mathrm{gel}}$, $\pi_t$ satisfies the following distribution-valued differential equation:
\begin{equation}\label{eq:ageDE_intro}
\frac{\mathrm{d}\pi_t}{\mathrm{d}t} =
- \delta_0' * \pi_t - \varphi(t)\theta_t \pi_t +\varphi(t)\delta_0\,.
\end{equation}
We prove \eqref{eq:ageDE_intro} by identifying it with the Kolmogorov forward equation of the McKean--Vlasov process $(a_t)_{t \geq 0}$.
Let us give interpretations of the individual terms.   As in \eqref{eq:ageDEprecrit}, the transport term $-\delta_0'*\pi_t$ describes the constant, deterministic growth of the ages of all vertices not instantaneously involved in fires. The term $-\varphi(t) \theta_t \pi_t$ describes the change due to the removal of burning vertices. The final term $\varphi(t)\delta_0$ corresponds to the fact that all vertices burned at time $t$ reappear with age zero.  We will show that criticality of the forest fire equations corresponds to criticality of the IRG which describes the system conditional on the ages.
The local structure of the IRG seen at time $t$ is well-approximated by a multitype Poisson branching process (see Definition \ref{def:branchingprocess} and Theorem \ref{thm_local_limit_mfff}\eqref{local_weak_limit_in_ffm}), and $\theta_t$ is the right eigenfunction corresponding to the principal eigenvalue $\lambda_t=1$ of the branching operator. The heuristic idea is that $\theta_t \pi_t$ approximates the distribution of ages in very large components of the IRG, which account for nearly all of the burning vertices (see Section \ref{subsubsection_frozen_IRG_Dominic} for details).  In fact we do not prove this global statement but instead obtain~\eqref{eq:ageDE_intro} by considering the local limit, showing that the rate at which $a_t$ jumps down to $0$ when $a_t = s$ is $\varphi(t)\theta_t(s)$. By careful analysis of the multitype branching process we show
\begin{equation}\label{eq:phiformula_intro}
\varphi(t) = \left(\int \theta_t(s)^3\,\mathrm{d}\pi_t(s)\right)^{-1}\,.
\end{equation}
Combining~\eqref{eq:tempdefnmu},~\eqref{eq:ageDE_intro} and~\eqref{eq:phiformula_intro}, we have an autonomous differential equation, describing the evolution of $\pi_t$  in terms of $\pi_t$, without reference to $t$.

%Differential equations such as \eqref{eq:ageDE_intro} driven by the principal eigenvector of an operator-valued function of the current state of the system are, to our knowledge, novel, although a finite-dimensional analogue appears in the third author's thesis \cite{Yeo}. Unlike in that simpler case of finitely many types, we do not attempt in this paper to prove that the differential equation is well-posed. \begin{color}{blue} Refer to Ed's solo paper?\end{color}

\section{Statements of results}
\label{section_statements_of_results}
We precisely introduce the mean field forest fire model in \S \ref{SS: MFFF formal}, age-driven inhomogeneous random graphs in \S \ref{subsection_intro_age_driven_IRG}, age-driven multitype branching processes in \S \ref{SS: age-driven bps}, and the critical forest fire equations in \S \ref{subsection_intro_forest_fire_eqs}. These provide the necessary background to understand the statements of our main results, stated in \S \ref{subsection_intro_main_results}.
In \S\ref{subsub_other_work} we set our results in the context of related literature. In \S\ref{subsection_open_q} we pose some open questions.
 In \S \ref{subsection_overview} we give an overview of the contents of the rest of the paper.

\subsection{The mean field forest fire}
\label{SS: MFFF formal}

We will always use the following definition of a mean field forest fire process on vertex set $[n]:=\{1,\ldots,n\}$, with lightning rate $\lambda$, following \cite{RathToth}. We refer to this model as $\mathrm{MFFF}(n,\lambda)$.

\begin{definition}[Graphical construction of $\mathrm{MFFF}(n,\lambda)$]
\label{def:graphical_mfff}
Let $\mathcal{E}$ be a Poisson point process (PPP) of rate $1/n$ on $\binom{[n]}{2} \times [0,\infty)$, and $\Lambda$ be an independent PPP of rate $\lambda$ on $[n] \times [0,\infty)$. These PPPs will determine edge arrivals and lightning strikes, respectively, at times given by their second coordinates. Given some (possibly random) initial graph $\mathcal{G}^n_0$ with vertex set $[n]$, we construct the random graph-valued process $\left( \mathcal{G}^n_t \right)_{t=0}^\infty$ started from $\mathcal{G}^n_0$ as follows, working through the points of $\mathcal{E} \cup \Lambda$ in increasing order of their time coordinates:
\begin{enumerate}[(i)]
 \item add the edge $\{i,j\}$ to obtain $\mathcal{G}^n_{t}$ from $\mathcal{G}^n_{t-}$ if $\{i,j\} \times t$ is a point of $\mathcal{E}$, and edge $\{i,j\}$ is not already present in $\mathcal{G}^n_{t-}$ (and otherwise do nothing),
 \item erase the edges of the connected component $\mathcal{C}^n_{t-}(i)$ of vertex $i$ in $\mathcal{G}^n_{t-}$ to obtain $\mathcal{G}^n_{t}$ from $\mathcal{G}^n_{t-}$ if $\{i\} \times t$ is a point of $\Lambda$.
 \end{enumerate}
In the latter case, we say that the vertices of $\mathcal{C}^n_{t-}(i)$ are burned at time $t$.
\begin{remark}\label{lightningremark}
We will always view $\mathrm{MFFF}(n,\lambda)$ as a graph-valued process coupled with its lightning process $\Lambda$. This will simplify the choice of probability space in subsequent arguments.
\end{remark}
\end{definition}

In \cite{RathToth} the asymptotic behaviour of MFFF processes is studied as $n\rightarrow\infty$, where the lightning rate $\lambda(n)$ satisfies the \emph{critical relation}
\begin{equation}\label{lightning_critical_regime}
1/n \ll \lambda(n)\ll 1.
\end{equation}
 We will assume throughout that this critical relation holds. Informally, this has the effect that small components are negligibly affected by lightning, whereas components of size $\Theta(n)$ are burned into singletons instantly. Roughly speaking, this means that such giant components never appear.

Let us stress that the earlier results of \cite{RathToth} and \cite{CraneFreemanToth} concerning the MFFF (which will be recalled in \S \ref{subsection_intro_forest_fire_eqs} and \S \ref{subsection_cgp_w_a}, respectively)  are about the \emph{sizes} of connected clusters in the process $(\mathcal{G}^n_{t})$, but the main results of this paper are about the \emph{graph structure} itself for $(\mathcal{G}^n_{t})$.

\subsection{Age-driven inhomogeneous random graphs}
\label{subsection_intro_age_driven_IRG}

We will make a connection between the $\textup{MFFF}(n, \lambda)$ and the theory of inhomogeneous random graphs.
In order to do so, we need to study the MFFF augmented with extra information about the \emph{ages} of vertices.

\begin{definition}[Ages in the mean field forest fire model] \label{def:ages}
For $a^n_0(i) \in [0,\infty), \, i \in [n]$, we define a \emph{mean field forest fire with ages} $\mathrm{MFFFA}(n,\underline a^n_0,\lambda)$ as follows. We take $\mathrm{MFFF}(n,\lambda)$ as in Definition \ref{def:graphical_mfff}, with some initial graph $\mathcal{G}^n_0$, along with its lightning process $\Lambda$, and we call $a^n_0(i)$ the \emph{initial age} of vertex $i$.

For $t>0$, denote by $a^n_t(i)$  the \emph{age} of vertex $i$ at time $t$, i.e., $a^n_t(i)$ is $a^n_0(i)+t$ if $i$ did not burn during the time interval $[0,t]$, otherwise  $a^n_t(i)$ is $t-s$, where
$s$ is the last burning time of $i$ on the time interval $[0,t]$.
  We write $\underline{a}^n_t:=(a^n_t(i))_{i \in [n]}$ for the vector of ages, and
\begin{equation}\label{pi_n_t_empirical_measure}
\pi^n_t:= \frac{1}{n}\sum_{i=1}^n \delta_{a^n_t(i)}
\end{equation} for the empirical measure of ages.
\end{definition}

So the age of a vertex in $\mathrm{MFFFA}(n,\underline{a}^n_0,\lambda)$ increases deterministically at unit rate, but is reset to zero at each burning time.

We introduce a special case of the class of inhomogeneous random graphs introduced by Bollob\'as, Janson and Riordan \cite{BollobasJansonRiordan}.
\begin{definition}[Age-driven IRG]\label{def:agedrivenIRG}
Given a  (possibly random) sequence $$\underline{a}^n=(a^n(1),\ldots,a^n(n))\in[0,\infty)^n\,,$$ we define
 the \emph{age-driven inhomogeneous random graph (IRG)} $\Gage(n,\underline{a}^n)$ on the vertex set $[n]$ as follows.  Conditional on $\underline{a}^n$,
\begin{itemize}
\item $a^n(i)$ is said to be the \emph{age} of vertex $i\in[n]$;
\item independently for different edges, the edge $\{i,j\}\in\binom{[n]}{2}$ is present with probability $$1-\exp\left( -\frac{a^n(i)\,\wedge\, a^n(j)}{n}\right)\,.$$
\end{itemize}
\end{definition}

Our main results concern the evolution of the process of ages in the MFFF model started from an age-driven IRG. So, although Definition \ref{def:ages} makes sense in the generality stated, from now on, we will usually make the following assumption.

\begin{assumption}[MFFFA started from age-driven IRG]\label{assumption_age_driven_initial_state}

$ $

The $\mathrm{MFFFA}(n,\underline a^n_0,\lambda)$ is started from the initial graph $\mathcal G^n_0\stackrel{d}= \Gage(n,\underline a^n_0)$.
\end{assumption}

The central idea underpinning our results is that the dynamics of the MFFFA preserve the class of age-driven IRGs.

\begin{theorem}[$(\mathcal{G}^n_t)$ and $(\underline{a}^n_t)$ are intertwined]\label{thm_mfff_inhomog}
If $(\mathcal{G}^n_t)$ is a MFFFA process with initial condition distributed as $\Gage(n,\underline a^n_0)$ then, for any $t\ge 0$,
\begin{equation}\label{eq:mfff_inhom_rg}
 \text{conditional on $(\underline a^n_s,\,s\in[0,t])$, we have $\mathcal{G}^n_t\stackrel{d}=\Gage(n,\underline a^n_t)$.}
\end{equation}
\end{theorem}

A version of this result with weaker conditioning and with monodisperse initial condition appears in \cite{Yeothesis} as Lemma 5.8. We will give a shorter proof here in \S\ref{subsection_ff_preserves_IRG}.

\begin{remark}\label{remark_MFFF_intertwined}
It follows from \eqref{eq:mfff_inhom_rg} that $(\underline{a}^n_t)_{t \geq 0}$ is a time-homogeneous Mar\-kov  process, and by the vertex exchangeability of the MFFF dynamics $(\pi^n_t)_{t \geq 0}$ is also a time-homogeneous Markov process. Using the notation of \cite[\S 3.2]{swart_intertwining},
the Markov process $(\mathcal{G}^n_t)$ is \emph{intertwined} on top of $(\underline{a}^n_t)$ (c.f.\ \cite[Proposition 3.4]{swart_intertwining}).
 \end{remark}

In order to state some of our main results, we need to define the notion of local weak limits of age-driven IRGs: this is what we will do in the next section.

\subsection{Age-driven multitype branching processes}\label{SS: age-driven bps}

We introduce a family of branching process trees, also augmented with ages, which appear as local weak limits of age-driven IRGs.

\begin{definition}[Age-driven multitype branching process (MBP)] \label{def:branchingprocess}

$ $

Given a Borel probability measure $\pi$ on $[0,\infty)$ we define $T^\pi$, a multitype Galton--Watson tree with vertex set $V(T^\pi)$ and \emph{with ages} $a:V(T^\pi)\rightarrow[0,\infty)$, as follows. The root $\rho$ has age $a(\rho)\stackrel{d}=\pi$, and then any vertex of age $s$ has an independent set of offspring vertices with ages given by a Poisson random measure with intensity $(s\wedge u)\mathrm{d}\pi(u)$.

 We denote by $|T^\pi|$ the cardinality of $V(T^\pi)$.

We  write $T^\pi_s$ for the random tree $T^\pi$ constructed in the same way but starting with a root vertex of deterministic age $s$. We also define $T^\pi_\infty$ to be the random tree whose root has infinite age, meaning that its offspring have ages described by a Poisson random measure with intensity $u \,\mathrm{d}\pi(u)$. If $\int u \,\mathrm{d}\pi(u) < \infty$ then this is a finite intensity measure so the root almost surely has only finitely many offspring.

We write $\mathcal L_\pi$ for the branching operator, that is the Perron--Frobenius operator defined for $f \in L^1(\pi)$ by
\begin{equation}\label{L_def_statement}
\mathcal L_\pi f(s):= \int f(u)(s\wedge u) \mathrm{d}\pi(u).
\end{equation}
When $\mathcal{L}_\pi$ maps $L^2(\pi)$ into itself, denote by $\|\mathcal{L}_\pi\|$ the $L^2(\pi)$-operator norm of $\mathcal{L}_{\pi}$.
\end{definition}

$\mathcal{L}_\pi$ has a probabilistic interpretation. The root of $T^{\pi}_{s}$ has a random number $K$ of offspring with ages $a_1 \le \dots \le a_K$, and
\[\mathcal L_\pi f(s)=\mathbb{E}\left( f(a_1)+\dots+ f(a_K) \right) .\]

\begin{lemma}[Normalized principal eigenfunction]
\label{lemma_L_operator_basic_properties}
Suppose $\pi$ is a Borel probability measure on $[0,\infty)$ such that $0<\int x\,d\pi(x) < \infty$.  Then the Perron--Frobenius operator
 $\mathcal{L}_\pi$ is a compact self-adjoint operator on $L^2(\pi)$ whose principal eigenvalue $\lambda$ satisfies $0 < \lambda = \|\mathcal{L}_\pi\|$. Moreover there exists a unique nonnegative eigenfunction $\theta\in L^2(\pi)$ for which $\mathcal{L}_\pi \theta=\lambda \theta$ and $\int \theta(x)\mathrm{d}\pi(x)=1$.
\end{lemma}
We will prove a more detailed result in \S \ref{subsection_L_op_prop} that includes the statements of Lemma~\ref{lemma_L_operator_basic_properties}.

 In \S \ref{subsec_multitype_offspr_decay} we combine \cite[Lemma 6.1]{BollobasJansonRiordan} (characterizing supercriticality) and some further arguments to prove the following result.
\begin{proposition}[Trichotomy of $|T^\pi|$]\label{prop_multitype_trichotomy}
Let $\pi$ be a Borel probability measure on $(0,\infty)$ with $\int x\,d\pi(x) < \infty$. Then the following trichotomy holds:
\begin{itemize}
\item  $\|\mathcal{L}_\pi\|< 1$ if and only if $\E{|T^\pi|}<\infty$;
\item  $\|\mathcal{L}_\pi\|= 1$ if and only if $\Prob{|T^\pi|<\infty}=1$ and $\E{|T^\pi|}=\infty$;
\item  $\|\mathcal{L}_\pi\|>1$ if and only if $\Prob{|T^\pi|=\infty}>0$.
\end{itemize}
We will say that $\pi$ is \emph{age-subcritical}, \emph{age-critical}, and \emph{age-supercritical} when $\|\mathcal{L}_\pi\|<1$, $\|\mathcal{L}_\pi\|=1$, and $\|\mathcal{L}_\pi\|>1$, respectively.
\end{proposition}
In fact, the results of Proposition \ref{prop_multitype_trichotomy} hold under slightly weaker conditions on $\pi$. Further details are deferred to \S \ref{subsec_multitype_offspr_decay}.

To make sense of approximating age-driven IRGs by age-driven multitype Galton--Watson trees, we briefly introduce one version of \emph{local weak convergence} of random graphs, as studied by Aldous, Benjamini and Schramm. See \cite{van der Hofstad 2} for a comprehensive account of this.

\begin{definition}[Local weak convergence]\label{def:localweakconv}
Let $(G^n,\, n\ge 1)$ be a sequence of random graphs where $G^n$ has vertex set $[n]$, and let $\mathcal{B}^n_k(w)$ be the $k$-neighborhood in $G^n$ of a vertex $w\in[n]$, viewed as a graph rooted at $w$. Then we say $G^n$ converges in probability in the local weak sense to the random rooted graph $(G,\rho)$ if for every rooted graph $(H,v)$, and every $k\ge 1$,
\begin{equation}\label{eq:localweakprob1}\frac{1}{n}\sum_{w\in[n]} \mathds{1}[ \mathcal{B}^n_k(w)\simeq (H,v) ] \;\stackrel{\mathbb{P}}\longrightarrow \; \Prob{\mathcal{B}_k(\rho) \simeq (H,v)},
\qquad n \to \infty,
\end{equation}
 where $\mathcal{B}_k(\rho)$ is the $k$-neighborhood of $\rho$ in $G$ and the relation $\simeq$ denotes the root-preserving isomorphism of rooted graphs.
Let us denote by $\rho^n$ a vertex which is independent of $G^n$ and uniformly distributed in $[n]$.
As a consequence of \eqref{eq:localweakprob1}, we obtain the weaker condition
\begin{equation}\label{eq:localweakdistexch}
\Prob{ \mathcal{B}^n_k(\rho^n) \simeq (H,v)} \;  \longrightarrow  \; \Prob{\mathcal{B}_k(\rho)\simeq (H,v)}, \qquad n \to \infty,
\end{equation}
for every rooted graph $(H,v)$, which is called convergence in distribution in the local weak sense. (See \cite[Def. 2.7]{van der Hofstad 2}.)
\end{definition}

\begin{definition}[Convergence in probability of age distributions]\label{def_conv_of_measures_simple}
 If $\pi$ is a probability measure on $\mathbb{R}$ and $\pi^1,\pi^2,\dots$ is a sequence of (possibly random) probability measures on $\mathbb{R}$ then we say that
  $\pi^n \stackrel{\mathbb{P}}{\Rightarrow} \pi$ as $n \to \infty$
  if $\int f(s)\, \mathrm{d} \pi^n(s) \stackrel{\mathbb{P}}{\rightarrow} \int f(s)\, \mathrm{d} \pi(s)$ as $n \to \infty$ for any bounded continuous function $f:\mathbb{R} \to \mathbb{R}$.
\end{definition}

\begin{proposition}[IRG locally converges to MBP]\label{prop:localweakconv}
Let $G^n\stackrel{d}=\Gage(n,\underline{a}^n)$ be a sequence of age-driven IRGs, for which the empirical age distributions satisfy $\pi^n\stackrel{\mathbb{P}}{\Rightarrow} \pi$,
where $ \int x\,d\pi(x) < \infty$. Then $G^n$ converges in probability in the local weak sense to the age-driven multitype branching process  tree $T^\pi$.
\end{proposition}

We deduce the proof of Proposition \ref{prop:localweakconv} from \cite[Theorem 3.11]{van der Hofstad 2} by a truncation argument in \S \ref{subsection_local_weak_convergence}.

\subsection{Critical forest fire equations}
\label{subsection_intro_forest_fire_eqs}

In this section we recall the main results of \cite{RathToth}, which concern the hydrodynamic limit of the component size vector of the MFFF. The assumptions of this paper on the initial state of the MFFF are different from those of \cite{RathToth}, so the adaptation of the results of \cite{RathToth} to our setting will require extra work.

\begin{definition}[Component size vector in $\textup{MFFF}(n, \lambda)$]\label{def_v_n_k_t}
For any $t \geq 0$ and for each $k=1,2,\dots$  let us define
\begin{equation}\label{eq:defnvn}
v^n_k(t):= \frac{1}{n}\#\Big\{ \text{vertices in size $k$ components at time }t\Big\}.
\end{equation}
We write $\underline{v}^n(t)$ for the vector $(v_k^n(t))_{k=1}^\infty$.
\end{definition}

The central observation of \cite{RathToth} is that limits of $\underline{v}^n(t)$ in a sequence of forest fires in the self-organized critical regime \eqref{lightning_critical_regime} satisfy a family of coupled differential equations. Our assumptions on the initial state will be different from those of
 \cite[Theorem 1]{RathToth}, but the conclusions will be the same.

\begin{assumption}[$\underline{v}(0)$ is the law of $|T^{\pi_0}|$] \label{assumption_v_k_0_comes_from_BP}
Suppose that $\pi_0$ is a Borel probability measure on $[0,\infty)$ satisfying
 $\int u \,d\pi_0(u) < \infty$, moreover $\pi_0$ is either age-critical or age-subcritical. Let us define $v_k(0)=\mathbb{P}( |T^{\pi_0}|=k )$ for
$k=1,2,\dots$ and let $\underline{v}(0)=\left( v_k(0) \right)_{k=1}^{\infty}$.
\end{assumption}

\begin{proposition}[Critical forest fire equations]\label{prop_Smol_uniqueness_simplified}
If the initial condition $\underline{v}(0)$  satisfies Assumption \ref{assumption_v_k_0_comes_from_BP}, then
 the critical forest fire equations
 \begin{eqnarray}\label{smol_ff_eq}
  \frac{\mathrm{d}}{\mathrm{d}t} v_k(t) & = &    \frac{k}{2} \sum_{\ell=1}^{k-1} v_\ell(t)v_{k-\ell}(t)  -k v_k(t), \qquad
 k \geq 2,\\ \label{smol_bc}
 \sum_{k=1}^\infty v_k(t) & = & 1,
  \end{eqnarray}
have a unique solution $\underline{v}(\cdot)$, which also has the following properties:
\begin{enumerate}
\item
\begin{equation}\label{varphi_def}
\frac{\mathrm{d}}{\mathrm{d} t } v_1(t)= \begin{cases} - v_1(t) & \text{ if } \quad 0 \leq t <t_{\mathrm{gel}}, \\
-v_1(t)+\varphi(t) & \text{ if } \quad t > t_{\mathrm{gel}},
\end{cases}
\end{equation}
where the \emph{gelation time} $t_{\mathrm{gel}}$ is defined by
 \begin{equation}\label{def_gelation_time}
t_{\mathrm{gel}} = \left(\sum_{\ell=1}^\infty \ell v_\ell(0)  \right)^{-1} \quad \text{(using the convention $\infty^{-1}=0$)},
\end{equation}
 and
\begin{equation}\label{varphi_positive_loc_lip}
\varphi: \, [t_{\mathrm{gel}}, +\infty)  \to (0,+\infty) \text{ is locally Lipschitz-continuous.}
\end{equation}
\item  $v_k(t)$ decays exponentially as $k \to \infty$ if $0\leq t < t_{\mathrm{gel}}$, but
\begin{equation}\label{v_k_polynomial_decay}
\sum_{\ell=k}^\infty v_\ell(t) \approx \sqrt{ \frac{2 \varphi(t)}{\pi}} k^{-1/2} \quad \text{as} \quad k \to \infty   \quad \text{for any} \quad t \geq t_{\mathrm{gel}}
\end{equation}
(where $a_k \approx b_k$ is a shorthand for $ \lim_{k \to \infty} \frac{a_k}{b_k} =1$).
\end{enumerate}
\end{proposition}

We will prove Proposition \ref{prop_Smol_uniqueness_simplified} in \S \ref{subsection_smol_uniqueness}.
In \cite[Theorem 1]{RathToth} the same conclusions are made under the assumption that $\sum_{\ell=1}^\infty \ell^3 v_\ell(0) <+\infty$
(which neither implies nor is implied by Assumption \ref{assumption_v_k_0_comes_from_BP}).

\begin{remark}\label{remarks_varphi}

$ $

\begin{enumerate}[(i)]
\item \label{remark_varphi_control}
 As soon as one picks any \emph{control function}
$\varphi(\cdot)$, one can construct $v_1(\cdot)$ by solving \eqref{varphi_def}, and then $v_k(\cdot)$ for $k=2,3,\ldots$ inductively using \eqref{smol_ff_eq}. Proposition \ref{prop_Smol_uniqueness_simplified} states that there is a unique $\varphi(\cdot)$
such that the resulting family $\left( v_k(t) \right)_{k=1}^\infty$ of functions satisfies \eqref{smol_bc} for all $t \geq 0$.
\item \label{remark_explicit_smol_stationary}
Currently there is no known explicit solution of (\eqref{smol_ff_eq}+\eqref{smol_bc}) except for the unique stationary solution
$v_k(t)=v_k(0)=\frac{2}{k} \binom{2k-2}{k-1}4^{-k}$, $k \geq 1$. Note that this $\underline{v}(0)$ satisfies Assumption \ref{assumption_v_k_0_comes_from_BP}, see Remark \ref{remark_about_v_k_t_pi_t}\eqref{remark_statiorary_pi}.
\end{enumerate}
\end{remark}

 The connection of the critical forest fire equations  to the MFFF is given by the following result.

\begin{proposition}[Convergence of component size vector]\label{prop_RathToth_conv}
Suppose that $\underline{v}(0)$  satisfies Assumption \ref{assumption_v_k_0_comes_from_BP}.
Let $\underline{v}(\cdot)$ denote the unique solution to the critical forest fire equations (\eqref{smol_ff_eq}+\eqref{smol_bc}) with initial state $\underline{v}(0)$.
Let $(\mathcal G^n,\,n\ge 1)$ be a sequence of $\textup{MFFF}(n, \lambda(n))$ processes as defined in Definition \ref{def:graphical_mfff}, for which \eqref{lightning_critical_regime} holds, and
 $| v^n_k(0)-v_k(0)| \stackrel{\mathbb{P}}{\rightarrow} 0$ for all $k \in \mathbb{N}$ as $n \to \infty$.
Let $\underline{v}^n(\cdot)$ be defined by  \eqref{eq:defnvn}.
 Then for any $t_{\max} \in [0,\infty)$ we have
%\begin{equation}\label{sup_conv_v_n_k_t_to_v_k_t}
% \sup_{0 \leq t \leq t_{\max}} \sum_{k=1}^\infty | v^n_k(t)-v_k(t)| \stackrel{\mathbb{P}}{\rightarrow} 0,
% \qquad n \to \infty.
%\end{equation}
\begin{equation}\label{sup_conv_v_n_k_t_to_v_k_t}
\sup_{k \geq 1} \sup_{0 \leq t \leq t_{\max}}  | v^n_k(t)-v_k(t)| \stackrel{\mathbb{P}}{\rightarrow} 0,
 \qquad n \to \infty.
\end{equation}
\end{proposition}
The above result for fixed $t$ and fixed $k$ is proved in \cite[Theorem 2]{RathToth}, while the stronger \eqref{sup_conv_v_n_k_t_to_v_k_t}
 is proved in \cite[Theorem 1.5]{CraneFreemanToth}.
   Note that the assumptions under which we prove the uniqueness of
the solution of the critical forest fire equations (c.f.\ Proposition \ref{prop_Smol_uniqueness_simplified}) are different from those of the analogous uniqueness result
in \cite{RathToth} (c.f.\ Theorem 1 therein), but the proofs of \cite[Theorem 2]{RathToth} and \cite[Theorem 1.5]{CraneFreemanToth} only take the uniqueness result as an input, so the same proof also gives our Proposition \ref{prop_RathToth_conv}.

Proposition \ref{prop_RathToth_conv} describes the hydrodynamic limit of the component size vector $\underline{v}^n(t)$ of the MFFF in terms of the critical forest fire equations.
We want to give a similar description of the hydrodynamic limit of the empirical age measure $\pi^n_t$, which will also allow us to identify the
 local weak limit of the graph $\mathcal{G}^n_t$.

\subsection{Main results}
\label{subsection_intro_main_results}

Throughout \S \ref{subsection_intro_main_results} we enforce Assumption \ref{assumption_age_driven_initial_state}, that the MFFFA is started from an age-driven IRG.

Our first main result is a limit theorem for the empirical measure of ages in a family of $\mathrm{MFFFA}$ processes whose initial empirical age distributions converge in probability.

\begin{theorem}[Convergence in probability of empirical age distribution]\label{thm_convergence}{\quad}\\
Let $(\G^n_t,\,t \geq 0)$ be a family of $\mathrm{MFFFA}(n,\underline{a}^n_0,\lambda(n))$ processes that satisfy Assumption \ref{assumption_age_driven_initial_state}.
Suppose the initial empirical age measures %(see \eqref{pi_n_t_empirical_measure})
satisfy $\pi^n_0 \stackrel{\mathbb{P}}{\Rightarrow} \pi_0$,
where  $\int \!x\,d\pi_0(x) < \infty$ and $\pi_0$ is either age-critical or age-subcritical.
Assume that $\lambda(n)$ satisfies
\eqref{lightning_critical_regime}.
Then for all $t \geq 0$ we have
$\pi^n_t \stackrel{\mathbb{P}}{\Rightarrow} \pi_t$ as $n \to \infty$,
where $(\pi_t)_{t \geq 0}$ is a family of probability measures determined by $\pi_0$ that satisfies $\int x\,d\pi_t(x) < \infty$ for all $t \geq 0$.
\end{theorem}
The proof of Theorem~\ref{thm_convergence} is carried out in \S \ref{section_cluster_growth_age} and is completed in \S \ref{subsection_conc_emp_age_distr}.
As we have already mentioned in \S \ref{section_intro}, we will identify  $\pi_t$ as the distribution of the age $a_t$ of the watched vertex at time $t$ of the \emph{cluster growth process with ages} $(C_t,a_t)_{t \geq 0}$, the law of which is determined uniquely by $\pi_0$. We refer the reader interested in the rigorous
definition of $(C_t,a_t)_{t \geq 0}$ to the introduction of \S \ref{section_cluster_growth_age} for details.
%Please don't delete the next sentence: we have to state here that $\int x\,d\pi_t < \infty$ in order for the statements of the other main theorems to make sense.
%  It follows that $$\int x\,d\pi_t(x) \le t+ \int x \,d\pi_0(x) < \infty\,,$$
%  so that Lemma~\ref{lemma_L_operator_basic_properties} and Proposition~\ref{prop_multitype_trichotomy} apply to $\pi_t$.

Our next main result concerns the local weak limit of the graph of the MFFFA process at time $t$.
\begin{theorem}[The local weak limit of MFFFA at time $t$]\label{thm_local_limit_mfff}
Let $(\G^n_t,\, t \geq 0 )$ and $\pi_0$ satisfy the conditions of Theorem \ref{thm_convergence}. Then the following hold.
\begin{enumerate}[(i)]
\item \label{local_weak_limit_in_ffm} For any $t \geq 0$, $\mathcal{G}^n_t$ converges in probability in the local weak sense to $T^{\pi_t}$ as $n\rightarrow\infty$,
where $\pi_t$ is defined in Theorem \ref{thm_convergence}.
\item  \label{v_k_t_using_pi_t} The family of functions $t \mapsto v_k(t), \, k=1,2,\dots$ defined by
\begin{equation}\label{eq_v_k_t_using_pi_t}
v_k(t)= \Prob{\left|T^{\pi_t}\right|=k}, \qquad t\geq0
\end{equation}
coincides with the unique solution of  the critical forest fire equations (\eqref{smol_ff_eq}+\eqref{smol_bc}) guaranteed by
Proposition \ref{prop_Smol_uniqueness_simplified}.
\item \label{age_crit_after_gel} $\pi_t$ is age-subcritical for $t < t_{\mathrm{gel}}$ and $\pi_t$ is age-critical for $t \geq t_{\mathrm{gel}}$ (where $t_{\mathrm{gel}}$ was introduced in Proposition \ref{prop_Smol_uniqueness_simplified}).
\end{enumerate}
\end{theorem}
We will prove Theorem~\ref{thm_local_limit_mfff} in \S \ref{subsection_proof_of_thm_local_limit_mfff}.
%Note that one can write down an explicit integral formula that expresses the r.h.s.\ of \eqref{eq_v_k_t_using_pi_t} in terms of $\pi_t$; see
%Remark~\ref{R: explicit density}.

Our final main result is the precise statement of equations \eqref{eq:ageDE_intro} and \eqref{eq:phiformula_intro} which describe the driving forces behind the time evolution of $\pi_t$.
\begin{theorem}[Age differential equations]\label{thm_age_pde}Assume $\int u \,d\pi_0(u) < \infty$.
Consider the family $(\pi_t)$ from Theorem \ref{thm_convergence}, and for any $t \geq t_{\mathrm{gel}}$ denote by $\theta_t$ the eigenfunction of $\mathcal{L}_{\pi_t}$ corresponding to the eigenvalue $\lambda=1$, as in Theorem \ref{thm_local_limit_mfff}\eqref{age_crit_after_gel} and Lemma \ref{lemma_L_operator_basic_properties}.
 For every compactly supported and continuously differentiable test function $f: [0,\infty) \to \mathbb{R}$,
\begin{multline}\label{eq:ageDE}
\frac{\partial}{\partial t} \int f(s)\mathrm{d}\pi_t(s) = \\
 \begin{cases} \int f'(s)\mathrm{d}\pi_t(s)  & \quad t < t_{\mathrm{gel}}\\
 \int f'(s)\mathrm{d}\pi_t(s) - \int f(s)\varphi(t)\theta_t(s)\mathrm{d}\pi_t(s) + \varphi(t) f(0)
  &\quad t> t_{\mathrm{gel}}
 \end{cases},
\end{multline}
where $\varphi(\cdot)$ denotes the control function appearing in equation~\eqref{varphi_def} that corresponds to the solution of (\eqref{smol_ff_eq}+\eqref{smol_bc}) arising from Theorem \ref{thm_local_limit_mfff}\eqref{v_k_t_using_pi_t}.
 For $t \geq t_{\mathrm{gel}}$  we have
\begin{equation}\label{eq:phiexpression}
\varphi(t) = \left( \int_0^\infty \theta_t(s)^3\mathrm{d}\pi_t(s) \right)^{-1}\,.
\end{equation}
\end{theorem}
We prove Theorem \ref{thm_age_pde} in \S \ref{subsection_proof_of_age_pde}. Our proof crucially relies on the precise asymptotics of the generating function of the total number of vertices in the multitype branching process tree $T^{\pi_t}_s$
that we derive in \S \ref{sec_properties_multitype}.

We emphasise that the age differential equation (\eqref{eq:ageDE}+\eqref{eq:phiexpression}) describing the dynamics of $\pi_t$
 is autonomous and time-homogeneous, which is why we have chosen to formulate it in terms of ages rather than `birth-times' of vertices.

In the companion paper \cite{Crane2} it is shown that the system (\eqref{eq:ageDE}+\eqref{eq:phiexpression}) has a unique solution over any time interval $[0,T]$, for each age-critical Borel probability measure  $\pi_0$ satisfying $\int x\,d\pi_0(x) < \infty$. It is also shown there that $\pi_t$ is locally Lipschitz in $\pi_0$ with respect to the Wasserstein 1-metric $W_1$, uniformly over $[0,T]$. $W_1$ is the $L^1$ distance between cumulative distribution functions, and it metrizes simultaneous weak convergence and convergence of first moment. %The conditions on $\pi_0^n$ in Theorem~\ref{thm_convergence} do not imply that $\pi_0^n$ converges in probability to $\pi_0$ with respect to the topology induced by $W_1$.
It is also shown in \cite{Crane2} that for any solution of (\eqref{eq:ageDE}+\eqref{eq:phiexpression}), $\theta_t$ is a continuous function of
$t \in [t_{\mathrm{gel}},\infty)$, with respect to $L^{\infty}([0,\infty))$, therefore in~\eqref{eq:ageDE} we see that $\int f(s) \,d\pi_t(s)$ is continuously differentiable for $t> t_{\mathrm{gel}}$.

\begin{corollary} \label{cor_jensen} Under the assumptions of Theorem \ref{thm_age_pde},
the total burning rate is bounded by one: for all $t \ge t_{\mathrm{gel}}$, $\varphi(t) \le 1$.
\end{corollary}
\begin{proof}
 Apply Jensen's inequality to the integral in~\eqref{eq:phiexpression}, using the convexity of $x \mapsto x^3$ on $[0,\infty)$, noting that $\int \theta_t(s) d\pi_t(s) = 1$ and $\theta_t(s) \ge 0$.
\end{proof}
Solving the critical forest fire equations starting from the $m$-disperse state, where $v_k(0) = \mathds{1}[k=m]$ (for some $m \ge 2$),  at the gelation time $t_{gel} = 1/m$ the value of $\varphi$  is $m$. So Corollary \ref{cor_jensen} genuinely relies on Assumption \ref{assumption_v_k_0_comes_from_BP}.

% Old text of this remark:
%\begin{remark}\label{remark_phi_can_be_bigger_than_one} In Corollary \ref{cor_jensen} we used that $\underline{v}(0)$ satisfies Assumption \ref{assumption_v_k_0_comes_from_BP}. If we do not make this assumption, then $\varphi(t)$ can be arbitrarily large, as we now explain. Given some $\underline{v}(0)$  that satisfies $\sum_{k=1}^{\infty} k^3 v_k(0)<+\infty$ and a positive integer $m$, define $\underline{\tilde{v}}(0)$ by $\tilde{v}_{km}(0)=v_k(0)$ (and $\tilde{v}_{\ell}(0)=0$ if $m$ does not divide $\ell$). Then  Proposition \ref{prop_Smol_uniqueness_simplified} is not applicable, but \cite[Theorem 1]{RathToth} guarantees that the critical forest fire equations started from $\underline{\tilde{v}}(0)$   are well-posed and the conclusions of Proposition \ref{prop_Smol_uniqueness_simplified} do  hold for the solution $(\underline{\tilde{v}}(t))$. One can then check that    $\tilde{v}_{km}(t/m)=v_k(t)$ holds for any $t \in [0,t_{\mathrm{gel}}]$ (and $\tilde{v}_{\ell}(t/m)=0$ if $m$ does not divide $\ell$), and then $\tilde{t}_{\mathrm{gel}} = t_{\mathrm{gel}}/m$ by \eqref{def_gelation_time} and one can check $\tilde{\varphi}(\tilde{t}_{\mathrm{gel}})=m\varphi(t_{\mathrm{gel}})$ using \eqref{v_k_polynomial_decay}. \end{remark}

\begin{remark}\label{remark_about_v_k_t_pi_t}

$ $
\begin{enumerate}[(i)]
\item \label{remark_assumtion_holds_at_time_t} If $\underline{v}(0)$ satisfies Assumption \ref{assumption_v_k_0_comes_from_BP} and $(\underline{v}(t))_{t \geq 0}$ denotes the
unique solution of (\eqref{smol_ff_eq}+\eqref{smol_bc}) guaranteed by Proposition \ref{prop_Smol_uniqueness_simplified}, then $\widetilde{\underline{v}}(0):= \underline{v}(t)$ also satisfies Assumption \ref{assumption_v_k_0_comes_from_BP} for any $t \geq 0$
 by Theorems \ref{thm_convergence}, \ref{thm_local_limit_mfff}.
 \item \label{remark_statiorary_pi} One can  check that the unique fixed point of the age differential equation described in Theorem~\ref{thm_age_pde} is the age-critical measure $\pi$
which has density $\mathrm{d}\pi(x)= \tfrac{1}{2} \mathrm{sech}^2(\tfrac{x}{2})\,\mathrm{d}x$. \\ The  principal eigenfunction of the operator $\mathcal{L}_\pi$ is $\theta(x) = 2\tanh(\tfrac{x}{2})$, $\varphi = 1/2$ and  $ \mathbb{P}(|T^{\pi}|=k)=\frac{2}{k} \binom{2k-2}{k-1}4^{-k}$ for any $k \geq 1$, c.f.\
Remark \ref{remarks_varphi}\eqref{remark_explicit_smol_stationary}.
 Many questions related to the stationary MFFF process are discussed in \cite{Crane}.
 \end{enumerate}
\end{remark}

\begin{remark}\label{remark_phi_t_giant_speed} Our proof of \eqref{eq:phiexpression} involves the asymptotics of the generating function
$\mathbb{E}\left(z^{|T^{\pi_t}|}\right)$ as $z \to 1$ (see Section \ref{sec_properties_multitype}) and thus it is analytic in nature.
 Let us therefore provide a non-rigorous probabilistic explanation of \eqref{eq:phiexpression}. Given an age-critical distribution $\pi_t$ and $h \in \mathbb{R}_+$,
let us define the distribution $\pi_{t,h}$ by $\int  f(x) \mathrm{d}\pi_{t,h}(x) = \int f(t+h) \mathrm{d} \pi_{t}(x) $. One can show that
\begin{equation}\label{eq_giant_growth}
\mathbb{P}\left( \, |T^{\pi_{t,h}}|=+\infty \, \right)=2 \textstyle{\left( \int_0^\infty \theta_t(s)^3\mathrm{d}\pi_t(s) \right)^{-1}} h +o(h), \qquad h \to 0.
\end{equation}
Note that a similar expression characterising the infinitesimal rate of emergence of the giant component in a family of IRGs near criticality appears in equation  (3.12) in Theorem 3.17 of \cite{BollobasJansonRiordan}. The fires in the MFFF burn the incipient giant clusters as they try to appear, so it is natural to expect that $\varphi(t)$,  the limiting total rate of mass of burnt vertices at time $t$, is proportional to the speed at which the giant cluster wants to grow, i.e. $2 \left( \int_0^\infty \theta_t(s)^3\mathrm{d}\pi_t(s) \right)^{-1}$.
In Section \ref{subsubsection_smol} we further comment on the reason why we have to multiply this speed by $1/2$ to obtain the value of $\varphi(t)$ given in \eqref{eq:phiexpression}.
\end{remark}

\subsection{Relation to other work}
\label{subsub_other_work}

\subsubsection{Erd\H{o}s--R\'enyi graphs and the Flory equations}

If we consider the $\mathrm{MFFF}(n,\lambda)$ model of Definition \ref{def:graphical_mfff} without lightning (i.e., $\lambda=0$) and start it from the empty graph on $n$ vertices,  we get back the classical Erd\H{o}s--R\'enyi graph process.
 It is well-known \cite{buffet_pule_1, buffet_pule_2} that if we consider the component sizes $(v^n_k(t))$ (c.f.\ Definition \ref{def_v_n_k_t}) then
  the limits $v^n_k(t) \stackrel{\mathbb{P}}{\rightarrow} v_k(t)$ exist as $n \to \infty$ and the limiting $(v_k(t))$ satisfy the Flory equations:
\begin{equation}\label{flory}
\frac{\mathrm{d}}{\mathrm{d}t} v_k(t) = \frac{k}{2}\sum_{\ell=1}^{k-1}v_\ell(t)v_{k-\ell}(t) - kv_k(t) \sum_{\ell=1}^\infty v_\ell(0),\quad k\ge 1,
\end{equation}
where $\sum_{\ell=1}^\infty v_\ell(0)=1$ follows from our normalization \eqref{eq:defnvn}. Note that \eqref{smol_ff_eq} agrees with \eqref{flory}
for any $k \geq 2$, but we obtain the r.h.s.\ of \eqref{varphi_def} by adding $\varphi(t)$ to the r.h.s.\ of \eqref{flory} when $k=1$ and $t > t_{\mathrm{gel}}$.
Also note that \eqref{smol_bc} fails for the solution of \eqref{flory} when $t > t_{\mathrm{gel}}$.

The local weak limit of the Erd\H{o}s--R\'enyi graph at time $t$ is the Galton--Watson tree with $\mathrm{Poi}(t)$ offspring distribution, see e.g.\ \cite[Theorem 2.11]{van der Hofstad 2}.
In the notation of Definition \ref{def:branchingprocess}, this tree is $T^{\delta_t}$, where $\delta_t$ is the Dirac measure concentrated on $t$.
The explicit solution of \eqref{flory} is given by
$$v_k(t)=\mathbb{P}(|T^{\delta_t}|=k)=\frac{k^{k-1}}{k!} e^{-kt}t^{k-1}, \qquad t \in [0,+\infty), $$
i.e., the Borel distribution with parameter $t$.
Note that in the supercritical regime $t>1$ we have $\sum_{k=1}^{\infty} v_k(t) = 1- \mathbb{P}(|T^{\delta_t}|=\infty)<1$, so the Borel distribution is a defective distribution when $t > 1$.

\subsubsection{Stochastic models of Smoluchowski's coagulation equations}\label{subsubsection_smol}

The family of equations \eqref{flory} is related to Smoluchowski's coagulation equations, for which the final sum in \eqref{flory} is replaced by $\sum_{\ell=1}^\infty v_\ell(t)$:
\begin{equation}\label{smol_intro}
\frac{\mathrm{d}}{\mathrm{d}t} v_k(t) = \frac{k}{2}\sum_{\ell=1}^{k-1}v_\ell(t)v_{k-\ell}(t) - kv_k(t) \sum_{\ell=1}^\infty v_\ell(t),\quad k\ge 1.
\end{equation}
 The solutions to the Flory and Smoluchowski equations coincide until the \emph{gelation time} $t_{\mathrm{gel}}:= (\sum_{k=1}^\infty k v_k(0))^{-1}$,
 moreover $\sum_{\ell=1}^\infty v_\ell(t)=\sum_{\ell=1}^\infty v_\ell(0)=1$ holds for all $t \in [0, t_{\mathrm{gel}}]$.
  Beyond $t_{\mathrm{gel}}$ the evolution depends, informally, on whether small blocks are allowed to interact with the so-called \emph{gel}, which has mass $\sum_{\ell=1}^\infty v_\ell(0)-\sum_{\ell=1}^\infty v_\ell(t)$ at time $t$. In both cases, however, the total mass $\sum_{k=1}^\infty v_k(t)$ of small components is strictly decreasing for $t\in[t_{\mathrm{gel}},\infty)$.

 The mean field frozen percolation model  \cite{Rath} is defined exactly like the MFFF (c.f.\ Definition \ref{def:graphical_mfff}), only differing in that when a connected component is struck by lightning, the \emph{vertices} of the component get deleted forever. Vertices that have not yet been deleted are called \emph{alive vertices}.
 The number of alive vertices in the frozen percolation model decreases over time. The analogue of Proposition \ref{prop_RathToth_conv} holds in the
 frozen percolation model by \cite[Theorem 1.2]{Rath} (see also \cite[Chapter 4]{Yeothesis} for a shorter proof with weaker assumptions on the initial state), where the limiting component size densities $\left( v_k(t) \right)_{k=1}^\infty$ solve
 Smoluchowski's coagulation equations \eqref{smol_intro}.

The solution of \eqref{smol_intro} is known to be unique and explicit for general initial conditions, see \cite{NormandZambotti} and also \cite[\S 2]{Rath}. The frozen percolation model shares the feature of SOC with the MFFF, i.e. \eqref{v_k_polynomial_decay} holds  for the solution of \eqref{smol_intro} with
$\varphi(t)=-\frac{\mathrm{d}}{\mathrm{d}t}\sum_{\ell=1}^\infty v_\ell(t) $; see \cite[Theorem 1.5]{Rath}. One way to approximate SOC is to allow tiny giant clusters
to grow and then destroy them successively (see \cite[Theorems 1.4, 1.9]{Rath}), and the multiplying factor $1/2$ mentioned at the end of Remark \ref{remark_phi_t_giant_speed} is explained in \cite[Section 7]{Rath}, for an equivalent approximation of SOC in thc context of mean field frozen percolation. Roughly speaking, the model spends half of its time in a slightly subcritical state (recovering from the destructions), while the other half is spent in a slightly supercritical state (growing tiny giants).

\medskip

Another closely related random graph model where \eqref{smol_intro} describes the limiting component size densities was proposed by Aldous in \cite[\S 5.5]{Aldous_frozen} and
  studied by Merle and Normand
 in \cite{merle_normand}. Two connected components of size $k$ and $l$
merge at rate $\frac{kl}{n}$ (just as in Definition \ref{def:graphical_mfff}) and  connected components disappear if their size exceeds
a threshold  $\omega(n)$ satisfying $1 \ll \omega(n) \ll n$.
 \cite[Theorem 1.1]{merle_normand}   states that $\underline{v}^n(t)$ converge to the solution $\underline{v}(t)$ of  \eqref{smol_intro}.
 Note that this result is a special case of the main result of  \cite{fournier_laurencot_smol},
where
 discrete models of Smoluchowski's coagulation equations with more general coagulation kernels are studied.

Similarly to our Theorem \ref{thm_local_limit_mfff}\eqref{local_weak_limit_in_ffm}, it is shown in \cite[Theorem 1.3]{merle_normand} that the local weak limit of the ``$\omega(n)$-threshold deletion'' graph model, started from the empty graph on $n$ vertices, is a Galton--Watson branching process tree with $\mathrm{Poisson}(1 \wedge t)$ offspring distribution at time $t$.
 Also note that in this case the fraction of alive vertices at time $t$ converges in probability to $1 \wedge 1/t$ as $n \to \infty$ by \cite[Theorem 1.2]{merle_normand}.

 \subsubsection{Frozen percolation on inhomogeneous random graphs}\label{subsubsection_frozen_IRG_Dominic}

The mean field frozen percolation model of \cite{Rath} started from an IRG is studied by the third author in \cite{Yeo}.
 Initially, each of the $n$ vertices has one of $k$ types, and the fraction of vertices with type $i$ is $\pi^n_i(0)$. At time zero, conditional on these types, a pair of vertices with types $i$ and $j$ is connected with probability
  $1-\exp\left( -\kappa_{i,j}/n \right)$ independently of other pairs, where $(\kappa_{i,j})_{i,j=1}^k$ (the kernel of the IRG)
  is some  symmetric matrix with strictly positive entries.
Analogously to our Theorem \ref{thm_mfff_inhomog},   if we   condition on the set of alive vertices at any time $t$ and their types, the graph is also an IRG with kernel
$(\kappa_{i,j}+t)_{i,j=1}^k$, see \cite[Proposition 7]{Yeo}.

 By \cite[Theorem 3]{Yeo} the vector $(\pi^n_i(t))_{i=1}^k$ of proportions of vertices of each type that remain alive at time $t$ has a limit $(\pi_i(t))_{i=1}^k$ as $n \to \infty$. Moreover, if $t \geq t_{\mathrm{gel}}$ then the vector $(\pi_i(t))_{i=1}^k$ obeys the differential equation
\begin{equation}\label{frozen_irg_ode}
 \frac{\mathrm{d}}{\mathrm{d}t} \pi_i(t)=- \varphi(t) \theta_i(t)\pi_i(t),
 \end{equation}
  where $\varphi(t)$ is the total destruction rate at time $t$ and $(\theta_i(t))_{i=1}^k$ is the right eigenvector corresponding to the principal eigenvalue $\lambda=1$
of the critical Perron--Frobenius matrix  $\left(  (\kappa_{i,j}+t)\pi_j(t) \right)_{i,j=1}^k$, normalized so that $\sum_{i=1}^k \theta_i(t) \pi_i(t) =1$.
(Note that the results in \cite{Yeo} are formulated using the left eigenvector $(\mu_i(t))_{i=1}^k$, where $\mu_i(t)=\theta_i(t)\pi_i(t)$.)
This result is  similar to our Theorem \ref{thm_age_pde}, but
the method of proof is quite different, as we now explain. A key ingredient of the approach of \cite{Yeo} is that the empirical distribution of the vertex types
in a very large component of the IRG seen at time $t$ is close to $(\theta_i(t)\pi_i(t))_{i=1}^k$  with high probability (see \cite[Section 4]{Yeo} for the precise statement and proof). Instead of proving a similar result for age-driven IRGs, for which the continuous type space presents considerable technical challenges, we derive Theorem \ref{thm_age_pde} by directly studying the fine properties of the local limit objects (such as the cluster growth process of \cite{CraneFreemanToth} and age-driven multitype branching process trees) that arise from the MFFF as $n \to \infty$ (see the introduction of
\S \ref{section_cluster_growth_age} for further details about our methods).

\subsubsection{A model with limited aggregations}

In \cite{merle_normand_2} a dynamical variant of the configuration model is discussed where
 edges matching remaining stubs are added one by one in a uniform fashion. Connected
 clusters get deleted as soon as their size exceeds a threshold $\alpha(n)$, where $n$ denotes the initial number of vertices and $1 \ll \alpha(n) \ll n$. Similarly to our Theorem \ref{thm_local_limit_mfff}, the local weak limit of this random graph model at any time $t$ after gelation turns out to be a critical Galton--Watson tree.

\subsection{Open questions}\label{subsection_open_q}

\begin{enumerate}
\item For each $n$, the evolution of the cluster size densities $\underline{v}^n(t)$ of the MFFF can be viewed as a coagulation-fragmentation process
(see \cite[Section 2.2]{RathToth}), which converges in distribution as $t \to \infty$ to $\underline{v}^n(\infty)$, which is distributed according to the unique stationary law of the MFFF. The critical forest fire equations (\eqref{smol_ff_eq}+\eqref{smol_bc}) themselves have a unique fixed point $\underline v(\infty)$, see Remark \ref{remarks_varphi}\eqref{remark_explicit_smol_stationary}.
 However, it has not yet been proved
 that the limit as $n \to \infty$ of $\underline{v}^n(\infty)$ is deterministic and equal to $\underline{v}(\infty)$, see \cite[Conjecture 1]{Crane}.

\item A related question is whether there exist non-constant periodic or chaotic solutions to the critical forest fire equations (\eqref{smol_ff_eq}+\eqref{smol_bc}).
% or to the age differential equation  \eqref{eq:ageDE}+\eqref{eq:phiexpression}.
  Indeed, it is not yet known whether we have $\underline{v}(t)\rightarrow \underline v(\infty)$ as $t\rightarrow\infty$, even in the \emph{monodisperse} case where $\underline v(0)=(1,0,0,\ldots)$. The analogous question for solutions to the age differential equations (\eqref{eq:ageDE}+\eqref{eq:phiexpression}), for which the expected limit $\pi$ is discussed in Remark \ref{remark_about_v_k_t_pi_t}\eqref{remark_statiorary_pi}, is also unknown. In both settings, the conjecture that $\varphi(t)\rightarrow\frac12$ as $t\rightarrow\infty$ is of central importance.

%As noted in \cite{CraneFreemanToth}, whenever one can show that $\varphi(t)\rightarrow \frac12$ as $ t \to \infty$, then the convergence $\underline v(t)\rightarrow \underline v(\infty)$ follows. Although we have not shown that the distribution of $\left|T^\pi\right|$ determines $\pi$ in the age-subcritical and age-critical cases, one can show that $\varphi(t) \to \frac{1}{2}$ also implies that the characteristic curves converge in shape and so, using Theorem~\ref{thm_theta_in_terms_of_psi}, that the age distribution converges to the stationary age distribution $\pi_{\mathrm{stat}}$, which we will state explicitly in Section \ref{subsection_age_driven_examples}.

\item Is it possible to write down any explicit non-constant age-critical solutions of the age differential equations (\eqref{eq:ageDE}+\eqref{eq:phiexpression})?
  No explicit solutions of the critical forest fire equations (\eqref{smol_ff_eq}+\eqref{smol_bc}) are currently known beyond $t_{\mathrm{gel}}$, apart from the stationary solution $\underline{v}(\infty)$.

\item An age-critical or age-subcritical distribution $\pi$ determines a cluster size distribution $v_k=\mathbb{P}(|T^{\pi}|=k),\, k=1,2,\dots$. Not every such probability measure arises from an age distribution $\pi$. For example, the support of $(v_k)_{k=1}^\infty$ must be all of $\mathbb{N}$ if $\pi \neq \delta_0$. In the other direction, does $(v_k)_{k=1}^\infty$ uniquely determine $\pi$?
\end{enumerate}

\subsection{Overview of the contents of the rest of this paper}\label{subsection_overview}
In \S \ref{section_inhom} we connect the MFFFA and the notion of age-driven IRG to the literature of IRGs and their local weak limits.
In \S \ref{section_analysis-1} we discuss fundamental properties of the branching operator and age-driven MBPs.
In \S \ref{sec_properties_multitype} we state and prove some results about the fine asymptotics of the probability generating functions of the
size of age-driven MBP trees. In \S \ref{section_cluster_growth_age} we prove the main results of this paper.
We strongly encourage the reader interested in our methods to read the introduction of \S \ref{section_cluster_growth_age} now: in \S \ref{subsection_cgp_w_a} we
recall and introduce some auxiliary Markov processes (e.g.\ the \emph{cluster growth process with age}) which allow us to describe the asymptotics of the time evolution of the empirical age distribution in the MFFFA. Some properties of these auxiliary Markov processes  stated in \S \ref{subsection_cgp_w_a} (e.g.\ the \emph{intertwining relation} between them) are interesting in their own right.

\section{Age-driven IRGs and their limits}\label{section_inhom}

In \S \ref{subsection_ff_preserves_IRG} we prove Theorem \ref{thm_mfff_inhomog}, that the dynamics of the $\mathrm{MFFFA}$ preserve the class of age-driven IRGs.

 In \S \ref{subsection_local_weak_convergence} we recall some elements of the theory of IRGs from \cite{BollobasJansonRiordan} and \cite{van der Hofstad 2} and prove Proposition \ref{prop:localweakconv}.

 In \S \ref{subsection_weak_lim_two_c_s_a} we extend this to understand the limiting joint distribution of the ages and component sizes of two i.i.d.\ uniformly chosen vertices in a large age-driven IRG,  for the purpose of running a second moment argument.

\subsection{Forest fire dynamics preserve age-driven IRGs}
\label{subsection_ff_preserves_IRG}

\begin{proof}[Proof of Theorem \ref{thm_mfff_inhomog}] Without loss of generality we assume that $\underline{a}^n_0$ is deterministic, by conditioning on its value.

In order to prove \eqref{eq:mfff_inhom_rg} we will prove an even stronger statement: we
will prove that \eqref{eq:mfff_inhom_rg} holds even if we further condition
on the lightning PPP $\Lambda$. In practice, this means
that we assume that  $\Lambda$ consists of the deterministic points $(i_k,t_k), \, k = 1,2,\dots$, where
$0=t_0<t_1<t_2<\dots$, that is, vertex $i_k$'s component is burned at time $t_k$.

We will prove that \eqref{eq:mfff_inhom_rg} holds for all $t \leq t_k, \, k=0,1,2,\dots$  by induction
on $k$. By construction of $(\mathcal{G}^n_0,\underline{a}^n_0)$ (as an age-driven IRG), \eqref{eq:mfff_inhom_rg} holds for $k=0$.

Now let us assume that \eqref{eq:mfff_inhom_rg} holds for all $t \leq t_{k-1}$.
In particular, the conditional distribution of $(\mathcal{G}^n_{t_{k-1}},\underline{a}^n_{t_{k-1}})$ given $\underline{a}^n_s, \, 0 \leq s \leq t_{k-1}$  is an age-driven IRG. Since no point of $\Lambda$ lies in $[n]\times (t_{k-1},t_k)$, we have
\begin{equation}\label{eq:antk-1}
a^n_t(i)=a^n_{t_{k-1}}(i)+(t-t_{k-1}), \qquad t_{k-1} \leq t< t_k, \quad i \in [n].
\end{equation}
Edge $\{i,j\}$ is present in $\mathcal{G}^n_t$ either if it is present in $\mathcal{G}^n_{t_{k-1}}$, or if it appears during time interval $(t_{k-1},t]$. Using the induction hypothesis and the fact that edges appear independently at rate $1/n$ on this interval, we conclude that conditional on $(\underline{a}^n_s,\,0\le s\le t)$ and $\Lambda$, each edge $\{i,j\}$ is present in $\mathcal{G}^n_t$ independently with probability
\begin{multline}
 1-\exp\left(-\tfrac{1}{n}(a^n_{t_{k-1}}(i)\wedge a^n_{t_{k-1}}(j))\right) \exp\left(-\tfrac{1}{n}(t-t_{k-1})\right)\stackrel{\eqref{eq:antk-1}}
 =\\
 1-\exp\left(-\tfrac{1}{n}(a^n_t(i)\wedge a^n_t(j))\right).
\end{multline}

So \eqref{eq:mfff_inhom_rg} holds for any $t_{k-1} \leq t< t_k$, and if we denote by $\mathcal{G}^n_{t_{k}-}$ and $ \underline{a}^n_{t_{k}-} $ the graph and the age configuration that we see right before the lightning strike at time $t_k$, then
\begin{equation}\label{eq:mfff_inhom_rg_at_tk_minus}
 \text{conditional on $(\underline a^n_s,\,s\in[0,t_k))$, we have $\mathcal{G}^n_{t_{k}-}\stackrel{d}=\Gage(n,\underline{a}^n_{t_{k}-})$.}
\end{equation}

At time $t_{k}$, a lightning strikes vertex $i_{k}$ and the vertices of the connected component $\mathcal{C}$ of vertex $i_k$ in the graph
$\mathcal{G}^n_{t_{k}-}$
burn, thus we have
\begin{equation}\label{ages_after_lightning}
a^n_{t_k}(i)=\begin{cases} 0 & \text{ if } i \in \mathcal{C},\\
a^n_{t_{k}-}(i) = a^n_{t_{k-1}}(i)+(t_k-t_{k-1}) & \text{ if } i \in [n] \setminus \mathcal{C}.
\end{cases}
\end{equation}
Note that $\mathcal{C}$ is determined by $\underline{a}^n_s, \, 0 \leq s \leq t_k$
because $a^n_{t_k}(i)=0$ if and only if $i \in \mathcal{C}$. Also note that there are no edges between
$\mathcal{C}$ and $[n] \setminus \mathcal{C}$ in $\mathcal{G}^n_{t_{k}-}$, since $\mathcal{C}$ is a connected component of
$\mathcal{G}^n_{t_{k}-}$. Also note that it follows from \eqref{eq:mfff_inhom_rg_at_tk_minus} and \eqref{ages_after_lightning} that if we condition on $(\underline a^n_s,\,s\in[0,t_k])$ then the subgraph of $\mathcal{G}^n_{t_{k}-}$ spanned
by the vertices $[n] \setminus \mathcal{C}$ is an age-driven inhomogeneous random graph with ages $(a^n_{t_{k}}(i), i \in [n] \setminus \mathcal{C})$, i.e., independently for different edges, the edge $\{i,j\}\in\binom{[n] \setminus \mathcal{C} }{2}$ is present with probability
$1-\exp\left( -\frac{1}{n}( a^n_{t_{k}}(i)\wedge a^n_{t_{k}}(j)) \right)$. Also note that the subgraphs of $\mathcal{G}^n_{t_{k}-}$ and $\mathcal{G}^n_{t_{k}}$ spanned
by the vertices $[n] \setminus \mathcal{C}$ are the same, but if either
 $i$ or $j$ belongs to $\mathcal{C}$, the edge $\{i,j\}$ is not present in $\mathcal{G}^n_{t_k}$. Putting these observations together with
Definition \ref{def:agedrivenIRG} and \eqref{ages_after_lightning}, we obtain that  \eqref{eq:mfff_inhom_rg} holds for $t=t_k$.
This completes the proof of Theorem \ref{thm_mfff_inhomog}.
\end{proof}

\subsection{Local weak convergence of age-driven IRGs}
\label{subsection_local_weak_convergence}

The goal of \S \ref{subsection_local_weak_convergence} is to prove Proposition~\ref{prop:localweakconv}.

\smallskip

In order to apply the theory of IRGs established in \cite{BollobasJansonRiordan}, we recall the notion of graphical and irreducible kernels from
\cite[Definitions 2.7, 2.10]{BollobasJansonRiordan} (or, alternatively, from \cite[Definition 3.3]{van der Hofstad 2}).
In our case the \emph{ground space} $(\mathcal{S}, \mu)$ is $([0,\infty), \pi)$, where $\pi$ is a Borel probability measure on $[0,\infty)$, and the kernel is $\kappa(x,y)=x \wedge y$.
Our \emph{vertex space} (c.f.\ \cite[\S 2]{BollobasJansonRiordan}) is $\left([0,\infty), \pi, (\underline{a}^n)_{n \ge 1}\right)$, where $\underline{a}^n = (a^n(1),\ldots,a^n(n)) \in [0,\infty)^n$ is a sequence of (possibly random) ages in $G^n = G^{\textup{age}}(n,\underline{a}^n)$, which determines the (possibly random) empirical measure $\pi^n$. Our definition of $\pi^n \stackrel{\mathbb{P}}{\Rightarrow} \pi$ (c.f.\ Definition \ref{def_conv_of_measures_simple})
is one of the equivalent characterizations given in \cite[Lemma A.2]{BollobasJansonRiordan}.
Our particular definition of the conditional edge probabilities given the types (i.e., ages) of the vertices in an age-driven IRG (c.f.\ Definition \ref{def:agedrivenIRG})
is in line with  \cite[(2.6)]{BollobasJansonRiordan}.

\smallskip

 Consider $G^n \stackrel{(d)}{=} G^{\textup{age}}(n, \underline{a}^n)$
 and for any $B \ge 0$ define $a_B^n(i):=a^n(i) \wedge B$ for $i=1, \dots, n$. We may use these truncated ages to define an age-driven IRG $G_{(B)}^n \stackrel{(d)}{=} G^{\textup{age}}(n, \underline{a}_B^n)$, coupled such that $G_{(B)}^n$ is a subgraph of $G^n$.
 We also define the bounded continuous kernel $\kappa_B(x,y) = x \wedge y \wedge B$ and note that $G_{(B)}^n$ can be viewed as an IRG with types $\underline{a}^n$ and kernel $\kappa_B$. Let us also introduce the multitype Galton--Watson tree
$T_{(B)}^\pi$, which is defined in the same way as $T^\pi$ (c.f.\ Definition \ref{def:branchingprocess}) but using the kernel $\kappa_B$, so that the offspring of a vertex of type $s$ are given by a PPP with intensity $(s \wedge u \wedge B)\,d\pi(u)$.

We first show that the ``truncated" analogue of Proposition \ref{prop:localweakconv} holds.
\begin{lemma}[Truncated IRG locally converges to MBP]\label{lemma_truncated_kernel_graphical}
$G_{(B)}^n$ converges in probability in the local weak sense as $n \to \infty$ to  $T_{(B)}^\pi$.
\end{lemma}
  \begin{proof}
   First we check that $\kappa_B$ is a \emph{graphical kernel} in our vertex space: see conditions (i), (ii) and (iii) of \cite[Definition 2.7]{BollobasJansonRiordan}.
\begin{enumerate}[(i)]
 \item $\kappa_B$ is continuous.
 \item $\kappa_B$ is bounded, which implies $\int \int \kappa_B(x,y)\,d\pi(x)d\pi(y)<\infty$.
 \item The expected number of edges in $G^{\textup{age}}\left(n,\underline{a}^n\right)$ divided by $n$ tends to \\ $\frac{1}{2}\int\!\int \kappa_B(x,y) \,d\pi(x)\,d\pi(y)$ as $n \to \infty$, and this holds for bounded continuous kernels by \cite[Lemma 8.1 and Remark 8.4]{BollobasJansonRiordan}.
\end{enumerate}
Thus our kernel $\kappa_B$ is graphical. $\kappa_B$ is also irreducible (c.f.\ \cite[Definition 2.10]{BollobasJansonRiordan}), therefore the proof of Lemma \ref{lemma_truncated_kernel_graphical} follows from \cite[Theorem 3.11]{van der Hofstad 2}.
   \end{proof}
 \begin{remark}
To ensure that $\kappa$ itself is a graphical kernel we would have to assume some extra condition on $\pi^n$ in addition to its convergence in probability to $\pi$, e.g.\ $\mathbb{E}(\int x \,d\pi^n(x)) \to \int x\,d\pi(x)$  would suffice. Our truncation argument  allows us to avoid making such an extra assumption.
\end{remark}

 Let us denote by $\mathcal{S}^{n,k}_{(B)}$ the set of vertices $v$ of $G^n$ such that there is a vertex of age $B$ within distance $k$ of $v$ in $G_{(B)}^n$.
 Note that for any $v \in [n] \setminus \mathcal{S}^{n,k}_{(B)}$ the ball of radius $k$ about $v$ in $G^n$ agrees with that in $G_{(B)}^n$.
\begin{lemma}[$G_{(B)}^n$ and $G^n$ locally look similar]\label{L: tricky lemma} Let $\epsilon > 0$ and $k \in \mathbb{N}$. Then, there exists a constant $B=B(\epsilon,k)$ such that
for any  sufficiently large $n$ we have
\begin{equation}\label{S_n_r_B_bound}
 \mathbb{P}\left( |\mathcal{S}^{n,k}_{(B)}| \geq n \epsilon  \right)\leq  \epsilon \, .
 \end{equation}
\end{lemma}
\begin{proof}
 Consider the set $\textup{Path}(\ell,B)$ that consists of directed simple paths in $G_{(B)}^n$ of length $\ell$, beginning in the set $S = \{v \in [n] \,:\, a_B^n(v) = B\}$.
 Conditional on the ages, the probability that any particular directed path $(v_0, v_1, \dots, v_\ell) \in \textup{Path}(\ell,B)$  is present in $G_{(B)}^n$ is bounded as follows:
\begin{multline*}
 \mathbb{P}\left[ \{v_{i-1},v_i\} \in E(G_{(B)}^n): 1 \le i \le \ell \,\Big|\, \underline{a}_B^n \right] = \\
  \prod_{i=1}^\ell \left( 1-e^{-(a_B^n(v_{i-1}) \wedge a_B^n(v_i))/n}\right) \leq
   \prod_{i=1}^\ell \frac{a_B^n(v_{i-1}) \wedge a_B^n(v_i)}{n}   \le  \prod_{i=1}^\ell \frac{a_B^n(v_i)}{n}\,.
 \end{multline*}
 Summing over all choices of $v_0 \in S$ and $(v_1, \dots, v_\ell) \in [n]^{\ell}$, we obtain
$$\mathbb{E}\left[|\textup{Path}(\ell,B)|\,\Big|\, \underline{a}_B^n\right] \le |S|  \left( \sum_{v \in [n]} \frac{a_B^n(v)}{n}\right)^\ell
=|S|  \left( \int (x \wedge B)\, d\pi^n(x)  \right)^\ell
\,.$$
Note that $\mathcal{S}^{n,k}_{(B)}$ is the set of vertices whose $k$-ball in $G_{(B)}^n$ meets $S$,
thus
 \begin{equation}\label{condbound_Y_n_Z_n}
  \mathbb{E}\left[  |\mathcal{S}^{n,k}_{(B)}|/n   \,|\, \underline{a}_B^n\right]\leq
\frac{|S|}{n} \sum_{\ell=0}^k \left( \int (x \wedge B)\, d\pi^n(x)  \right)^\ell=:Z_n \, .
\end{equation}
If $\pi(\{B\})=0$ then $Z_n$ converges to
$ \pi([B,\infty)) \sum_{\ell=0}^k \left(\int (x\wedge B) \,d\pi(x)\right)^\ell$ in probability, and the latter quantity is at most $\epsilon^2/4$ if we choose $B$ big enough, since
$\int x \,d\pi(x)<\infty$. Then for large enough $n$ we have
$\mathbb{P}(Z_n > \frac{\epsilon^2}{2})\leq \frac{\epsilon}{2}$, moreover $\mathbb{P}( |\mathcal{S}^{n,k}_{(B)}|/n \geq \epsilon \, | \, Z_n \leq \frac{\epsilon^2}{2} )\leq \frac{\epsilon}{2}$ holds by \eqref{condbound_Y_n_Z_n}, thus
  \eqref{S_n_r_B_bound} follows.
\end{proof}

Next we observe that $T_{(B)}^\pi$ converges to $T^\pi$ in the local weak sense as $B \to \infty$: this follows from the fact that $\kappa_B \nearrow \kappa$,
therefore one can couple $T_{(B)}^\pi$ simultaneously for all $B\geq 0$ to $T^\pi$ so that $T_{(B)}^\pi$  converges monotonically (with respect to the partial ordering of graph inclusion) to $T^\pi$ almost surely (for the details of this coupling, see the proof of \cite[Theorem 3.10]{van der Hofstad 2}).

\smallskip

Now we can conclude the proof of Proposition~\ref{prop:localweakconv}. Let us fix some finite rooted graph $(H,v)$ and $k \in \mathbb{N}$. Recalling  Definition \ref{def:localweakconv}, let $\mathcal{B}^n_k(w)$ be the $k$-neighborhood in $G^n$ of a vertex $w\in[n]$, and similarly let $\mathcal{B}^n_{k,(B)}(w)$ be the $k$-neighborhood in $G^n_{(B)}$ of $w$.
Let $X_n:=\frac{1}{n}\sum_{w\in[n]} \mathds{1}[ \mathcal{B}^n_k(w)\simeq (H,v)]$ and $X_{n,(B)}:=\frac{1}{n}\sum_{w\in[n]} \mathds{1}[ \mathcal{B}^n_{k,(B)}(w)\simeq (H,v) ]$. Let $q:=\Prob{\mathcal{B}_k(\rho) \simeq (H,v)}$ and $q_{(B)}:=\Prob{\mathcal{B}_{k,(B)}(\rho) \simeq (H,v)}$, where $\mathcal{B}_k(\rho)$ and $\mathcal{B}_{k,(B)}(\rho)$ denote the $k$-neighborhoods of the root in $T^\pi$ and $T_{(B)}^\pi$, respectively.

Given some $\epsilon>0$,  we want to show that for all large enough $n$ we have $\mathbb{P}(|X_n -  q| \geq \epsilon) \leq  \epsilon$.
Recall that for any $v \in [n] \setminus \mathcal{S}^{n,k}_{(B)}$ we have $\mathcal{B}^n_k(v) = \mathcal{B}^n_{k,(B)}(v)$,
thus by Lemma \ref{L: tricky lemma} we can choose $B^*=B^*(\epsilon,k)$ such that if $B \geq B^*$ then we have $\mathbb{P}(|X_n-X_{n,(B)}| \geq \frac{\epsilon}{3} ) \leq \frac{\epsilon}{2}$ for large enough values of $n$.
There exists $B^{**}$ such that $|q-q_{(B)}|\leq \frac{\epsilon}{3}$ for any $B\geq B^{**}$, since $T_{(B)}^\pi$ converges to $T^\pi$ in the local weak sense as $B \to \infty$. Let $B=B^* \vee B^{**}$.
  By Lemma \ref{lemma_truncated_kernel_graphical} we have $\mathbb{P}(|X_{n,(B)} -  q_{(B)}| \geq \frac{\epsilon}{3}) \leq \frac{\epsilon}{2}$ for large enough $n$.
  We obtain the desired $\mathbb{P}(|X_n -  q| \geq \epsilon) \leq  \epsilon$ for large enough $n$ by the triangle inequality and the union bound.
 The proof of Proposition~\ref{prop:localweakconv} is complete.

\subsection{Joint weak limit of component sizes and ages}
\label{subsection_weak_lim_two_c_s_a}
In order to prove that $\pi^n_t$ is concentrated around $\pi_t$ in Theorem \ref{thm_convergence} using a second moment argument, we will show in Proposition \ref{2rhocvg} that the age processes $a^n_t(\rho^n_1)$ and $a^n_t(\rho^n_2)$ of two i.i.d.\ uniformly distributed vertices $\rho^n_1,\rho^n_2$ in an MFFFA evolve asymptotically independently as $n\rightarrow\infty$. The following result implies that the ages of $\rho^n_1$ and $\rho^n_2$ in the initial graph $\mathcal{G}^n_0$ are asymptotically independent.

\begin{definition}[Components and cardinalities]\label{def_component_and_component_size}
Given an age-driven IRG $\Gage(n,\underline{a}^n)$ (c.f.\ Definition \ref{def:agedrivenIRG}) and a vertex $i \in [n]$, let $\mathcal{C}^n(i)$ denote the connected component of vertex $i$ in  the graph
$\Gage(n,\underline{a}^n)$, moreover let $C^n(i)$ denote the number of vertices in $\mathcal{C}^n(i)$, i.e., $C^n(i)=|\mathcal{C}^n(i)|$.
\end{definition}

\begin{proposition}[Weak limits for two i.i.d.\ vertices]\label{2agecvgprop}
Let $\Gage(n,\underline{a}^n)$ be a sequence of age-driven IRGs, for which the empirical age distributions satisfy $\pi^n\stackrel{\mathbb{P}}{\Rightarrow} \pi$, where $\int x\,d\pi_0(x) < \infty$ and  $\pi$ is an age-subcritical or age-critical probability measure.
 Let $\rho^n_1,\rho^n_2$ be i.i.d.\ with uniform distribution on $[n]$. Then as $n\rightarrow\infty$,
\begin{equation}\label{eq:2initialagescvg}
\left( \left(a^n(\rho^n_i),\,C^n(\rho^n_i)\right), i\in \{1,2\} \right)
\;\Rightarrow\;
\left( \left(a(\rho^{(i)}),|T^{(i)}| \right), i\in \{1,2\} \right)\, ,
\end{equation}
where $\left(a(\rho^{(i)}),|T^{(i)}| \right),\, i=1,2$ are i.i.d.\ copies
of $(a(\rho),|T^\pi|)$, where $\rho$ is the root of $T^\pi$ and $a(\rho)$ is the age of the root.
\end{proposition}
Note that when $\pi$ is age-supercritical then $\mathbb{P}(|T^\pi|=\infty )>0$ (c.f.\ Proposition \ref{prop_multitype_trichotomy}), and therefore $(C^n(\rho^n_i))_{n \in \mathbb{N}}$ is not tight.

Let us now explain why we omit the proof of Proposition \ref{2agecvgprop}.
The statement of Proposition \ref{prop:localweakconv} concerns only distributions of rooted graphs, not graphs augmented with ages. However, the proof of
\cite[Theorem 3.11]{van der Hofstad 2} (which we used in the proof of Lemma \ref{lemma_truncated_kernel_graphical}) at its core uses a coupling argument in which an exploration of $G^n$ from a uniformly chosen vertex is coupled to an exploration of $T^\pi$ from its root. In this coupling the ages of corresponding vertices are in fact close with high probability. In the discussion in \cite[\S 3.5.1]{van der Hofstad 2}, immediately following the statement of the relevant Theorem 3.11, van der Hofstad remarks that the convergence in probability in the local weak sense can be extended to apply to the topology where two local subgraphs are close only if they are isomorphic as rooted graphs by an isomorphism under which the corresponding vertices' types are uniformly close.  This is similar to Benjamini et al \cite{BenjaminiLyonsSchramm}, who study a setting with marked edges. %[CITE - Unimodular Random Trees]
We actually do not require the full power of this extension, only the fact that the initial step of the coupling argument is to couple a uniformly chosen random vertex $\rho^n$ of $G^n$ with the root $\rho$ of $T^\pi$ in such a way that with high probability the age of $\rho^n$ is close to the age of $\rho$.
This observation can be used to show that in fact for any continuity set $A$ of the measure $\pi$ and any finite rooted graph $(H,v)$ and any $k \ge 1$ we have as $n \to \infty$ that \begin{equation}\label{eq:localweakprob2}
\frac{1}{n}\sum_{w\in[n]} \mathds{1}[\, \mathcal{B}^n_k(w)\simeq (H,v),\, a^n(w) \in A\, ] \;\stackrel{\mathbb{P}}\longrightarrow \; \Prob{\mathcal{B}_k(\rho) \simeq (H,v),\, a(\rho) \in A}.
\end{equation}
In order to prove \eqref{eq:2initialagescvg}, we only need to show that for any pair of continuity sets $A_1, A_2$ of $\pi$ and finite rooted graphs $(H_1,v_1), (H_2,v_2)$, we have
\begin{multline}\label{weak_pair_conv_want}
  \mathbb{P}\left( \mathcal{B}^n_k(\rho^n_i)\simeq (H_i,v_i), a^n(\rho^n_i) \in A_i, i \in \{1,2\} \right) \to \\
  \prod_{i=1}^2  \mathbb{P}\left( \mathcal{B}_k(\rho)\simeq (H_i,v_i), a(\rho) \in A_i \right)
\end{multline}
as $n$ goes to infinity. Taking the product of two instances of \eqref{eq:localweakprob2} with $A_i$ and $(H_i,v_i)$ in place of $A$ and $(H,v)$ (with $i=1,2$, respectively), we obtain that
\begin{equation}\label{conv_prob_sum_pair}
  \frac{1}{n^2}\sum_{w_1,w_2 =1}^n  \mathds{1}[ \, \mathcal{B}^n_k(w_i)\simeq (H_i,v_i),\, a^n(w_i) \in A_i,\, i\in \{1,2\}\, ]
\end{equation}
converges in probability to the r.h.s.\ of \eqref{weak_pair_conv_want}.
Taking the expectation,  we obtain  \eqref{weak_pair_conv_want}.

\section{The operator $\mathcal{L}_\pi$ and criticality of $T^\pi$}
\label{section_analysis-1}

In \S\ref{subsection_L_op_prop} we prove (a strengthening of) Lemma \ref{lemma_L_operator_basic_properties}. In \S \ref{subsec_multitype_offspr_decay} we prove Proposition \ref{prop_multitype_trichotomy}.
%In \S \ref{subsection_age_driven_examples} we discuss notable examples of age distributions $\pi$ and related properties of $\mathcal{L}_\pi$.

\subsection{Basic properties of $\mathcal{L}_\pi$}
\label{subsection_L_op_prop}

The main goal of \S \ref{subsection_L_op_prop} is to prove Lemma \ref{lemma_L_operator_basic_properties}. In fact, we will prove a more detailed result.

Let $(X,\mu)$ be a measure space, let $\kappa$ be a measurable real kernel on $X$ and define an integral operator $T$ by $(T f)(x) = \int \kappa(x,y) f(y)\,d\mu(y)$.  Then the Hilbert--Schmidt norm of $T$ is
\begin{equation}\label{hs_norm_def}
\|T\|_{\mathrm{HS}} = \left(\int\int \kappa(x,y)^2 \,d\mu(x)\,d\mu(y)\right)^{1/2}.
\end{equation}
When $\|T\|_{\mathrm{HS}}$ is finite, $T$ is a  \emph{Hilbert--Schmidt operator} on $L^2(\mu)$. It is compact (see \cite[Theorem VI.22]{reed_simon}) with operator norm bounded by
\begin{equation}\label{HS_op_ineq}
\|T\| \leq \|T\|_{\mathrm{HS}}.
\end{equation}

%The assumption of Lemma \ref{lemma_L_operator_basic_properties} is  $\int s\,d\pi(s) < + \infty$.
%In Lemma \ref{L: finite mean implies trace-class} we show that this implies $\|\mathcal{L}_\pi\|_{\mathrm{HS}}< \int s,d\pi(s)$.
%In Lemma \ref{L: operator properties} we show that the conclusions of Lemma \ref{lemma_L_operator_basic_properties} (and more) already hold if we only assume $\|\mathcal{L}_\pi\|_{\mathrm{HS}}<+\infty$.

\begin{lemma}\label{L: finite mean implies trace-class}
If $\int s\,d\pi(s) < + \infty$ then $\mathcal{L}_\pi$ is Hilbert--Schmidt, and
\begin{equation}\label{op_norm_mean_ineq}
 \|\mathcal{L}_\pi\| \leq  \|\mathcal{L}_\pi\|_{\mathrm{HS}} \le  \int x\,d\pi(x) < \infty\,.
 \end{equation}
\end{lemma}
\begin{proof}
If we let  $m:=\int x\,d\pi(x) < \infty$ then we have
\begin{equation}\label{HS_mean_ineq}
\|\mathcal{L}_\pi\|_{\mathrm{HS}}^2 \stackrel{ \eqref{hs_norm_def} }{=} \int\!\int (x \wedge y)^2 \,d\pi(x)\,d\pi(y) \le   \int\!\int x  y \,d\pi(x)\,d\pi(y) = m^2\,.
\end{equation}
Thus  \eqref{op_norm_mean_ineq}
holds by \eqref{HS_op_ineq} and \eqref{HS_mean_ineq}.
\end{proof}

\begin{remark}In fact \cite[Thm 4.6]{Aleksandrov} shows (after a change of variables) that $\mathcal{L}_\pi$ is a bounded operator on $L^2(\pi)$ if and only if $\pi([x,\infty)) = O(1/x)$ as $x \to \infty$, it is compact if and only if $\pi([x,\infty)) = o(1/x)$ as $x \to \infty$, and it belongs to the trace class if and only if $\int x\,d\pi(x) < \infty$. In the latter case one can show that the sum of the eigenvalues converges (absolutely) to $\int x \,d\pi(x)$. See also \cite[Example 17.6]{BollobasJansonRiordan}.
\end{remark}

\begin{lemma}[Properties of $\mathcal{L}_\pi$]\label{L: operator properties}
If $\|\mathcal{L}_\pi\|_{\mathrm{HS}}<\infty$ and $\pi \neq \delta_0$ then
\begin{enumerate}[(i)]
\item \label{L:self_adjoint_pos} $\mathcal{L}_\pi$ is a positive semidefinite compact self-adjoint operator on $L^2(\pi)$.
\item \label{L:Lipsch} Elements of the image of $\mathcal{L}_\pi$ are represented by Lipschitz functions and $\mathcal{L}_\pi$ maps non-negative functions to non-decreasing non-negative functions.
\item \label{L:principal}
$\mathcal{L}_\pi$ has a simple principal eigenvalue $\lambda$ satisfying $0 < \lambda = \|\mathcal{L}_\pi\|$.
\item \label{L:eigen} There is a unique eigenfunction $\theta\in L^2(\pi)$ for which $\mathcal{L}_\pi \theta=\lambda \theta$ and $\int\! \theta(x)\mathrm{d}\pi(x)=1$. We identify $\theta$ with its Lipschitz-continuous representative, which is a non-decreasing function on $[0,\infty)$ with $\theta(0) = 0$.
\end{enumerate}
\end{lemma}

\begin{proof}
We begin with the proof of \eqref{L:self_adjoint_pos}.
% First observe that $\mathcal{L}_\pi$ is a Hilbert--Schmidt integral operator acting on $L^2(\pi)$. Indeed by Cauchy-Schwarz we have
%\begin{eqnarray*} \|\mathcal{L}f\|^2_\pi & = & \int \left(\int f(x) (x \wedge y) \,d\pi(x)\right)^2 \,d\pi(y) \\ & \le & \int \left(\int f(x)^2 \,d\pi(x))\right)\left(\int (x \wedge y)^2 \,d\pi(x) \right)\,d\pi(y)
%\\ & = & \| f\|^2_\pi \int\!\int (x \wedge y)^2 \,d\pi(x)\,d\pi(y)
%\\ & \le &  \| f\|^2_\pi \int\!\int x  y \,d\pi(x)\,d\pi(y) \; = \;  \| f\|^2_\pi \left(\int x \,d\pi(x) \right)^2\,.
%\end{eqnarray*}
%Here we have shown that the operator norm $\|\mathcal{L}_\pi\|$ with respect to $L^2(\pi)$ is bounded by the Hilbert--Schmidt norm $$\|\mathcal{L}_\pi\|_{\mathrm{HS}} := \left(\int\!\int (x \wedge t)^2 \,d\pi(x)\,d\pi(t)\right)^{1/2}\,$$ which was assumed to be finite.
$\mathcal{L}_\pi$ is self-adjoint because the kernel $x \wedge y$ is real and symmetric.
% In fact $\|\mathcal{L}_\pi\|_{\mathrm{HS}}^2$ is the trace of $\mathcal{L}_\pi^*\mathcal{L}_\pi = \mathcal{L}_\pi^2$.
%This means that for any orthonormal basis $\left(e_k\right)_{k=1}^\infty$ of $L^2(\pi)$ we have
%\begin{equation}
%\label{eq: HS norm} \|\mathcal{L}_\pi\|_{\mathrm{HS}}^2  =  \sum_{i=1}^\infty \sum_{j=1}^\infty \langle \mathcal{L}_\pi e_i, e_j\rangle^2 = \int\!\int (u \wedge s)^2 \,d\pi(s)\,d\pi(u),
%\end{equation}
%where the sum is absolutely convergent.
To see that $\mathcal{L}_\pi$ is a positive semidefinite operator, note that for $f,g \in L^2(\pi)$ we have
\begin{eqnarray*}\langle \mathcal{L}_\pi f, g \rangle_\pi & = & \int\!\int f(y)g(x) (x \wedge y) \,d\pi(x) \,d\pi(y) \\
 & = & \int\!\int\!\int_0^\infty f(y) g(x)\mathds{1}[x \ge u, \,y \ge u] \,du\,d\pi(x)\,d\pi(y) \\ &
 \stackrel{(*)}{=} & \int_0^\infty\left(\int_u^\infty f(y) \,d\pi(y)\right) \left(\int_u^\infty g(x)\,d\pi(x)\right)\,du.
\end{eqnarray*}
The application of Fubini in the equation marked by $(*)$ is justified by the absolute integrability which follows from
$\langle \mathcal{L}_\pi |f|, |g|\rangle_\pi \leq \Vert \mathcal{L}_\pi \Vert \Vert f \Vert \Vert g \Vert  < \infty$, see \eqref{HS_op_ineq}. In particular we have
$ \langle \mathcal{L}_\pi f, f \rangle_\pi  = \int_0^\infty \left( \int_u^\infty f(x) \,d\pi(x) \right)^2 \,du \ge 0.
$

Next we prove \eqref{L:Lipsch}. For any $f \in L^2(\pi)$, $\mathcal{L}_\pi f$ is represented by a Lipschitz function. Indeed, since $\pi$ is a finite measure, the constant function $\mathbf{1}$ belongs to $L^2(\pi)$, and because $|(u \wedge s) - (u \wedge s')| \le |s - s'|$, we have
$$ |\mathcal{L}_\pi f(s) - \mathcal{L}_\pi f(s')|  \le  \int |f(u)|\, |(u \wedge s) - (u \wedge s')| \,d\pi(u) \le |s-s'|\langle|f|,\mathbf{1}\rangle_\pi\,.$$
Moreover, if $f$ is a non-negative function and   $0 \le s \le s'$ then
$$ 0 \leq  \mathcal{L}_\pi f(s) = \int f(u) (u \wedge s) \,d\pi(u) \le \int f(u) (u \wedge s') \,d\pi(u) = \mathcal{L}_\pi f(s')\,.$$

Next we prove \eqref{L:principal}.   Because $\mathcal{L}_\pi$ is positive semidefinite and self-adjoint, and because Hilbert--Schmidt operators are compact, the spectrum of $\mathcal{L}_\pi$ is contained in the real interval $[0,\|\mathcal{L}_\pi\|]$, with $0$ being the only possible accumulation point and every positive eigenvalue having finite multiplicity. Except in the trivial case $\mathcal{L}_\pi = 0$, which only occurs when $\pi = \delta_0$, the Perron--Frobenius theorem for Hilbert--Schmidt integral operators (see \cite[\S 6.5]{Mode} or \cite[Lemma 5.15]{BollobasJansonRiordan}) guarantees that the leading eigenvalue is $\lambda = \|\mathcal{L}_\pi\|$, and it is simple, i.e. has a one-dimensional eigenspace.
% We have already shown the other statements of \eqref{L:principal}.

Next we prove
\eqref{L:eigen}. \cite[Lemma 5.15]{BollobasJansonRiordan} shows that there is an eigenfunction $\theta \in L^2(\theta)$ with eigenvalue $\lambda$ such that $\theta$ is positive $\pi$-a.e.~ and since $\int \theta \,d\pi = \langle \theta, \mathbf{1}\rangle_\pi < \infty$ it is valid to normalize $\theta$ by multiplying it by a non-zero scalar so that $\int \theta \,d\pi = 1$. We have
\begin{equation}
\label{E: leading eigenfunction}
\theta(s) = \lambda^{-1} \mathcal{L}_\pi\theta(s) = \lambda^{-1} \int_0^\infty \theta(u) \,(u \wedge s) \,d\pi(u)\,, \qquad s \in [0,+\infty)
\end{equation}
 and we now choose the representative for $\theta$ defined by the right-hand side of~\eqref{E: leading eigenfunction} which is non-decreasing and Lipschitz in $s$ by \eqref{L:Lipsch}, and satisfies $\theta(0) = 0$ and $\theta(s) > 0$ for  $s > 0$.
\end{proof}

The proof of Lemma \ref{lemma_L_operator_basic_properties} immediately follows from Lemmas \ref{L: finite mean implies trace-class} and \ref{L: operator properties}.

\subsection{Subcriticality, criticality and supercriticality of $T^\pi$}
\label{subsec_multitype_offspr_decay}

\S\ref{subsec_multitype_offspr_decay} is devoted to the proof~of Proposition \ref{prop_multitype_trichotomy}. Note that
we will only use the properties of $\mathcal{L}_\pi$ derived in Lemma \ref{L: operator properties}, so our proof works if we only assume
 $\|\mathcal{L}_\pi\|_{\mathrm{HS}}<\infty$ and $\pi \neq \delta_0$ instead of $\int s\,d\pi(s) < + \infty$.

Denote by $\lambda=\|\mathcal{L}_\pi\|$ the principal eigenvalue of  $\mathcal{L}_\pi$, see Lemma~\ref{L: operator properties}\eqref{L:principal}.
It follows from \cite[Theorem 6.1]{BollobasJansonRiordan} that $\|\mathcal{L}_\pi\|_{\mathrm{HS}}<\infty$ implies that $\lambda>1$ if and only if $\Prob{|T^\pi|=\infty}>0$. Therefore it suffices to prove:
\begin{equation}\label{eq:critdichotomy}
\begin{cases} \lambda<1 \; &\Rightarrow\quad \E{|T^\pi|}<\infty,\\ \lambda=1\quad&\Rightarrow\quad \E{|T^\pi|}=\infty.\end{cases}
\end{equation}

 Note that the expected number of vertices at distance $k$ from the root of the tree $T^\pi$ is given by $\langle\mathbf{1},\mathcal{L}_\pi^k\mathbf{1}\rangle_\pi$, thus
\begin{equation}\label{iterated_L_expect}
\E{|T^\pi|}=\sum_{k=0}^\infty \langle\mathbf{1},\mathcal{L}_\pi^k\mathbf{1}\rangle_\pi.
\end{equation}
Using this we can prove that $\lambda<1$ implies $\E{|T^\pi|}<\infty$:
\begin{equation}
\E{|T^\pi|} \stackrel{ \eqref{iterated_L_expect} }{\leq}
\sum_{k=0}^\infty \| \mathbf{1} \| \cdot \|  \mathcal{L}_\pi^k\mathbf{1} \| \leq
\sum_{k=0}^\infty  \| \mathbf{1} \|^2 \| \mathcal{L}_\pi^k \|  \leq \sum_{k=0}^\infty \langle\mathbf{1},\mathbf{1}\rangle_\pi \lambda^k =
  \frac{1}{1-\lambda}.
\end{equation}

We now treat the critical case $\lambda=1$.
Recall from Lemma~\ref{L: operator properties}\eqref{L:eigen} that in the $\lambda=1$ case $\theta$ satisfies $\langle \mathbf{1}, \theta \rangle_\pi=1$ and $\theta= \mathcal{L}_\pi \theta$.
 Let us write $\mathbf{1}=a \theta + f$, where $a=\langle \mathbf{1}, \theta \rangle_\pi /\langle \theta, \theta \rangle_\pi > 0$, so that
 $f \in \theta^\perp$. Using that $\mathcal{L}^k_\pi$ is self-adjoint for any $k\geq 0$ (c.f.\ Lemma~\ref{L: operator properties}\eqref{L:self_adjoint_pos}), we obtain
\begin{equation}
\E{|T^\pi|}\stackrel{ \eqref{iterated_L_expect} }{=}
a^2 \sum_{k=0}^\infty \langle \theta,\mathcal{L}_\pi^k \theta \rangle_\pi + \sum_{k=0}^\infty \langle f,\mathcal{L}_\pi^k f \rangle_\pi,
\end{equation}
where the first sum on the r.h.s.\ is infinite since $\mathcal{L}^k_\pi \theta=\theta$, while the second sum is non-negative, since $\mathcal{L}_\pi^k$ is positive semidefinite, c.f.\ Lemma~\ref{L: operator properties}\eqref{L:self_adjoint_pos}.
It then follows that $\E{|T^\pi|}  = \infty$. The proof of Proposition \ref{prop_multitype_trichotomy} is complete.

\section{Multitype Galton--Watson trees and generating functions}\label{sec_properties_multitype}

 In this section we derive further properties of the multitype Galton--Watson trees $T^\pi$ and $T^\pi_s$ defined using a probability measure $\pi$ on $[0,\infty)$ that satisfies $\int u \,d\pi(u) < \infty$, see Definition \ref{def:branchingprocess}. Recall from \eqref{L_def_statement} that we define the Perron--Frobenius operator $\mathcal{L}_\pi$ acting on $L^2(\pi)$ by
\begin{equation}
\label{E: Perron-Frobenius operator definition}
\mathcal{L}_\pi f (x) = \int_0^\infty f(s) (x \wedge s) \,d\pi(s)\,.
\end{equation}
Recall from Definition \ref{L: operator properties}\eqref{L:eigen} the notion of the eigenfunction $\theta$  corresponding to the principal
eigenvalue $\lambda$ of $\mathcal{L}_\pi$.
 In the age-critical case (see Proposition \ref{prop_multitype_trichotomy}), let us define
\begin{equation}\label{phi_theta_cube_int_def}
\phi := \left(\int \theta(u)^3 \,d\pi(u)\right)^{-1}\,.
\end{equation}
%Note the a priori difference between $\phi$ defined above and the control function $\varphi(\cdot)$ associated to a solution of the critical forest fire equations
%(c.f.\ \eqref{varphi_def} and Remark \ref{remarks_varphi}\eqref{remark_varphi_control}): one goal of this paper is to connect the two, c.f.\ \eqref{eq:phiexpression}.

One of the main results of \S \ref{sec_properties_multitype} is the next lemma.

\begin{lemma}[Critical generating function asymptotics]\label{lemma_tree_s_gen_fn_asymp}
Let $\pi$ be an age-critical measure satisfying $\int u \,d\pi(u) < \infty$. Then on the intersection of the complex open unit disc $\mathbb{D}$ with a neighborhood of $1$ the generating function $\mathbb{E}\left(z^{|T^\pi_s|}\right)$ agrees with an analytic function of a branch of $\sqrt{1-z}$ that depends continuously on $s \in [0,\infty]$. In particular, as $z \nearrow 1$ through the reals, we have
\begin{equation}\label{f_s_expansion}
 \mathbb{E}\left(z^{|T^\pi_s|}\right) = 1 - \sqrt{ 2\phi }\, \theta(s) \sqrt{1-z} + O(1-z)\,,
 \end{equation}
where the implied constant in the error term is uniform as a function of $s$, and also
\begin{equation}\label{f_T_expansion}
\mathbb{E}\left(z^{|T^\pi|}\right) = 1 - \sqrt{2\phi}\,\sqrt{1-z} + O(1-z)\,.
\end{equation}
\end{lemma}

Lemma \ref{lemma_tree_s_gen_fn_asymp} is an ingredient of the proof of Theorem \ref{thm_age_pde}, see \S \ref{subsection_proof_of_age_pde}.

In \S \ref{subsection_gen_fn} we set up the notation required for the study of age-driven multitype branching processes using generating functions and write down a system of nonlinear equations satisfied by these generating functions. This leads to the study of a nonlinear Volterra equation of the second kind.

In \S\ref{subsection_volterra_well_posed} we prove that this nonlinear Volterra equation is well-posed.

In \S\ref{subsection_taylor_coefficiants} we prove Lemma \ref{lemma_tree_s_gen_fn_asymp}.

 In \S\ref{subsection_smol_uniqueness} we prove Proposition \ref{prop_Smol_uniqueness_simplified}.

The details of the proofs presented in
\S\ref{sec_properties_multitype} are independent of the rest of this paper, so the rest of \S \ref{sec_properties_multitype} can be skipped at first reading.

\subsection{The generating function of the total progeny}
\label{subsection_gen_fn}

Recall the notion of $T^\pi$ and $T^\pi_s$ from Definition \ref{def:branchingprocess}. In \S \ref{subsection_gen_fn} we assume
$\int u \,d\pi(u) < \infty$ and that $\pi$ is either
age-subcritical or age-critical (c.f.\ Proposition \ref{prop_multitype_trichotomy}).

Let us denote by $\overline{\mathbb{D}}$ the closed complex unit disc.

\begin{definition}[Generating function of $|T^{\pi}_s|$ and $|T^{\pi}|$]\label{def_gen_functions}
We write $[0,\infty]$ for the one-point compactification $[0,\infty) \cup \{\infty\}$.
  For $s \in [0,\infty]$ and $z \in \overline{\mathbb{D}}$ we define the probability generating functions
  \begin{equation}\label{f_s_z_def}
  f(s,z) := \mathbb{E}\left(z^{|T^\pi_s|}\right), \qquad
   f(z):= \mathbb{E}\left( z^{|T^{\pi}|} \right) =  \int f(s,z)\, \mathrm{d}\pi(s).
\end{equation}
 \end{definition}

\begin{lemma}[Basic properties of $f(s,z)$]\label{lemma_basic_gen_fn}

$ $

\begin{enumerate}[(i)]
\item\label{gen_fn_vi}    $f(s,z)$  satisfies $f(s,1)=1$ for all $s \in [0,\infty]$,
\item\label{gen_fn_i}
 $f(s,z) \in [0,1]$ if $z \in [0,1]$ and $f(s,z) \in \overline{\mathbb{D}}$ if $z \in \overline{\mathbb{D}}$,
\item\label{gen_fn_ii}
 $f(s,z)$  is strictly increasing in $z \in [0,1]$ for all $s \in [0,\infty]$,
 \item\label{gen_fn_iii}   for each fixed $s \in [0,\infty]$, $f(s,z)$ is analytic over $z \in \mathbb{D}$,
 \item \label{gen_fn_iv} $f(s,z)$  is
  non-increasing in $s$ for all $z \in [0,1]$,
\item\label{gen_fn_v}   $f(s,z)$ is uniformly continuous in $(s,z) \in [0,\infty] \times \overline{\mathbb{D}}$.
 \end{enumerate}
\end{lemma}
\begin{proof}
In order to prove \eqref{gen_fn_vi}, we only need to show $\mathbb{P}( |T^\pi_s| =+\infty )=0$ for all $s \in [0,\infty]$. It is enough to show
this just for $s=\infty$, since $|T^\pi_\infty|$ stochastically dominates $|T^\pi_s|$ for each $s \in [0,\infty]$.
Note that it follows from $\int u \,d\pi(u) < \infty$ that the function $s \mapsto \mathbb{P}( |T^\pi_s| =+\infty )$  is non-decreasing and continuous on $[0,\infty]$.
Also note that $\mathbb{P}( |T^\pi| =+\infty )=0$ follows from Proposition \ref{prop_multitype_trichotomy} and our assumption that $\pi$ is either
age-subcritical or age-critical. It remains to observe
 that  $\int \mathbb{P}( |T^\pi_s| =+\infty ) \, \mathrm{d}\pi(s)=\mathbb{P}( |T^\pi| =+\infty )=0$ and \eqref{gen_fn_vi} follows.

 \eqref{gen_fn_i}, \eqref{gen_fn_ii} and \eqref{gen_fn_iii} are standard facts about generating functions.

 \eqref{gen_fn_iv} follows from the fact that
$|T^\pi_s|$ is stochastically increasing in $s$.

\eqref{gen_fn_v} follows as soon as we observe that
$\int u \,d\pi(u) < \infty$ implies that the law of $T^\pi_s$ depends continuously on $s \in [0,\infty]$ (uniform continuity of $f$ follows from continuity and the Heine-Cantor theorem, since
$[0,\infty] \times \overline{\mathbb{D}}$ is compact).
\end{proof}

\begin{lemma}[A system of nonlinear equations for $f(s,z)$]
\label{lemma_campbell_existence} For each $s \in [0,\infty]$,
 $f(s,z)$ satisfies the recursive equation
 \begin{equation}\label{E: Campbell} f(s,z) = z \exp \int (f(u,z) - 1)(u \wedge s) \,d\pi(u), \qquad z \in \overline{\mathbb{D}}.
 \end{equation}
\end{lemma}

\begin{proof}
Denote by $\lambda_s$ the measure $d\lambda_s(a) = (s \wedge a)\,d\pi(a)$. The root of $T_s$ has offspring of ages $a_1, \dots, a_K$
 which are the points of a Poisson point process on $[0,\infty)$ of intensity $\lambda_s$. Thus $K$ is a Poisson random variable of mean $|\lambda_s|:= \int d \lambda_s(a)  < \infty$ and conditional on $K$, the ages $a_1, \dots, a_K$ are independent with law $\lambda_s/\left|\lambda_s\right|$.
\begin{multline*} f(s,z) = \mathbb{E}\left(z^{\left|T_s^\pi \right|}\right)  =  \mathbb{E}\left(z\,\prod_{i=1}^K z^{\left|T^\pi_{a_i}\right|}\right)
  \\ =
 z \sum_{k=0}^\infty \frac{e^{-\left|\lambda_s\right|}\left|\lambda_s\right|^k}{k!}\prod_{i=1}^k \mathbb{E}_{a \sim \lambda_s/\left|\lambda_s\right|}\left(z^{\left|T_a^\pi\right|}\right)
  =  z \sum_{k=0}^\infty \frac{e^{-\left|\lambda_s\right|} e_s^k}{k!}  \\ = z \exp(e_s - |\lambda_s|),\qquad
\text{ where } \qquad
e_s = \int f(a,z) d \lambda_s(a) \,.
\end{multline*}
\end{proof}

In the case when $\pi$ is age-subcritical, we will show that for each $s$, the generating function $f(s,z)$ has an analytic continuation on a small neighborhood of $z=1$.
On the other hand, in the case when $\pi$ is age-critical, we will show that $f(s,z)$ agrees with an analytic function of $\sqrt{1-z}$ in the intersection of $\mathbb{D}$ and a neighborhood of $1$.

Let us denote by $B_\delta(z_0)=\{ z \in \mathbb{C}\, : \, |z-z_0| \leq \delta \}$ the disc of radius $\delta$ centered at $z_0$ in the complex plane.

\begin{lemma}[A system of nonlinear equations for $\log f(s,z)$]\label{lemma_log_f_eqs}
There exists $\delta>0$ such that for any $s \in [0,\infty]$ and $z \in \mathbb{D} \cap B_\delta(1)$ we have
$|\log f(s,z)| \leq 1$ and
\begin{equation}\label{E: log f recursion}
 \log f(s,z)  =  \log f(\infty,z) -  \int_s^\infty \left(e^{\log f(u,z)} -1\right) (u-s) \,d\pi(u) \,.
\end{equation}
\end{lemma}
\begin{proof}
 In order to achieve $|\log f(s,z)| \leq 1$, we want to choose $\delta>0$ such that
 $f(s,z)$ falls in an appropriate small neighborhood of $1$ for all  $s \in [0,\infty]$ and $z \in \mathbb{D} \cap B_\delta(1)$.
Such a choice is possible by statements \eqref{gen_fn_vi} and \eqref{gen_fn_v} of Lemma \ref{lemma_basic_gen_fn}.

We can thus rewrite~\eqref{E: Campbell} in terms of the function $\log f(s,z)$ as
\[\log f(s,z) = \log z +  \int \left(e^{\log f(u,z)} -1\right)\,(u \wedge s)\d\pi(u). \]
In particular, substituting $s= \infty$ we obtain
\[\log f(\infty,z) = \log z +  \int \left(e^{\log f(u,z)} -1\right)\,u\d\pi(u)\,.\]
Taking the difference of the last two equations and using $u - (u \wedge s) = (u-s)_{+}$ we obtain
\eqref{E: log f recursion}.
\end{proof}

We will use \eqref{E: log f recursion} to show that $f(s,z)$ is an analytic function of $f(\infty,z)$.
This will help us in showing that $f(s,z)$ is analytic in $z$ in a small neighborhood of $1$ in the age-subcritical case and, more importantly,
we will use $w = \log f(\infty,z)$ as a local coordinate to resolve the algebraic singularity of  $f(s,\cdot)$ at $z=1$ in the age-critical case.  The next lemma explains how to find $\log f(s,z)$ in terms of $w$.  The continuity statements are given w.r.t.\ the one-point compactification $[0,\infty]$: a function $f$ defined on $[0,\infty]$ is continuous when its restriction to $[0,\infty)$ is continuous and $\lim_{x \to \infty}  f(x) = f(\infty)$.

\begin{lemma}[A nonlinear Volterra equation]\label{L: final value problem}
 There exists $\delta > 0$ such that for each $w \in B_\delta(0) \subset \mathbb{C}$, there exists a unique bounded solution $F(\cdot,w): [0,\infty] \to \mathbb{C}$ of the equation
 \begin{equation}\label{E: F recursion}
F(s,w) = w - \int_s^\infty \left(e^{F(u,w)} - 1\right)\,(u-s)\,d\pi(u)\,.
  \end{equation}
We have $F(s,0) = 0$ for all $s$. The solution depends continuously on $(s,w) \in [0,\infty] \times B_\delta(0)$. For each fixed $s$ the solution depends analytically on $w$. When we consider the Taylor series of $F$ as a function of $w$, each Taylor coefficient is a bounded and continuous function of $s \in [0,\infty]$.
\end{lemma}

Equation~\eqref{E: F recursion} is a nonlinear Volterra equation of the second kind, posed as a final value problem. It is somewhat non-standard because the integral is with respect to $\pi$ rather than Lebesgue measure, and the assumption that $\pi$ has finite mean will be crucial. We give a detailed proof to keep our argument self-contained, and because it would take just as much space to verify that the equation can be transformed into a standard form satisfying suitable hypotheses to guarantee the conclusions of the lemma.

\subsection{The nonlinear Volterra equation is well-posed}
\label{subsection_volterra_well_posed}

$ $

The goal of \S\ref{subsection_volterra_well_posed} is to prove Lemma \ref{L: final value problem}.

\subsubsection{Proof of existence}

 We begin by proving the existence of a solution of \eqref{E: F recursion}
  with the required properties.
Let us now briefly outline our construction.
 We break $[0,\infty)$ into finitely many intervals of form $[\alpha_i, \alpha_{i-1}]$, where
 \begin{equation}\label{partition_alpha}
 \alpha_k = 0 < \alpha_{k-1} < \dots < \alpha_1 < \alpha_{0}=\infty
 \end{equation}
  and solve  equation \eqref{E: F recursion} piecewise for $i=0,1,2, \dots$, starting at $\infty$.
   On each interval we use a Picard iteration to solve a final value problem with respect to $s$. We will show that on each interval $[\alpha_i, \alpha_{i-1}]$ all the iterates depend analytically on $w$, and are continuous on $[\alpha_i,\alpha_{i-1}] \times B_{\delta_i}(0)$ for some $\delta_i > 0$. We will also show that the iterates converge uniformly. It follows that there exists $\delta>0$ such that the limit $F$ of the iteration is analytic in $w \in B_{\delta}(0)$ for each $s \in [0,\infty]$ and continuous in $(s,w) \in [0,\infty] \times B_{\delta}(0)$, and that for each fixed value of $w \in B_\delta(0)$, $F(\cdot,w)$ solves~\eqref{E: F recursion}.

Let us now start to carry this plan out in detail.

 Also note that $F(\infty,w) = w$, since $\pi$ does not have an atom at $\infty$.

To define the pieces of the iteration, recall that $\int u\,d\pi(u) < \infty$, so we can choose $k \in \mathbb{N}$ and a finite sequence $\alpha_0,\dots,\alpha_{k}$ satisfying \eqref{partition_alpha}
 such that for each $i = 1, \dots, k$, we have
\begin{equation}\label{subdivision_half}
 \int_{\alpha_i}^{\alpha_{i-1}} (u-\alpha_i)\,d\pi(u) < \tfrac{1}{2}\,.
 \end{equation}
We will prove by induction over $i$ that for each $i$ there exists $\delta_i > 0$ such that there is a solution $F(\cdot,w): [\alpha_i,\infty] \to \mathbb{C}$ with the continuity (on $[\alpha_i,\infty]$)  and analyticity
(on $B_{\delta_i}(0)$) required by Lemma \ref{L: final value problem}.

For $i=0$, equation~\eqref{E: F recursion} simply tells us that $F(\infty,w) = w$.

For $i \ge 1$, suppose that we have already a solution on $[\alpha_{i-1},\infty]$, continuous in $(s,w)$ and analytic in $w$. Consider $s \in [\alpha_i,\alpha_{i-1}]$. Applying \eqref{E: F recursion} to both $F(s,w)$ and $F(\alpha_{i-1},w)$ and subtracting, we obtain
\begin{multline}\label{E: induction recursion} F(s,w) = F(\alpha_{i-1},w) - (\alpha_{i-1}-s)\int_{\alpha_{i-1}}^{\infty} \left(e^{F(u,w)}-1\right) \,d\pi(u)\, \\ - \int_s^{\alpha_{i-1}} \left(e^{F(u,w)}-1\right)\,(u-s) \,d\pi(u)\,.
\end{multline}
Note that if $i=1$ then $\alpha_{i-1}=\infty$ by \eqref{partition_alpha}, so the second term on the r.h.s.\ of \eqref{E: induction recursion} is zero, and
thus for $i=1$ \eqref{E: induction recursion} just becomes \eqref{E: F recursion}.

We will solve~\eqref{E: induction recursion} for $s \in [\alpha_i,\alpha_{i-1}]$ and $w$ in some ball $B_{\delta_i}(0)$, where we may choose $\delta_i \in (0,\delta_{i-1}]$. The first two terms on the right-hand side of \eqref{E: induction recursion} are analytic in $w$  and are  jointly continuous in $s$ and $w$ by the induction hypothesis. Only the left-hand side and the final term on the right-hand side depend on the restriction of $F$ to $[\alpha_i,\alpha_{i-1}]$. We solve equation~\eqref{E: induction recursion} by Picard iteration.
Define $F_0$ by
\begin{equation}\label{F_0_def_in_i}
F_0(s,w) = F(\alpha_{i-1},w), \qquad (s,w) \in [\alpha_i,\alpha_{i-1}]\times B_{\delta_i}(0).
\end{equation}
 Then for $j \ge 1$ define inductively
\begin{multline}\label{E: F_iteration}
 F_j(s,w) = F(\alpha_{i-1},w) - (\alpha_{i-1}-s)\int_{\alpha_{i-1}}^{\infty} \left(e^{F(u,w)}-1\right) \,d\pi(u)\, \\ - \int_s^{\alpha_{i-1}} \left(e^{F_{j-1}(u,w)}-1\right)\,(u-s) \,d\pi(u), \qquad s \in [\alpha_i,\alpha_{i-1}].
 \end{multline}
Analogously to \eqref{E: induction recursion}, if $i=1$ then we define the second term on the r.h.s.\ of \eqref{E: F_iteration} to be equal to zero.

Denote by $\|\cdot\|_{\infty,i}$  the infinity norm on the domain $[\alpha_i,\alpha_{i-1}] \times B_{\delta_i}(0)$, where $\delta_i$ will be specified later.
Our next goal is to prove the following lemma.

\begin{lemma}\label{lemma_norm_control}
 With the above notation we have
\begin{equation}\label{E: norm control}
 \|F_{j+1} - F_j\|_{\infty,i}\,  \le
  \tfrac{1}{2} e^{\textup{max}(\|F_j\|_{\infty,i}, \|F_{j-1}\|_{\infty,i})}\, \|F_{j} - F_{j-1}\|_{\infty,i}\, ,  \quad j \geq 1.
\end{equation}
\end{lemma}
\begin{proof}
For any $s \in [\alpha_i,\alpha_{i-1}]$ and $w \in B_{\delta_i}(0)$
we have
\begin{multline}\label{i_norm_bound}
 |F_{j+1}(s,w) - F_j(s,w)| \stackrel{ \eqref{E: F_iteration} }{\le}
  \int_s^{\alpha_{i-1}} \left|e^{F_{j}(u,w)}-e^{F_{j-1}(u,w)}\right|\,(u-s) \,d\pi(u)\,\\
\stackrel{(*)}{\le} \int_s^{\alpha_{i-1}} e^{\textup{max}(|F_j(u,w)|, |F_{j-1}(u,w)|)} |F_{j}(u,w) - F_{j-1}(u,w)|\,(u-s)\,d\pi(u)\\
\le e^{\textup{max}(\|F_j\|_{\infty,i}, \|F_{j-1}\|_{\infty,i})}\, \|F_{j} - F_{j-1}\|_{\infty,i} \,\int_s^{\alpha_{i-1}} (u-\alpha_i)\,d\pi(u),
\end{multline}
where in $(*)$ we used that
 if $x_1, x_2 \in \mathbb{C}$ and $|x_1|\leq a, |x_2|\leq a$ both hold  then $|e^{x_1}-e^{x_2}|\leq e^{a}|x_1-x_2|$ also holds (this inequality follows from the fact that $|\frac{\mathrm{d}}{\mathrm{d}z} e^z| \leq e^{|z|}$).
By \eqref{i_norm_bound} and our assumption \eqref{subdivision_half} we obtain \eqref{E: norm control}.
\end{proof}

Lemma \ref{lemma_norm_control} implies that
  we will have a strict contraction in $\|\cdot\|_{\infty,i}$ as long as we can keep control of the norms $\|F_j\|_{\infty,i}$. This is what we show in the next lemma.

\begin{lemma} One can choose $\delta_i>0$ sufficiently small so that
\begin{align}
\label{E: norm hypothesis}
\|F_0\|_{\infty,i} + 3 \|F_1 - F_0\|_{\infty,i} & \le 1/4\, ,  \\
 \label{E:_one_quarter_bound}
\|F_j\|_{\infty,i} &\le 1/4\, , \quad & j=0,1,\dots \, ,  \\
\label{E: Cauchy sequence}
 \|F_{j} - F_{j-1}\|_{\infty,i}  & \leq  \left(\tfrac{2}{3}\right)^{j-1} \| F_1 - F_0\|_{\infty,i}\, , \quad & j=1,2,\dots\, .
\end{align}
\end{lemma}
\begin{proof} First we show  \eqref{E: norm hypothesis}.
We have
\begin{multline}\label{from_zero_to_one}
   F_1(s,w)-F_0(s,w) \stackrel{\eqref{F_0_def_in_i},\eqref{E: F_iteration} }{=} (\alpha_{i-1}-s)\int_{\alpha_{i-1}}^{\infty} \left(e^{F(u,w)}-1\right) \,d\pi(u)\, \\ - \left(e^{F(\alpha_{i-1},w)}-1\right) \int_s^{\alpha_{i-1}} \,(u-s) \,d\pi(u), \qquad s \in [\alpha_i,\alpha_{i-1}].
\end{multline}
Recall from below \eqref{E: F_iteration} that if $i=1$ then $\alpha_{i-1}=\infty$ and we define the first term on the r.h.s.\ of \eqref{from_zero_to_one} to be equal to zero.
Our induction hypothesis implies that $F$ is jointly continuous in $(s,w) \in [\alpha_{i-1},\infty] \times B_{\delta_{i-1}}(0)$, moreover we have $F_0(s,0) = 0$ by the definition of $F_0$ in \eqref{F_0_def_in_i} and our induction hypothesis (which implies $F(\alpha_{i-1},0)=0$). Putting these observations together with \eqref{subdivision_half} and \eqref{from_zero_to_one}
we see that we can choose (and fix) $\delta_i>0$ sufficiently small to ensure that \eqref{E: norm hypothesis} holds.

We will prove \eqref{E:_one_quarter_bound} and \eqref{E: Cauchy sequence}  simultaneously by induction on $j$.
  First observe that~\eqref{E: norm hypothesis} implies \eqref{E:_one_quarter_bound}
     for $j = 0$ and $j=1$. Also, \eqref{E: Cauchy sequence} obviously holds for $j=1$.
          Now let $j \ge 1$ and suppose that \eqref{E:_one_quarter_bound} and \eqref{E: Cauchy sequence} hold for all indices up to $j$. We have
\begin{equation*}
 \|F_{j+1} - F_j\|_{\infty,i} \stackrel{\eqref{E: norm control},\eqref{E:_one_quarter_bound} }{\leq}
  \tfrac{1}{2} e^{1/4}\, \|F_{j} - F_{j-1}\|_{\infty,i} \leq \tfrac{2}{3}  \|F_{j} - F_{j-1}\|_{\infty,i}\, ,
\end{equation*}
thus \eqref{E: Cauchy sequence} holds for $j+1$. Next we observe that
\begin{equation*}
\|F_{j+1} - F_0\|_{\infty,i}  \le \sum_{k=1}^{j+1}  \|F_{k} - F_{k-1}\|_{\infty,i}
  \stackrel{ \eqref{E: Cauchy sequence} }{\le}  3\|F_1 - F_0\|_{\infty,i}\,.
\end{equation*}
This, together with \eqref{E: norm hypothesis} implies that \eqref{E:_one_quarter_bound} holds for $j+1$, completing the induction step.
\end{proof}
Now we can finish our construction: since~\eqref{E: Cauchy sequence} implies that $F_j$ is a Cauchy sequence with respect to $\|\cdot\|_{\infty,i}$, so it converges. By considering the limit as $j \to \infty$ of both sides of~\eqref{E: F_iteration}, we see that the limit function $F$ solves~\eqref{E: induction recursion} on $[\alpha_i,\alpha_{i-1}]\times B_{\delta_{i}}(0)$.
 Since each $F_j$ is jointly continuous in $(s,w)$, the uniform limit $F$ also has this property.
 Since each $F_j$ is analytic with respect to $w$ on $B_{\delta_{i}}(0)$, the uniform limit $F$ is also analytic by Morera's theorem.
It is easy to see by induction on $j$ that $F_j(s,0)=0$ for all $s \in [\alpha_i,\alpha_{i-1}]$, so $F$ also has this property.

By piecing together the solutions for $i = 1, \dots,k$, and taking $\delta = \delta_k$ we obtain a solution of~\eqref{E: induction recursion} defined on $[0,\infty] \times B_\delta(0)$ that is continuous in $(s,w)$ and analytic in $w$ for each $s$. Note that $F(\cdot,w)$ is also bounded, since it is continuous on the compact space $[0,\infty]$.
 Since the Taylor coefficients of $F$ about $w=0$ can be expressed by the Cauchy integral formula,
   the joint continuity of $F$ implies the continuity of the Taylor coefficients as functions of $s \in [0,\infty]$. The Taylor coefficients of $F$ are also bounded, since the space $[0,\infty]$ is compact.

\subsubsection{Proof of uniqueness} It remains to show the uniqueness of the solution of \eqref{E: F recursion}
  with the required properties.
Suppose that a value $w \in B_\delta(0)$ is fixed, and that $\tilde{F}(\cdot,w)$ is another bounded solution of~\eqref{E: F recursion} defined on $[0,\infty]$.  We may assume that both $\tilde{F}(\cdot,w)$ and $F(\cdot,w)$ are bounded in absolute value by $M < \infty$. Then for every $s$ we have
\begin{eqnarray*} |F(s,w) - \tilde{F}(s,w)| & \stackrel{\eqref{E: F recursion}}{\le} &  \int_{s}^\infty \left|e^{F(u,w)} - e^{\tilde{F}(u,w)}\right|\,(u - s)\,d\pi(u)\\ &
\stackrel{(*)}{\le} & \int_s^\infty  |F(u,w) - \tilde{F}(u,w)| e^M \,u\,d\pi(u)\, ,\end{eqnarray*}
where in $(*)$ we used that
 if $x_1, x_2 \in \mathbb{C}$ and $|x_1|\leq M, |x_2|\leq M$ both hold  then $|e^{x_1}-e^{x_2}|\leq |x_1-x_2|e^{M}$ also holds.
Since $e^M \,u\,d\pi(u)$ is a finite measure, Gr\"onwall's inequality implies that $ \tilde{F}(\cdot,w) \equiv F(\cdot,w)$.

\medskip

The proof of Lemma \ref{L: final value problem} is complete.

\subsection{Taylor coefficients}
\label{subsection_taylor_coefficiants}

The goal of \S\ref{subsection_taylor_coefficiants} is to prove Lemma~\ref{lemma_tree_s_gen_fn_asymp},
therefore we assume that $\pi$ is age-critical in \S\ref{subsection_taylor_coefficiants}.
 Recall from Lemma \ref{L: operator properties}\eqref{L:eigen} that $\mathcal{L}_\pi \theta=\theta$, i.e.,
\begin{equation}\label{theta_eigenfunction}
\theta(s) = \int \theta(u) (u \wedge s) \,d \pi(u)
\end{equation}
and $\int \theta \d\pi = 1$ and $\theta(0) = 0$.
Let us define
\begin{equation}\label{theta_infty}
\theta(\infty):=\lim_{s \to \infty} \theta(s) \stackrel{ \eqref{theta_eigenfunction} }{=} \int \theta(u) u \,d\pi(u).
\end{equation}

Let us recall the notion of the function $F$ from Lemma \ref{L: final value problem}.

\begin{lemma}[Taylor coefficients]\label{lemma_taylor}
Suppose that $\pi$ is age-critical. Then $0 < \theta(\infty) < \infty$ and the Taylor expansion of $F$ about $w=0$ is
\begin{equation}\label{E: Taylor expansion of F}
F(s,w) = \frac{\theta(s)}{\theta(\infty)} w + A_2(s) w^2 + O(w^3)\,,
\end{equation}
where
\begin{equation}\label{A_two_zero}
A_2(0) =  -\frac{\int \theta(u)^3 \,d\pi(u)}{2 \theta(\infty)^2} < 0.
\end{equation}
\end{lemma}

\begin{proof}%[Proof of Lemma~\ref{lemma_taylor}]
$F(s,0) \equiv 0$, so the constant term in the Taylor series vanishes identically. The coefficient of $w$ is $A_1(s):= F_w(s,0) : = \left. \frac{\partial}{\partial w}F(s,w)\right|_{w=0}$. By Lemma \ref{L: final value problem} $s \mapsto F_w(s,0)$  is bounded and continuous on $[0,\infty]$. To compute it, we differentiate both sides of \eqref{E: F recursion} and evaluate at $w=0$.
\begin{eqnarray}\notag F_w(s,0) & = & 1 - \int_{s}^{\infty} F_w(u,0) e^{F(u,0)} (u-s) \d\pi(u)\\ \notag & = & 1 - \int_{s}^{\infty} F_w(u,0) (u-s) \d\pi(u)\,
\\ \label{E: F_w recursion} & = & 1 - \int_0^\infty F_w(u,0) u \d\pi(u) + \int_0^\infty F_w(u,0)\,(u \wedge s)\,d\pi(u)
\end{eqnarray}
Denote by $\mathbf{1}: [0,\infty] \to \mathbb{R}$ the function $\mathbf{1}(s):= 1$. From \eqref{E: F_w recursion} we obtain
\begin{equation}\label{F_identity}
 (I-\mathcal{L}_\pi) F_w(\cdot,0) = \left(1 - \int_0^\infty F_w(u,0)\,u \d\pi(u)\right) \mathbf{1}\,.
 \end{equation}
(Note that the integrals in~\eqref{E: F_w recursion} and~\eqref{F_identity} converge since the first moment of $\pi$ is finite). Take the inner product of both sides of \eqref{F_identity} in $L^2(\pi)$ with $\theta$. Since $I-\mathcal{L}_\pi$ is self-adjoint and annihilates $\theta$, the left-hand side vanishes. Therefore
\begin{equation}
 0 = \left(1 - \int_0^\infty F_w(u,0)\,u \d\pi(u)\right)\, \int \theta(u) \d\pi(u)\,.
 \end{equation}
Now we can use $\int \theta(u) \d\pi(u)=1$ to  conclude
\begin{equation}\label{E: F_w normalisation}
\int_0^\infty F_w(u,0)\,u \d\pi(u) = 1\,.
\end{equation}
Substituting \eqref{E: F_w normalisation} into \eqref{F_identity} we obtain
$ F_w(\cdot,0) = \mathcal{L}_\pi F_w(\cdot,0)$, so by Lemma \ref{L: operator properties}\eqref{L:eigen} we must have $F_w(\cdot,0) = c\theta(\cdot)$ for some $c \in \mathbb{R}$.
We have already noted that $F_w(\cdot,0)$  is bounded and continuous on $[0,\infty]$, thus $\theta(\cdot)$ is bounded and
\begin{equation} 1  \stackrel{ \eqref{E: F_w normalisation} }{=}  c \int_0^\infty  \theta(u) u \d\pi(u) \stackrel{ \eqref{theta_infty} }{=}  c \,\theta(\infty)\,.
\end{equation}
 Hence $0 < \theta(\infty) < \infty$ and $c = 1/\theta(\infty)$.
Thus the first order Taylor coefficient in \eqref{E: Taylor expansion of F} is indeed $A_1(s)= \theta(s)/\theta(\infty)$.

It remains to prove \eqref{A_two_zero}. Differentiating both sides of \eqref{E: F recursion} twice with respect to $w$ and substituting $w=0$, we find
\begin{eqnarray*}
 A_2(s) & = &  - \int_s^\infty \left(A_2(u) + A_1(u)^2\right)(u-s)\,d\pi(u)\\
  & = & -\int_0^\infty \left(A_2(u) + \frac{\theta(u)^2}{2\theta(\infty)^2}\right) \left(u - u \wedge s\right) \,d\pi(u)\,.
\end{eqnarray*}
Hence
$$(I-\mathcal{L}_\pi)A_2 (s)  -  \frac{ \mathcal{L}_\pi \theta^2 (s)}{2\theta(\infty)^2} = - \int_0^\infty \left(A_2(u) + \frac{\theta(u)^2}{2\theta(\infty)^2}\right)\,u\,d\pi(u)\,.$$ The right-hand side of this equation does not depend on $s$, so it must be some constant $c'$. Therefore
\begin{equation}\label{E: A_2 operator equation}
(I-\mathcal{L}_\pi)A_2 = c'\mathbf{1} +  \frac{1}{2\theta(\infty)^2}\mathcal{L}_\pi \theta^2\,.
\end{equation}
%Since $\ker(I-\mathcal{L}) = \textup{span}(\theta)$, this determines $A_2$ up to the addition of $\lambda\theta$. The constant $\lambda$ will be determined by the condition $A_2(\infty) = 0$, after we have determined the constant $c'$. In fact we do not need to find $\lambda$: when we evaluate both sides of~\eqref{E: A_2 operator equation} at $s=0$ we obtain $A_2(0) = c'$, since $Lf(0) = 0$ for every $f \in L^2(\pi)$.
To find $c'$, take the inner product of both sides of \eqref{E: A_2 operator equation} with $\theta$:
\begin{multline}
 0 = \langle (I-\mathcal{L}_\pi)\theta, A_2\rangle_\pi =
 \langle \theta, (I-\mathcal{L}_\pi)A_2 \rangle_\pi
 \stackrel{ \eqref{E: A_2 operator equation} }{=}
  \langle \theta, c' \mathbf{1}\rangle_\pi + \langle \theta, \mathcal{L}_\pi \theta^2 \rangle_\pi / (2 \theta(\infty)^2) \\
   =
   c' \langle \theta, \mathbf{1}\rangle_\pi + \langle \mathcal{L}_\pi \theta, \theta^2\rangle_\pi /(2 \theta(\infty)^2)  =
    c' + \langle\theta, \theta^2\rangle_\pi/(2 \theta(\infty)^2)
     \end{multline}
Substituting $s=0$ into both sides of \eqref{E: A_2 operator equation} and using $(\mathcal{L}_\pi A_2)(0)=0$ and $\theta(0)=0$ we obtain $A_2(0) = c'$ and \eqref{A_two_zero} follows.
\end{proof}

\begin{proof}[Proof of Lemma~\ref{lemma_tree_s_gen_fn_asymp}] Recall the plan outlined in the paragraph before the statement of Lemma \ref{L: final value problem}.
We have the following expression for $\log z$ as an analytic function of $w = \log f(\infty,z)$, valid in some disc around $w=0$:
\begin{equation}\label{log_z_expansion}
\log z
\stackrel{ \eqref{f_s_z_def} }{=}
 \log f(0,z)
  \stackrel{ (*) }{=}
   F(0,w) \stackrel{(**)}{=}
     A_2(0) w^2 + O(w^3)\,,
\end{equation}
where in $(*)$ we used  Lemmas \ref{lemma_log_f_eqs} and  \ref{L: final value problem}, moreover
 in $(**)$ we used Lemma \ref{lemma_taylor} and $\theta(0)=0$.  From \eqref{log_z_expansion} we obtain
 $ 1 - z  = -A_2(0) w^2 + O(w^3),$
  hence $w$ has an expansion around $z=1$ as a convergent power series in $(1-z)^{1/2}$:
  \begin{equation}\label{w_Puiseux}
   w = \frac{1}{\sqrt{-A_2(0)}}(1-z)^{1/2} + a_2(1-z) + a_3(1-z)^{3/2} + \dots\,.
   \end{equation}

Recall from \eqref{phi_theta_cube_int_def} that we defined $\phi := 1/\int \theta(u)^3 d\pi(u)$. Note that we have $\log f(\infty,z) < 0 $ if $0\leq z < 1$, which determines the correct choice of branch of square root in the above expansion. Let $\sqrt{1-z}$ denote the branch of $(1-z)^{1/2}$ that takes positive values for $0\leq z < 1$. Then we have
\begin{multline}\label{f_expansion_log}
 \log f(s,z)  = F(s,w)  \stackrel{ \eqref{E: Taylor expansion of F} }{=} \\ \frac{\theta(s)}{\theta(\infty)} w  +  O(w^2)  \stackrel{ \eqref{A_two_zero}, \eqref{w_Puiseux} }{=}  -\sqrt{2\phi}\,\theta(s) \sqrt{1-z} + O\left(1-z\right)\,.
\end{multline}
Taking the exponential of both sides of \eqref{f_expansion_log} we obtain the desired \eqref{f_s_expansion}.
The implied constant in the error term is uniform as a function of $s$, since the Taylor coefficients of $F$ are bounded by Lemma~\ref{L: final value problem}.

Now the proof of \eqref{f_T_expansion} follows from \eqref{f_s_expansion} using \eqref{f_s_z_def}, $\int  \,d\pi(s)=1$ and $\int \theta(s) \,d\pi(s)=1$. The proof of Lemma \ref{lemma_tree_s_gen_fn_asymp} is complete.
\end{proof}

\subsection{Critical forest fire equations}
\label{subsection_smol_uniqueness}

The goal of \S\ref{subsection_smol_uniqueness} is to prove Proposition \ref{prop_Smol_uniqueness_simplified}.
 In the proof we will use generating functions. In fact we will prove a more general result where Assumption \ref{assumption_v_k_0_comes_from_BP} is replaced
 with an assumption on the initial state $\underline{v}(0)$ which is formulated entirely in terms of generating functions:

\begin{definition}[Analytic inverse of generating function]\label{def_analytic_inverse}
Let $\underline{v}=(v_k)_{k=1}^{\infty}$ satisfy $v_k\geq 0$ and $\sum_{k=1}^{\infty} v_k =1$.
We denote by $f(z)=\sum_{k=1}^\infty v_k z^k$ the generating function of $\underline{v}$.
Noting that $f(0)=0$, $f(1)=1$ and $f$ is strictly increasing on $[0,1]$, the inverse
function $f^{-1}$ of $f$ is well-defined on $[0,1]$ and satisfies $f^{-1}(1)=1$ and $(f^{-1})'(1)\geq 0$.
We say that $f$ has an analytic inverse if
  there exists some $\delta>0$ and
 a complex analytic function
$g: B_\delta(1) \to \mathbb{C}$, satisfying either $g'(1) > 0$ or both $g'(1) = 0$ and $g''(1) <0$, such that
\begin{equation}
g(z)=f^{-1}(z), \qquad z \in [1-\delta,1].
\end{equation}
\end{definition}
In other words, the generating function $f$ of $\underline{v}$ has the analytic inverse property if $f^{-1}$ can be extended analytically to  $B_\delta(1)$ for some $\delta > 0$,
moreover the second order Taylor polynomial of $f^{-1}$ about $z=1$ is not identically zero.

\begin{remark} Assume that $\underline{v}$ satisfies Definition \ref{def_analytic_inverse}.

  If $g'(1) > 0$ then we have  $\sum_{k=1}^{\infty} k v_k=f'(1)=1/g'(1)$ and $f$ is analytic in a small neighborhood of $z=1$ and thus $v_k$ decays exponentially as $k \to \infty$ by the Vivanti--Pringsheim theorem.

 If  $g'(1) = 0$ and $g''(1) <0$ hold, then $f$ cannot be extended analytically to a small neighborhood of $z=1$, and by  Example (c) of Theorem 4 of chapter XIII.5 of \cite{Feller}
 we have $\sum_{\ell=k}^\infty v_\ell \approx \sqrt{ \frac{-2 }{\pi g''(1)  }} k^{-1/2}$ as $k \to \infty$.
\end{remark}

The next result is a modification of \cite[Theorem 1]{RathToth} (but neither result follows from the other one).

\begin{proposition}[Critical forest fire equations]\label{prop_Smol_uniqueness_general}
Let us assume that the initial condition $\underline{v}(0)=(v_k(0))_{k=1}^{\infty}$ satisfies Definition \ref{def_analytic_inverse}.
Then the critical forest fire equations (\eqref{smol_ff_eq}+\eqref{smol_bc})
 have a unique solution $\underline{v}(\cdot)$, which also satisfies the properties listed in  Proposition \ref{prop_Smol_uniqueness_simplified}.
\end{proposition}

We will prove Proposition \ref{prop_Smol_uniqueness_general} in \S \ref{subsub_ffe_wellposed}.
 Proposition \ref{prop_Smol_uniqueness_simplified} immediately follows from the result of Proposition \ref{prop_Smol_uniqueness_general} as soon as we prove the next lemma.

\begin{lemma}[The law of $|T^\pi|$ satisfies the analytic inverse property]\label{lemma_tree_gen_fn_symp}
Suppose that $\int u \,d\pi(u) < \infty$ holds and $\pi$ is either age-critical or age-sub\-critical. Let us define $\underline{v}=(v_k)_{k=1}^{\infty}$ by $v_k:=\mathbb{P}( |T^\pi|=k ), k=1,2,\dots$.
Under these conditions $\underline{v}$ satisfies Definition \ref{def_analytic_inverse}.
Moreover, if $\pi$ is age-subcritical then $g'(1) > 0$ and if $\pi$ is age-critical then $g'(1) = 0$ and $g''(1) <0$.
\end{lemma}

In other words, if $\underline{v}(0)$ satisfies Assumption \ref{assumption_v_k_0_comes_from_BP} then it also satisfies Definition \ref{def_analytic_inverse}.

\begin{proof}[Proof of Lemma~\ref{lemma_tree_gen_fn_symp}]
Let us recall the notion of $f(z)$ and $f(s,z)$ from \eqref{f_s_z_def} and $F(s,w)$ from  Lemma \ref{L: final value problem}.
We will construct the analytic inverse function $g$ of $f$ required by Definition \ref{def_analytic_inverse} by letting $g=G \circ H^{-1}$, where
\begin{equation}
G(z):=e^{F(0,\, \log(z))},
\quad
H(z):=\int e^{F(s,\,\log(z))}\, \mathrm{d}\pi(s).
\end{equation}
Note that $G(1)=1$ and $H(1)=1$, moreover $G$ and $H$ are complex analytic functions in a  neighborhood of $z=1$.
By  Lemmas \ref{lemma_log_f_eqs} and \ref{L: final value problem}
we have
\begin{equation}\label{f_F_relations}
f(s,z)=e^{F(s,\, \log(f(\infty,z)))},
\quad
f(0,z)
\stackrel{ \eqref{f_s_z_def} }{=}
 z=e^{F(0,\, \log(f(\infty,z)))},
\end{equation}
therefore for any $z \in (1-\epsilon,1]$, where $\epsilon > 0$ is small enough,  we have
\begin{equation}\label{G_H_useful_facts}
G(f(\infty,z))= z,  \qquad f(z) \stackrel{ \eqref{f_s_z_def} }{=} H(f(\infty,z)).
\end{equation}
We want $H$ to be invertible in a small neighborhood of $1$, therefore our next goal is to show that $0<H'(1)<\infty$ holds. First note that $H'(1)=\int F_w(s,0)\, \mathrm{d}\pi(s)$ follows
from the definition of $H$. Next note that $F(s,0)=0$ and $F(s,w)<0$ if $w<0$ by \eqref{f_F_relations}, thus
$F_w(s,0) \geq 0$. Also note that $\sup_{s\geq 0}F_w(s,0)<\infty$ by Lemma \ref{L: final value problem}, thus $H'(1)<\infty$. Finally, we have to rule out the possibility that
$H'(1)=0$.
 Let us denote $A_1(s)=F_w(s,0)$ and observe that $A_1(s)=1-\int_s^{\infty} A_1(u)(u-s)\, \mathrm{d}\pi(u)$ holds by  \eqref{E: F_w recursion}  (even if  $\pi$ is age-subcritical). Differentiating this
 w.r.t.\ $s$ we obtain  $A'_1(s)=\int_s^\infty A_1(u)\, \mathrm{d}\pi(u)$. Now $A_1(\infty)=1$ and $A_1(s)\geq 0$, so if we indirectly
 assume $H'(1)=\int_0^\infty A_1(u)\, \mathrm{d}\pi(u)=0$, then $A'_1(s) \equiv 0$, so $A_1(s) \equiv A_1(\infty)$,
 which contradicts $\int_0^\infty A_1(u)\, \mathrm{d}\pi(u)=0$.
This completes the proof of $0<H'(1)<\infty$, which implies that $H$ has a complex analytic inverse function $H^{-1}$ in a small neighborhood of $z=1$, so
the desired inverse function $g(\cdot)$ of $f(\cdot)$ can be defined as $g=G \circ H^{-1}$ by \eqref{G_H_useful_facts}, moreover $g$ is complex analytic in a small neighborhood of $z=1$.

Note that $f'(1)=\sum_{k=1}^{\infty} k v_k=\mathbb{E}(|T^\pi|)$.

If $\pi$ is age-subcritical then $f'(1)<\infty$  by Proposition \ref{prop_multitype_trichotomy}, thus $g'(1)>0$.

If $\pi$ is age-critical then $f'(1)=\infty$ by Proposition \ref{prop_multitype_trichotomy}, thus $g'(1)=0$.
It remains to show that $g''(1)<0$. One can check that $G'(1)=0$ and thus $g''(1)=G''(1)/(H'(1))^2$.
Noting that  by Lemma \ref{lemma_taylor} we have $H'(1)=1/\theta(\infty)$, we can therefore write $g''(1)=G''(1) \theta^2(\infty)$.
We can then use $F_w(0,0)=A_1(0)=\theta(0)/\theta(\infty)=0$ to deduce $G''(1)=2A_2(0)$, whence $g''(1)=-1/\phi<0$ by \eqref{phi_theta_cube_int_def} and \eqref{A_two_zero}.

\end{proof}

\subsubsection{Analytic inverse property implies well-posedness}
\label{subsub_ffe_wellposed}

\begin{proof}[Proof of Proposition \ref{prop_Smol_uniqueness_general}] Let $f_0(z)=\sum_{k=1}^\infty v_k(0)z^k$.
The condition that\\   $\sum_{k=1}^\infty k v_k(0)<\infty$ is equivalent to $f'_0(1)<\infty$, which is in turn equivalent to $g'(1)> 0$, where $g$ is the inverse function of $f_0$.

 If $g'(1)>0$ then both $f_0$ and $g$ are invertible complex analytic functions in a small neighborhood of $1$. In particular, this implies that $f_0^{'''}(1)<\infty$ and thus
 $\sum_{k=1}^\infty k^3 v_k(0)<\infty$, hence  we can apply \cite[Theorem 1]{RathToth} to conclude the proof of Proposition \ref{prop_Smol_uniqueness_general}.

 The condition $g'(1)=0$ is equivalent to $\sum_{k=1}^\infty k v_k(0)=\infty$, which is in turn equivalent to  $t_{\mathrm{gel}}=0$ (see \eqref{def_gelation_time}). Recall from \cite[(41)]{RathToth} that we define
 $V_0(x)=f_0(e^{-x})-1$ for any $x \in [0, \infty)$. Recalling \cite[(57)]{RathToth} we define
 $
 E(x)=-V'_0(x)^3 / V^{''}_0(x)
 $.
In \cite[\S 4.2, \S 4.3]{RathToth} the uniqueness of the equations (\eqref{smol_ff_eq}+\eqref{smol_bc}) and all of the other conclusions of Proposition \ref{prop_Smol_uniqueness_simplified} are proved under the assumptions
 \begin{equation}\label{enough_for_smol_uniqueness}
 \lim_{x \to 0_+} E(x) \in (0,\infty), \qquad
   E'(x)= O(x^{-1/2}) \;\text{ as }\; x \to 0_+
   \end{equation}
  on the initial condition $(v_k(0))_{k=1}^{\infty}$, thus we only need to
 show that the assumptions of Proposition \ref{prop_Smol_uniqueness_general} imply \eqref{enough_for_smol_uniqueness}. Indeed: denote the inverse function of $-V_0(x)$ by $X(u)$, i.e.,
 $X(-V_0(x))=x$, $x \in [0,\infty)$. By  \cite[(63)]{RathToth} we have
 \begin{equation}\label{E_inv_second_der}
 E(x)=1/X''(-V_0(x)).
 \end{equation}

Using the above definitions we can write  $X(u)=-\log(g(1-u))$, thus we can use  $g'(1)=0$ to derive
\begin{equation}\label{E_limit_finite}
\lim_{x \to 0_+} E(x)= 1/X''(0)=-1/g''(1) \in (0,\infty)
\end{equation}
 using the assumption of Proposition \ref{prop_Smol_uniqueness_general}. Thus
the first statement of \eqref{enough_for_smol_uniqueness} holds.
Differentiating \eqref{E_inv_second_der} with respect to $x$ we obtain
\begin{equation}
E'(x)=E(x)^2 X^{'''}(-V_0(x)) V'_0(x).
\end{equation}
Now $\lim_{x \to 0_+} E(x)^2 X^{'''}(-V_0(x))$ exists and is finite by the assumption that $g$ is analytic
near $1$, and $|V'_0(x)|= O(x^{-1/2})$ follows from $t_{\mathrm{gel}}=0$, \eqref{E_limit_finite} and \cite[Lemma 3, (61)]{RathToth}.
This concludes the proof of \eqref{enough_for_smol_uniqueness}.
\end{proof}

Putting together the results of Proposition \ref{prop_Smol_uniqueness_general} and Lemma \ref{lemma_tree_gen_fn_symp},
we obtain the proof of Proposition \ref{prop_Smol_uniqueness_simplified}.

\section{Age evolution}\label{section_cluster_growth_age}

In \S\ref{section_cluster_growth_age} we prove our main results stated in \S\ref{subsection_intro_main_results}.
In order to do so, we need to  introduce some auxiliary Markov processes. In \S \ref{subsection_cgp_w_a} below
we will provide the detailed definitions of these processes and state some key lemmas about them. Let us now give a brief outline of the contents of
\S\ref{subsection_cgp_w_a}.

We first recall from \cite{CraneFreemanToth} the cluster growth process $(C_t)$ driven by a solution to the forest fire equations
(\eqref{smol_ff_eq}+ \eqref{smol_bc}) and augment it with its age process $(a_t)$ to obtain the \emph{cluster process with age} $(a_t,C_t)$ (see Definition \ref{def:cluster_growth_decorated_with_ages}).
We will identify the probability distribution $\pi_t$ that appears in the statements of our main results (Theorems \ref{thm_convergence}, \ref{thm_local_limit_mfff} and \ref{thm_age_pde})  as the marginal distribution of  $a_t$.

  The dynamics of the process
$(C_t)$ are approximated by the dynamics of the component size $(C^n_t(\rho^n))$ of a uniformly chosen vertex $\rho^n$ in the MFFF when $n \gg 1$. The two processes can be coupled in a way that with high probability they agree exactly except close to their burning or explosion times, see Theorem~\ref{CFTtheorem}.

 In Theorem \ref{T: asymptotic independence} we will upgrade the result of Theorem \ref{CFTtheorem}
and couple the cluster processes $(C^n_t(\rho_1^n), C^n_t(\rho_2^n) )$ of i.i.d.\ uniform vertices  $\rho_1^n$ and $\rho_2^n$ to a pair $( C_t^{(1)}, C_t^{(2)} )$ of i.i.d.\ copies of $(C_t)$.  In Proposition \ref{2rhocvg} we extend the coupling of Theorem \ref{T: asymptotic independence} to the setting where
 we keep track of the ages of the tagged vertices as well as the size of their clusters. In Corollary \ref{Cor: pointwise joint convergence} we conclude that, for any fixed $t$,
  $(  a_t^n(\rho_i^n), C_t^n(\rho_i^n))_{i \in \{1,2\}}$  weakly converge to i.i.d.\ copies  $( a_t^{(i)},C_t^{(i)})_{i \in \{1,2\}}$ of the joint distribution at time $t$ of the age $a_t$ and cluster size $C_t$ in the cluster process with age.

We prove concentration of the empirical age measure $\pi^n_t$ around $\pi_t$ using Corollary \ref{Cor: pointwise joint convergence}  and a second moment argument,  which gives the proof of  Theorem \ref{thm_convergence}.
The proof of the local weak limit statements of Theorem \ref{thm_local_limit_mfff} will also follow using the fact that given the ages, the graph that we see in the MFFFA at time $t$ is an age-driven IRG.

 In Lemma \ref{ConditionalCt} we show that conditional on $(a_s)_{0\le s\le t}$, the law of $C_t$ agrees with the size of the multitype Galton--Watson tree $T^{\pi_t}_{a_t}$ where the root has age $a_t$.
   Finally, we show in Lemma \ref{lemma:age_indeed_evolves_as_Markov} that the process $(a_t)_{t \geq 0}$ itself has the Markov property, and the proof of
 Theorem \ref{thm_age_pde} will follow as soon as we look at the Kolmogorov forward equations of $(a_t)_{t \geq 0}$ the right way.

\subsection{Cluster growth processes with age}
\label{subsection_cgp_w_a}

We extend the definition of the cluster growth process associated with the forest fire model from \cite[Definition 1.6]{CraneFreemanToth}, by augmenting it with the age of the root vertex. Denote by $E=(\mathbb{N}, d_E)$ the metric space with metric
$d_E(i,j)=|f(i)-f(j)|$, where $f(i)=1/i$ for $i \in \{2,3,4,\dots\}$ but $f(1)=0$. Thus $E$ is a compact metric space in which $\lim_{n \to \infty} n =1$, and $1$ is the only accumulation point of $E$. A function \hbox{$g: E  \to \mathbb{R}$} is continuous if and only if $\lim_{k \to \infty} g(k) = g(1)$.

\begin{definition}[Cluster growth process with age]\label{def:cluster_growth_decorated_with_ages}
Given an initial distribution $\pi_0$
 satisfying $\int u \,d\pi_0(u) < \infty$ that is age-critical or age-subcritical, denote by $(a_0, C_0)$ the joint realization
of the age of the root and the total number of vertices of $T^{\pi_0}$.  Let $v_k(0):=\mathbb{P}( |T^{\pi_0}|=k)$ and denote by $(\underline v(t))$ the corresponding solution of the critical forest fire equations (c.f.\ Proposition \ref{prop_Smol_uniqueness_simplified}).
We will use this solution to construct (in \S \ref{subsection_construction_of_cgp}) the \emph{cluster growth process with age} $(a_t,C_t)_{t \geq 0}$, which is an inhomogeneous $([0,\infty) \times E )$-valued continuous-time c\`adl\`ag Markov process
such that
\begin{itemize}
\item $a_t$ increases at rate $1$,
\item $C_t$ jumps from state $i$ with rate $i$, and conditional on jumping from $i$ at time $s$, it jumps to $i+k$ with probability $v_k(s)$,
\item whenever $C_t$ explodes, $C_t$ returns immediately to state $1$ and $a_t$ jumps immediately to $0$.
\end{itemize}
For any $t \in [0,\infty)$, we denote by  $\pi_t$ the distribution of $a_t$.
\end{definition}

The marginal process $(C_t: t \ge 0)$ is a Markov process on its own. It is the cluster growth process studied in \cite{CraneFreemanToth}, extended to allow a critical initial cluster size distribution $v_k(0) = \mathbb{P}(|T^{\pi_0}| = k)$ by Propositions~\ref{prop_Smol_uniqueness_simplified}~and~\ref{prop_RathToth_conv}.

Let us stress that the probability distribution $\pi_t$ from Definition \ref{def:cluster_growth_decorated_with_ages} will play the role of the $\pi_t$  in our main results
(i.e., Theorems \ref{thm_convergence}, \ref{thm_local_limit_mfff}, \ref{thm_age_pde}).

 Note that Definition \ref{def:cluster_growth_decorated_with_ages} implies
\begin{equation}\label{pi_t_finite_mean_eq_short}
\int u \,d\pi_t(u) = \mathbb{E}(a_t) \leq \mathbb{E}(a_0+t) = \int u \,d\pi_0(u) +t < \infty.
\end{equation}
Note that almost surely,  $(C_t: t \ge 0)$ has infinitely many explosion times, but the set of explosion times has no accumulation points (c.f.\ \cite[Lemma 3.16]{CraneFreemanToth}). In \cite[Section 3.4]{CraneFreemanToth} it is proved that for any $t,h \in \mathbb{R}_+$ we have
\begin{equation}\label{expected_number_of_explosions_and_phi}
\int_t^{t+h} \varphi(s)\, \mathrm{d}s=\mathbb{E}\left(\,  \# \left\{ \, s \in [t,t+h]  \, :   \, \text{$C$ explodes at time $s$}    \, \right\}  \,\right),
\end{equation}
where  $\varphi(s)$ is defined in Proposition \ref{prop_Smol_uniqueness_simplified} for $s \geq t_{\mathrm{gel}}$ and $\varphi(s)=0$ if $s < t_{\mathrm{gel}}$.

Let us recall \cite[Proposition 1.9]{CraneFreemanToth} about the marginal distributions of $C$. Again the proof of this extends to include our critical initial condition once we have Propositions~\ref{prop_Smol_uniqueness_simplified}~and~\ref{prop_RathToth_conv}.
\begin{proposition}[The law of $C_t$ is $\underline{v}(t)$]\label{prop_cft_cluster_marginal}
For all $t \geq 0$ and all $k \in \mathbb{N}$ the cluster growth process $C$ satisfies
\begin{equation}\label{eq_clgp_marginal_v_k_t}
\mathbb{P}[C_t=k]=v_k(t),
\end{equation}
 where
$(\underline v(t))$ denotes the solution of the critical forest fire equations (c.f.\ Proposition \ref{prop_Smol_uniqueness_simplified}) with initial condition $\underline{v}(0)$.
\end{proposition}

The main result of \cite{CraneFreemanToth} is stated in Theorem \ref{CFTtheorem} below.
Loosely speaking, Theorem \ref{CFTtheorem} states that  for large $n$, the time evolution of the cluster size of
 a uniformly chosen vertex $\rho^n$ in the self-organized critical MFFF looks like the cluster growth process $C$. In particular, burning times of $\rho^n$ correspond to explosion times of $C$.
For further intuitive discussion of the cluster growth process $C$ and how it arises from the MFFF,
 see \cite[\S 1.2]{CraneFreemanToth}.

\begin{definition}[Components and cardinalities at time $t$]
\label{def_comps_sizes_mfff_t}

$ $

Given a $\mathrm{MFFF}(n,\lambda)$ process (see Definition \ref{def:graphical_mfff})
 and a vertex $i \in [n]$, let $\mathcal{C}^n_t(i)$ denote the connected component of vertex $i$ in the graph at time $t$,  moreover let $C^n_t(i)$ denote the number of vertices in $\mathcal{C}^n_t(i)$, i.e., $C^n_t(i)=|\mathcal{C}^n_t(i)|$.
\end{definition}

The next theorem is \cite[ Theorem 1.7]{CraneFreemanToth}, extended to include the case where the limiting initial condition $\left(v_k(0)\right)$ is the law of $\left| T^{\pi_0}\right|$ for an age-critical $\pi_0$ with finite mean. The proof in \cite{CraneFreemanToth} extends to this case once we have Propositions~\ref{prop_Smol_uniqueness_simplified}~and~\ref{prop_RathToth_conv}.
\begin{theorem}[Coupling of $(C^n_t(\rho^n))$ and $(C_t)$]\label{CFTtheorem}
Let $(\mathcal{G}^n)$ be a sequence of MFFF processes satisfying the conditions of Proposition \ref{prop_RathToth_conv}, and let $\rho^n$ be a uniformly chosen vertex from $[n]$. There exists a coupling of $(C^n_t(\rho^n))$ and $(C_t)$ such that for any  $t_{\max}>0$,
$$\sup_{t \in [0, t_{\max}]} d_E\left( C^n_t(\rho^n), C_t\right)   \stackrel{\mathbb{P}}{\to} 0\quad \text{as} \quad n \to \infty.  $$
\end{theorem}

We will show a convergence result analogous to Theorem \ref{CFTtheorem} when we study the clusters of \emph{two} uniformly chosen vertices in a sequence of MFFF processes, and show that they are asymptotically independent.

\begin{theorem}[Coupling of $(C^n_t(\rho^n_i))_{i=1}^2$ and $(C^{(i)}_t)_{i=1}^2$] \label{T: asymptotic independence}
Let $(\mathcal{G}^n)$ be a sequence of MFFF processes as in Theorem \ref{CFTtheorem} and let $\rho^n_1,\rho^n_2$ be independent uniform choices of vertices from $[n]$. For each $n$ there exists a coupling of the process $(C^n_t(\rho^n_1), C^n_t(\rho^n_2);\,t\in[0,t_{\max}])$ to the process $(C_t^{(1)},C_t^{(2)};\,t\in[0,t_{\max}])$ consisting of two independent copies of the limit cluster growth process $(C_t)$, such that
\begin{equation}\label{pair_coupling_i_1_2}
\sup_{t \in [0, t_{\max}]} \d_E\left(C^n_t(\rho^n_i), C^{(i)}_t\right) \stackrel{\mathbb{P}}{\to} 0\quad \text{as $ \; \; n \to \infty$, for $i = 1,2$}.
\end{equation}
\end{theorem}

We will prove Theorem \ref{T: asymptotic independence} in \S \ref{subsection_indep_of_clust_proc}.

\begin{remark}\label{remark_propagation_of_chaos_for_C}
 Let us fix $k \in \mathbb{N}$. The proof of Theorem \ref{T: asymptotic independence} can be easily adjusted to the case of $(C^n_t(\rho^n_i))_{i=1}^k$, where $\rho^n_1,\dots, \rho^n_k$ are
i.i.d.\ with uniform distribution on $[n]$, and in this case the joint distributional limit  $(C^{(i)}_t)_{i=1}^k$ consists of $k$ i.i.d.\ copies of $(C_t)$. Thus, in the
terminology of statistical physics, the MFFF model has the \emph{propagation of chaos} property.

Note that $(C_t)$ fits into the framework of McKean--Vlasov equations with jumps. McKean--Vlasov equations are stochastic differential equations whose coefficients at each time depend on the law of their own solution at that time (similarly to our Proposition \ref{prop_cft_cluster_marginal}).
%Note that $(C_t)$ shares some features with the framework of McKean--Vlasov equations, which are stochastic differential equations whose coefficients at each time depend on the law of their own solution at that time (similarly to our Proposition \ref{prop_cft_cluster_marginal}).
 McKean--Vlasov equations arise as weak limits of the motion of a tagged particle in an exchangeable interacting particle system  which, similarly to our Theorem \ref{T: asymptotic independence}, exhibit propagation of chaos (c.f.\ \cite{sznitman_propagation}).
\end{remark}

We now lift Theorem \ref{T: asymptotic independence} to the setting with ages. Let us stress that we start the MFFFA from an age-driven IRG with possibly random initial ages
(c.f.\ Assumption \ref{assumption_age_driven_initial_state}).

\begin{proposition}[Coupling of cluster size and age processes]\label{2rhocvg} Let $(\mathcal{G}^n)$ be a sequence of MFFFAs satisfying the conditions of Theorem \ref{thm_convergence},
 and let $\rho^n_1,\rho^n_2$ be i.i.d.\ uniformly chosen vertices in $[n]$. There is a coupling of $(\mathcal{G}^n, \rho^n_1,\rho^n_2)$ with $\left((a^{(i)}_t,C^{(i)}_t),\,t\in[0,t_{\max}]\right)_{i\in \{1,2\}}$,
which are two i.i.d.\ copies of the process defined in Definition \ref{def:cluster_growth_decorated_with_ages}, such that we have both \eqref{pair_coupling_i_1_2} and
$$
\sup_{t \in [0, t_{\max}]} \left|\frac{a^n_t(\rho^n_i)}{C^n_t(\rho^n_i)} - \frac{a^{(i)}_t}{C^{(i)}_t}\right| \stackrel{\mathbb{P}}{\to} 0\quad \text{as $ \; \; n \to \infty$, for $i = 1,2$.}
$$
\end{proposition}
In the coupling the age of the tagged vertex $\rho_i^n$ does not jump to $0$ at exactly the same times as the age $a_t^{(i)}$ jumps to zero, so there are short intervals during which one age is close to zero and the other is not. The odd-looking division of the age by the cluster size takes care of this, since during this interval where the ages disagree substantially, one cluster size is very large and the other age is very small.

We deduce Proposition~\ref{2rhocvg} from Theorem \ref{T: asymptotic independence} in \S \ref{subsection_joint_conv_of_cluster_growth_and_age}.

Proposition~\ref{2rhocvg} implies the following statement, which is in fact all we will need for the proof of Theorems \ref{thm_convergence}, \ref{thm_local_limit_mfff} and \ref{thm_age_pde}.

\begin{corollary}[Convergence of ages and cluster sizes at a fixed time]\label{Cor: pointwise joint convergence}
In the situation of Proposition~\ref{2rhocvg}, for each fixed $t \geq 0$  we have
$$\left(  a_t^n(\rho_i^n), C_t^n(\rho_i^n)\right)_{i \in \{1,2\}} \Rightarrow \left( a_t^{(i)},C_t^{(i)}\right)_{i \in \{1,2\}}\quad \text{as $ \;\; n \to \infty$,}$$ where the convergence in distribution is with respect to $( [0,\infty) \times E )^2$.
\end{corollary}
We will deduce Corollary \ref{Cor: pointwise joint convergence} from Proposition~\ref{2rhocvg} in \S \ref{subsection_joint_conv_of_cluster_growth_and_age}.

 Recall the notion of the empirical age measure $\pi^n_t$ from \eqref{pi_n_t_empirical_measure}.
In \S\ref{subsection_conc_emp_age_distr} we will use Corollary \ref{Cor: pointwise joint convergence} and a second moment argument to prove the convergence of $\pi^n_t \stackrel{\mathbb{P}}{\Rightarrow} \pi_t$ stated in Theorem \ref{thm_convergence}.

In \S\ref{subsection_proof_of_thm_local_limit_mfff} we will use Theorem \ref{thm_convergence} to deduce Theorem \ref{thm_local_limit_mfff} about the weak local convergence of the MFFF graph $\mathcal{G}^n_t$ to the MBP tree $T^{\pi_t}$.

  Recall from Definition \ref{def:branchingprocess} the notion of the age-driven MBP tree $T^\pi_s$ where the age of the root is specified to be $s$.

The following statement bears some resemblance to Theorem \ref{thm_mfff_inhomog}.
\begin{lemma}[$(C_t)$ and $(a_t)$ are intertwined]\label{ConditionalCt}
Conditional on $(a_s)_{0\le s\le t}$,\; $C_t\stackrel{d}= |T^{\pi_t}_{a_t}|$.
\end{lemma}

We will prove Lemma \ref{ConditionalCt} in \S \ref{subsection_intertwine_C_a_cgp}.

\begin{remark}\label{cgp_intertwined}
Similarly to Remark \ref{remark_MFFF_intertwined}, in the terminology of of \cite[\S 3.2]{swart_intertwining},
the Markov process $(C_t)$ is \emph{intertwined} on top of $(a_t)$ (c.f.\ \cite[Proposition 3.4]{swart_intertwining}).
\end{remark}

Recall the notion of the control function $\varphi(t)$ associated to the solution of the critical forest fire equations from Proposition \ref{prop_Smol_uniqueness_simplified}.
For $t < t_{\mathrm{gel}}$ we may assume that $\varphi(t)=0$.
Note that $\int s \, \mathrm{d} \pi_t(s) <\infty$ for all $t \geq t_{\mathrm{gel}}$ by \eqref{pi_t_finite_mean_eq_short}.
 For $t \geq t_{\mathrm{gel}}$ the probability distribution
 $\pi_t$ is age-critical  by Theorem \ref{thm_local_limit_mfff}\eqref{age_crit_after_gel}.
 Moreover, the eigenfunction $\theta_t(\cdot)$ of $\mathcal{L}_{\pi_t}$ corresponding to the eigenvalue $\lambda=1$ is uniquely defined for all $t \geq 0$ by Lemma \ref{L: operator properties}\eqref{L:eigen}.

Let us define an auxiliary $[0,\infty)$-valued continuous-time inhomogeneous Markov process that we will use to give an autonomous description of the time evolution of the ages in $(a_t,C_t)_{t \geq 0}$.

\begin{definition}[An auxiliary Markov process]\label{def:A_t_age_Markov_proc}
Let us define a Markov process $(A_t)_{t\geq 0}$, with initial distribution $A_0 \sim \pi_0$ and infinitesimal generator
\begin{equation}\label{inf_gen_A_t}
\lim_{h \to 0_+} \frac{1}{h}
\mathbb{E}\left( f(A_{t+h})-f(s)\, | \, A_t=s  \right)=
f'(s)+ \varphi(t)\theta_t(s) (f(0)-f(s)),
\end{equation}
for every $f:[0,\infty) \to \mathbb{R}$, a compactly supported smooth test function.
\end{definition}
So $A$ increases deterministically at speed $1$ and if $A_t=s$ then $A$ jumps to zero at time $t$ at rate $\varphi(t)\theta_t(s)$.

\begin{lemma}[The age process is Markov]\label{lemma:age_indeed_evolves_as_Markov}
The process $(a_t)_{t \geq 0}$ (defined in Definition \ref{def:cluster_growth_decorated_with_ages}) has the same law as the process  $(A_t)_{t \geq 0}$ (defined in Definition \ref{def:A_t_age_Markov_proc}).
\end{lemma}

 We will prove Lemma \ref{lemma:age_indeed_evolves_as_Markov} in \S \ref{subsection_age_is}.
As an ingredient of that proof, we will also prove $\varphi(t)=(\int \theta_t(s)^3\, \mathrm{d}\pi_t(s))^{-1}$ when $t\ge t_{\mathrm{gel}}$ (i.e., \eqref{eq:phiexpression}).

\begin{remark}
Similarly to  Remark \ref{remark_propagation_of_chaos_for_C}, the McKean--Vlasov analogy applies to the age process $(a_t)$, which arises as the weak limit of the motion of a tagged particle in an exchangeable interacting particle system which exhibits propagation of chaos (c.f.\ Proposition \ref{2rhocvg}). In addition, $(a_t)$ is a Markov process governed by its own distribution: on the one hand the coefficient $\varphi(t)\theta_t(s)$ in  the infinitesimal generator \eqref{inf_gen_A_t} (which governs $(a_t)$ by Lemma \ref{lemma:age_indeed_evolves_as_Markov})  is determined by  $\pi_t$ (c.f.\  \eqref{eq:phiexpression} and Lemma \ref{L: operator properties}\eqref{L:eigen}), on the other hand $\pi_t$ is the distribution of $a_t$ (c.f.\ Definition \ref{def:cluster_growth_decorated_with_ages}).
\end{remark}

Finally, in \S\ref{subsection_proof_of_age_pde} we deduce Theorem~\ref{thm_age_pde}, i.e., that $(\pi_t)$ satisfies the age differential equations,
which turn out to be closely related to the Kolmogorov forward equations of the Markov process $(a_t)$.

%%%%%%%%%%%%%%%%%%%%%%%%%%%%%%

\subsection{Construction of the cluster growth process with age}
\label{subsection_construction_of_cgp}

In our  construction of $(a_t,C_t)$ satisfying Definition \ref{def:cluster_growth_decorated_with_ages},  we rely on \cite[Lemma 3.16]{CraneFreemanToth} about the almost sure explosion in finite time of $C_t$, where the stated hypotheses require the initial condition to be subcritical. However, the proof of \cite[Lemma 3.16]{CraneFreemanToth}  extends to the case of a critical initial cluster size distribution $v_k(0) = \mathbb{P}(|T^{\pi_0}| = k)$
 once we have Propositions~\ref{prop_Smol_uniqueness_simplified}~and~\ref{prop_RathToth_conv}.

We give a doubly-inductive construction, using an array $(\epsilon_{\ell,m})_{\ell \in \mathbb{N}_0 ,m \in \mathbb{N}}$ of independent $\mathrm{Exp}(1)$ r.v.s and an independent array of $(U_{\ell,m})_{\ell \in \mathbb{N}_0 ,m \in \mathbb{N}}$ of independent $U[0,1]$ r.v.s. The pair $(a_0,C_0)$ is random and independent of these arrays, jointly distributed as the root age and the number of vertices of the age-driven MBP tree $T^{\pi_0}$, defined in Definition~\ref{def:branchingprocess}.

 We construct $(C_t)_{t>0}$ given $C_0$, by concatenating a countably infinite sequence of time-inhomogenous continuous-time branching processes, each of which almost surely explodes in finite time. The outer induction is indexed by $\ell \ge 0$ and for $\ell \geq 1$ the $\ell^{th}$ explosion time will be denoted $\tau_\ell$. Given some $\ell \geq 0$, the inner induction constructs the $\ell^{\textup{th}}$ branching process which jumps at times $\left(J_{\ell,m}\right)_{m \in \mathbb{N}}$.

We define $J_{0,0}=0$.  If $\ell \ge 1$ then we suppose that $\tau_{\ell}$ has already been defined and $C_t$ has already been constructed up to time $\tau_{\ell}$, and
 define $J_{\ell, 0} := \tau_{\ell}$ if $\ell \ge 1$.
If $\ell = 0$ then define $C_{J_{\ell,0}}: = C_0$, and if $\ell \ge 1$ then we define $C_{J_{\ell,0}} = C_{\tau_{\ell}} : = 1$. Given some $\ell \geq 0$, we recursively define  for each  $m \ge 1$
\begin{align}
\label{constr_first}
J_{\ell, m} & :=  J_{\ell,m-1} + \frac{\epsilon_{\ell,m}}{C_{J_{\ell,m-1}}} \,,   \\
C_t & :=  C_{J_{\ell,m-1}}, \quad \text{for}\;\;t \in [J_{\ell,m-1},J_{\ell,m})\,,   \\
\label{constr_third}  C_{J_{\ell,m}} & :=  C_{J_{\ell,m-1}} + \chi(J_{\ell, m}, U_{\ell,m}) \,,    \text{ where }
\end{align}
\begin{equation}\label{def_eq_chi_t}
  \chi : [0,\infty)\times [0,1) \to \mathbb{N}_+ \, , \quad     \chi(t,x)  =  k,\,\forall x\in\left[\sum_{\ell=1}^{k-1}v_\ell(t), \sum_{\ell=1}^k v_\ell(t) \right).
 \end{equation}
 So $\chi(t,\cdot)$ is the inverse of the cumulative distribution function which corresponds to the distribution $\underline v(t)$.
We define $\tau_{\ell+1} := \lim_{m \to \infty} J_{\ell,m}$. Lemma 3.16 in \cite{CraneFreemanToth} shows that the properties of $(v_k(s))_{k=1}^\infty$ (see in particular \eqref{v_k_polynomial_decay}) imply that $\tau_\ell$ is almost surely finite.
Since  $\tau_{\ell+1} - \tau_\ell \ge J_{\ell,1} - \tau_\ell = \epsilon_{\ell,1}$, we have $\tau_{\ell} \to \infty$ as $\ell \to \infty$. The inductive construction therefore constructs $C_t$ for all $t \ge 0$ almost surely.
We have $\lim_{h \to 0_+} C_{\tau_\ell-h}=\infty$ w.r.t.\ the usual topology of $\mathbb{N}$. This means that $\lim_{h \to 0_+} C_{\tau_\ell-h}=1$ w.r.t.\ the topology of $E$, so $(C_t)$ is continuous from the left at time $\tau_\ell$, since $C_{\tau_\ell}=1$ .

We now construct $a_t$ for all $t > 0$. Denote $\tau_0:=-a_0$ and define the age $a_t$ at time $t$ by
\begin{equation}\label{a_t_L_t_def}
a_t:= t - \tau_{L_t}, \quad \text{where} \quad
L_t:= \max\{\, \ell \geq 0 \, : \, \tau_\ell \leq t \, \}.
\end{equation}
This completes the construction of $(a_t,C_t)$.
The marginal $(C_t)$ is precisely the cluster growth process of \cite{CraneFreemanToth}, therefore the properties of $(C_t)$ discussed in Section \ref{subsection_cgp_w_a}
apply to it.

\subsection{Asymptotic independence of cluster processes}
\label{subsection_indep_of_clust_proc}

 \S\ref{subsection_indep_of_clust_proc} is devoted to the proof of Theorem \ref{T: asymptotic independence}. %Note that the details of this proof are independent of the rest of this paper, so they can be skipped at first reading.

Recall from Definition \ref{def_comps_sizes_mfff_t} the notion of the connected cluster $\mathcal{C}^n_t(i)$ of vertex $i$ at time $t$ as well as the size $C^n_t(i)$ of this cluster.

Our proof of Theorem \ref{T: asymptotic independence} proceeds by showing that the evolutions of $C^n_t(\rho^n_1)$ and $C^n_t(\rho^n_2)$ are well-approximated by the limiting cluster process whenever they are not too large, and are approximately independent except when $\mathcal{C}^n_t(\rho^n_1)= \mathcal{C}^n_t(\rho^n_2)$. The following lemma controls the probability of this pathological event.

\begin{lemma}[The clusters of $\rho^n_1$ and $\rho^n_2$ do not merge] \label{two cpts not same}
Let
\begin{equation}\label{merge_time_def}
\tau^n:=\inf\{\,t\ge 0\,:\, \mathcal C^n_t(\rho^n_1)=\mathcal C^n_t(\rho^n_2)\,\}=\inf\{\,t\ge 0\,:\,  \rho^n_2 \in  \mathcal C^n_t(\rho^n_1) \,\}  \, .
\end{equation}
Then for any $t_{\max}>0$ we have
\begin{equation}\label{eq:CvnotCw}
\lim_{n\rightarrow\infty}\Prob{ \, \tau^n\le t_{\max} \, }=0.
\end{equation}
\end{lemma}
\begin{proof}
We begin by observing that  the model  asymptotically almost surely never includes a giant component, i.e.,  for any $\epsilon>0$,
\begin{equation}\label{eq:nogiantcpt}
\lim_{n\rightarrow\infty}\Prob{\, \exists t\in[0,t_{\max}],\,\exists v\in[n]\text{ s.t. }C^n_t(v) >  n \epsilon \, }=0.
\end{equation}
The proof of \eqref{eq:nogiantcpt} follows from Propositions \ref{prop_Smol_uniqueness_simplified} and \ref{prop_RathToth_conv}, as we now explain.
Recall the notion of the solution $(v_k(t))$ of the critical forest fire equations (\eqref{smol_ff_eq}+\eqref{smol_bc}). Recall the notion
$v^n_k(t)$ of the fraction of vertices contained in components of size $k$ at time $t$ from \eqref{eq:defnvn}.
It follows from \eqref{smol_bc} and Dini's theorem that there exists $K \in \mathbb{N}$ such that $\sum_{k=1}^{K} v_k(t) \geq 1-\frac{\epsilon}{2}$ for all $t \leq t_{\max}$.
Then we can use \eqref{sup_conv_v_n_k_t_to_v_k_t} to derive
 \begin{equation}\label{most_mass_in_small_components}
  \lim_{n \to \infty } \mathbb{P}\left( \, \inf_{t \leq t_{\max}}\sum_{k=1}^{K} v^n_k(t) \geq 1-\epsilon\, \right)=1.
 \end{equation}
Now \eqref{eq:nogiantcpt} follows from \eqref{most_mass_in_small_components} as soon as we observe that if $ n \epsilon > K$ and if $C^n_t(v) >  n \epsilon $ for some $t \leq t_{\max}$ and $v \in [n]$ then
$\sum_{k=1}^{K} v^n_k(t) < 1-\epsilon$.

Next, consider the number $\mathcal{N}$ of fires involving $\rho^n_1$ up to time $t$. Since $C^n(\rho^n_1)$ spends an $\mathrm{Exp}(1+\lambda(n))$ holding time at size 1 after each fire, $\mathcal{N}$ is stochastically bounded by $1+\mathrm{Poisson}(t_{\max}(1+\lambda(n)))$. So we can bound in expectation the total number $\mathcal{M}^n_{t_{\max}}$ of vertices which are ever in the same component as $\rho^n_1$ on time interval $[0,t_{\max}]$ as
\begin{multline*}
\mathbb{E}\left( \mathcal{M}^n_{t_{\max}} \right)=
\E{\left|\bigcup_{t\in[0,t_{\max}]} \mathcal{C}^n_t(\rho^n_1)\right|}\le \\
n\Prob{ \, \exists t\in[0,t_{\max}], \, C^n_t(\rho^n_1)\ge  n \epsilon \, } +
 \left(1+t_{\max}(1+\lambda(n))\right) n \epsilon .
\end{multline*}
Since this holds for all $\epsilon>0$, using \eqref{eq:nogiantcpt} we obtain $\lim_{n\rightarrow\infty} \frac{1}{n}\mathbb{E}\left( \mathcal{M}^n_{t_{\max}} \right)=0 $,
from which \eqref{eq:CvnotCw} follows as soon as we note that $\rho^n_1$ and $\rho^n_2$ are i.i.d.\ with uniform distribution on $[n]$.
\end{proof}

\begin{proof}[Proof of Theorem \ref{T: asymptotic independence}]
The proof of Theorem \ref{CFTtheorem} (c.f.\ \cite[\S 4]{CraneFreemanToth}) involves a coupling of the pair $(C^n(\rho^n_1),C)$ such that these processes are close with respect to $d_E$ throughout $[0,t_{\max}]$ with high probability. This means they are equal at all times when neither is very large, and when either is very large then the other is either very large or equal to $1$. We will restate this coupling, and explain how it may be extended to a coupling of \newline $(C^n(\rho^n_1), C^n(\rho^n_2), C^{(1)}, C^{(2)})$, where the final two processes are independent copies of the limit cluster growth process $C$. We will follow the style of \cite{CraneFreemanToth} and state the coupling in explicit detail, from which it will be easy to derive the rigorous results required.

As in \cite[\S 4]{CraneFreemanToth}, we work with a modified version of the MFFF process, which alters the rates of clocks on edges within clusters, and introduces clocks on loop edges (that is, edges from a vertex to itself). Precisely, we demand:
\begin{equation}\label{clock_trick1}
\begin{array}{c}
\text{edge clocks \emph{within a cluster} ring at rate $2/n$ rather than $1/n$} \\
\text{and each loop edge has an independent clock with rate $1/n$.}
\end{array}
\end{equation}
These adjustments have no effect on the evolution of the cluster sizes, but simplify the statements of some Poisson process calculations to follow.

For each $n$ we define the (random) function $\chi^n: [0,\infty) \times [0,1) \to \mathbb{N}$ similarly to \eqref{def_eq_chi_t}:
\begin{equation}\label{eq:defnchin}
\chi^n(t,x)  =  k,\,\forall x\in\left[\sum_{\ell=1}^{k-1}v_\ell^n(t), \sum_{\ell=1}^k v^n_\ell(t) \right).
\end{equation}
 We also introduce a truncation parameter $K\in\N$, which will be taken large enough at the end.

We now describe the coupling of $(C^n(\rho^n_1),\tilde C,S)$, where $\tilde{C}$ will be a copy of the cluster growth process $C$, moreover we augment the coupling with a \emph{failure parameter} $S\in\{0,1\}$, which, informally, will switch from $0$ to $1$ at a stopping time when the coupling of $C^n(\rho^n_1)$ and $\tilde{C}$ ceases to be effective.

\begin{enumerate}[i)]
\item Sample $U_0\sim U[0,1]$, and set $\tilde C_0=\chi(0,U_0)$, and select $\rho^n_1$ uniformly at random from the vertices $v\in[n]$ for which $|\mathcal{C}^n_0(v)|=\chi^n(0,U_0)$, so $C^n_0(\rho^n_1)=\chi^n(0,U_0)$. From now on, we write $C^n_t$ for $C^n_t(\rho^n_1)$. If $\tilde C_0=C^n_0$, set $S=0$, otherwise, set $S=1$.
\item Enumerate the burning times and growth times of $\mathcal{C}^n_t(\rho^n_1)$ as follows. Set $I_{0,0}=0$ and inductively
\begin{equation}\label{eq:defnIa}
I_{\ell,0}:=\inf\{ \, s>I_{\ell-1,0}\,:\, \mathcal{C}^n(\rho^n_1)\text{ is burned at time }s \, \},\quad \ell \ge 1.
\end{equation}
So $I_{\ell,0}$ is the time of the $\ell$th fire involving $\rho^n_1$. Now, define $I_{\ell,1}\le I_{\ell,2}\le \ldots \le I_{\ell,g_\ell}$ to be the sequence of times between $I_{\ell,0}$ and $I_{\ell+1,0}$ at which an edge clock rings, for which at least one incident vertex is in $\mathcal{C}^n(\rho^n_1)$ at that time. This sequence includes times when a clock corresponding to an edge \emph{within} $\mathcal{C}^n_t(\rho^n_1)$ rings, which doesn't lead to a change in $C^n_t$.

\item For every $\ell \ge 0$, $ m \in[1,g_\ell]$, as a result of the assumption \eqref{clock_trick1}, the random variable $I_{\ell,m}-I_{\ell,m-1}$ has exponential distribution
with rate $C^n_{I_{\ell,m-1}}(\rho^n_1)$.  Moreover we may assume that the edge whose clock rings at time $I_{\ell,m}$ is $(i_{\ell,m},j_{\ell,m})$, where $i_{\ell,m}$ is sampled uniformly from $\mathcal{C}^n_{I_{\ell,m-1}}(\rho^n_1)$ and, independently, $j_{\ell,m}$ is sampled uniformly from $[n]$. So if we define $L^n_{\ell,m}:= \chi^n(I_{\ell,m}-, U_{\ell,m})$, where $(U_{\ell,m},\,\ell \ge 0, m \geq 1)$ are i.i.d.\ $U[0,1]$ variables, we may assume that $j_{\ell,m}$ is chosen uniformly from those vertices $v\in[n]$ for which $C^n_{I_{\ell,m}-}(v)=L^n_{\ell,m}$.

\item For each growth time $I_{\ell,m}$, if $\max(C^n_{I_{\ell,m}-},\tilde C_{I_{\ell,m}-})\le K$ and $S=0$, we do the following. Set $\tilde L_{\ell,m}:= \chi(I_{\ell,m}-, U_{\ell,m})$, and $\tilde C_{I_{\ell,m}}=\tilde C_{I_{\ell,m}-} + \tilde L_{\ell,m}$. Then
\begin{itemize}
\item if $\tilde L_{\ell,m}\ne L^n_{\ell,m}$, set $S=1$;
\item if $j_{\ell,m}\in \mathcal{C}^n_{I_{\ell,m}-}(\rho^n_1)$, then $C^n_{I_{\ell,m}}=C^n_{I_{\ell,m}-}$, and set $S=1$;
\item otherwise, we have $C^n_{I_{\ell,m}}=C^n_{I_{\ell,m}-}+L^n_{\ell,m}$ and $\tilde C_{I_{\ell,m}}=\tilde C_{I_{\ell,m}-} + \tilde L_{\ell,m}$ (in this case we have $\tilde L_{\ell,m} = L^n_{\ell,m}$) and we retain $S=0$.
\end{itemize}

\item Whenever $\max(C^n_t,\tilde C_t)> K$, or when $S=1$, then let the two processes evolve independently.
 During this, preserve the value of $S$, \emph{unless} one process jumps away from $1$ while the other process is greater than $K$. If this happens while $S=0$, move to the failure state $S=1$.

\item For each burning time $I_{\ell,0}$, if $C^n_{I_{\ell,0}-}\le K$, set $S=1$ thereafter.
\end{enumerate}

The effect of these dynamics is that while $S=0$, we either have $C^n_t=\tilde C_t\le K$, or
\begin{equation*}
(C^n_t,\tilde C_t) \in  \{K+1,\ldots\}^2\cup \left(\{1\}\times\{K+1,\ldots\}\right)\cup\left(\{K+1,\ldots\}\times\{1\}\right).
\end{equation*}
Crucially, at all times the evolution of $\tilde C$ matches with the construction of the cluster growth process given in \S \ref{subsection_construction_of_cgp}, and so we have
\begin{equation}\label{replicate_C}
\text{ $\tilde C \stackrel{d}= C$}.
\end{equation}

\medskip

For the purposes of this article, we require a similar coupling \newline $(C^n(\rho^n_1), C^n(\rho^n_2), \tilde C^{(1)}, \tilde C^{(2)})$ for two independently chosen watched vertices $\rho_1, \rho_2\in[n]$, which we now describe. The use of the `phantom clocks', though notationally heavy, ensures that independence of $\tilde C^{(1)},\tilde C^{(2)}$ is immediate.
\begin{itemize}
\item The evolution of $(C^n(\rho^n_1),\tilde C^{(1)})$ is exactly as in the original coupling, with failure parameter $S^{(1)}$, and the additional stipulation that whenever $\tilde C^{(1)}$'s evolution was specified to be independent of $C^n(\rho^n_1)$, it is now also independent of $C^n(\rho^n_2)$ and $\tilde C^{(2)}$.
\item The construction of $\tilde C^{(2)}$ will mostly be the same, but we require extra machinery to ensure that it stays independent of $\tilde C^{(1)}$. The  addition required is a family of independent \emph{phantom clocks}, all with rate $1/n$, attached to edges between $\mathcal{C}^n_t(\rho^n_1)$ and $\mathcal{C}^n_t(\rho^n_2)$ for $t \leq \tau^n$ (see \eqref{merge_time_def}). We define the rate of phantom clocks
to be zero for $t> \tau^n$.

\item Now, to determine the evolution of $\tilde C^{(2)}$ before $\tau^n$, we use the regular edge clocks from the forest-fire model associated to the edges between $\mathcal{C}^n_t(\rho^n_2)$ and $[n]\backslash \mathcal{C}^n_t(\rho^n_1)$, but use the phantom clocks between $\mathcal{C}^n_t(\rho^n_1)$ and $\mathcal{C}^n_t(\rho^n_2)$.

\item We define $\bar{\tau}^n$ to be the first time that a phantom clock rings, which we think of as a \emph{phantom growth time} for $\mathcal{C}^n(\rho^n_2)$. Note that $\bar{\tau}^n \geq \tau^n$ implies $\bar{\tau}^n=\infty$.  We also define $\bar I_{0,0}=0$ and the burning times $(\bar I_{\ell,0},\,\ell\ge 1)$ for $\mathcal{C}^n_t(\rho^n_2)$ exactly as in \eqref{eq:defnIa}, and the growth times $(\bar I_{\ell,m},\,\ell \ge 0, m=1,\ldots,\bar g_\ell)$ similarly, \emph{up to and including} the phantom growth time $\bar\tau^n$. It remains the case, as in iii), that whenever a regular (non-phantom) clock attached to the component of $\rho^n_2$ rings, the incident vertices may be given by a uniform choice from $\mathcal{C}^n_{\bar I_{\ell,m}-} (\rho^n_2)$, and a uniform choice from $[n]$.
\item We define the evolution of $\tilde C^{(2)}$, and its failure parameter $S^{(2)}$ using the burning times and growth times $(\bar I_{\ell,m})$ exactly as in iv), v), vi) above, moreover we add the stipulation that $S^{(2)}$ moves to the failure state $S^{(2)}=1$ at $\tau^n\wedge \bar\tau^n$. So $\tilde C^{(2)}$ evolves independently of the forest-fire model and of $\tilde C^{(1)}$ whenever $\max(C^n_t(\rho^n_2),\tilde C^{(2)})>K$, or when $S^{(2)}=1$.
%Note that before $\tau^n\wedge \bar\tau^n$, the coupling of $C^n(\rho^n_2)$ and $\tilde C^{(2)}$ is the same as the original coupling of $C^n(\rho^n_1)$ and $\tilde C^{(1)}$.
It follows, analogously to \eqref{replicate_C}, that $\tilde C^{(2)}$ is a copy of $C$.
\item The jumps of $\tilde C^{(1)}$ and $\tilde C^{(2)}$ are independent, as we now explain. In the case where all $C^n_t(\rho^n_1),C^n_t(\rho^n_2),\tilde C_t^{(1)},\tilde C_t^{(2)}\le K$ and $S^{(2)}=0$, the jumps of $\tilde C^{(1)}$ and $\tilde C^{(2)}$ are driven by disjoint sets of edge clocks in the underlying forest fire model. This is because $\mathcal C^n_t(\rho^n_1)$ and $\mathcal C^n_t(\rho^n_2)$ are disjoint if $S^{(2)}=0$, and the true edges between them drive $\tilde C^{(1)}$, while the phantom edges drive $\tilde C^{(2)}$. The other cases are independent by definition, since at least one of $\tilde C^{(1)},\tilde C^{(2)}$ is evolving independently of everything else. In other words,
\begin{equation}\label{iid_C_copies}
\text{ $\tilde C^{(1)}$ and $\tilde C^{(2)}$ are  i.i.d.\ copies of $C$.}
 \end{equation}
\end{itemize}

Theorem 1.7 of \cite{CraneFreemanToth} shows that for any $\epsilon>0$, for large enough $K$ we have
\begin{equation}\label{cft_coupling_effective}
\lim_{n\rightarrow\infty}\Prob{\sup_{t\in[0,t_{\max}]} d_E\left(C^n_t(\rho^n_1),\tilde C_t^{(1)} \right)>\epsilon}=0.
\end{equation}

Now note that while $S^{(2)} = 0$, the times $\bar \tau^n$ and $\tau^n$ are reached at the same (random, time-dependent) rates, being given by the number of real and phantom edges, respectively, between the two watched clusters.
Since the rate of the phantom clocks is set to zero after $\tau^n$, if $\bar{\tau}^n \geq \tau^n$ then $\bar{\tau}^n=\infty$.
 Hence $\bar\tau^n$ stochastically dominates $\tau^n$, and so using Lemma \ref{two cpts not same}, we have, as $n\rightarrow\infty$,
\begin{multline*} \Prob{\tau^n\wedge \bar\tau^n \le t_{\max}}= \\
\Prob{ \{ \tau^n  \le t_{\max} \} \cup \{ \bar\tau^n  \le t_{\max} \} }
\le 2\Prob{\tau^n\le t_{\max}}\;\rightarrow\; 0.
\end{multline*}

Then we may obtain, similarly to \eqref{cft_coupling_effective}, that
\begin{equation}\label{cft_2_coupling_effective}
\lim_{n\rightarrow\infty}\Prob{\sup_{t\in[0,t_{\max}]} d_E\left(C^n_t(\rho^n_2),\tilde C_t^{(2)}\right)>\epsilon}\le
 0+ \lim_{n\rightarrow\infty} \Prob{\tau^n \wedge \bar \tau^n\le t_{\max}} = 0.
\end{equation}
and so the proof of Theorem \ref{T: asymptotic independence} follows from \eqref{iid_C_copies}, \eqref{cft_coupling_effective} and \eqref{cft_2_coupling_effective}.
\end{proof}

\subsection{Coupling of cluster size and age processes}
\label{subsection_joint_conv_of_cluster_growth_and_age}

\S \ref{subsection_joint_conv_of_cluster_growth_and_age} is devoted to the proof of Proposition \ref{2rhocvg} and Corollary \ref{Cor: pointwise joint convergence}.

We want to deduce the result about the coupling of the age processes from the properies of the coupling of the cluster size processes of the watched vertices (i.e., Theorem \ref{T: asymptotic independence}). The $t=0$ case will follow from Proposition \ref{2agecvgprop}. As for the proof of the coupling of the whole age processes, the main idea is that
 the last burning time (c.f.\ Definition \ref{def:graphical_mfff}) of the vertex $\rho^n_i$
 can almost be read off from the evolution of the cluster size $C^n_t(\rho^n_i)$, since it is the
same as the elapsed time since the cluster size last dropped to one, except in the
unlikely case that lightning strikes the vertex $\rho_i^n$ while it is a singleton.

We will slightly modify the notion of the age of a vertex $j$ at time $t$ in the MFFF (c.f.\ Definition \ref{def:ages}).

\begin{definition}[Modified ages]
For $t>0$, let $\bar a^n_t(j)$ denote  the minimum of
$a^n_0(t)+t$ and the time that has elapsed since the last burning time $s$ of vertex $j$ for which $C^n_{s-}(j)>1$ (and we let $s:=-\infty$ if $[0,t]$ does not contain
such a burning time).
\end{definition}
 In other words, $\bar a^n_t(j)$ ignores instances where vertex $j$ is struck by lightning while a singleton.

We will prove Proposition \ref{2rhocvg} by first showing in \eqref{eq:no_lightning_singletons} that with asymptotically high probability the above modification of ages does not actually change the ages during $[0,t_{\max}]$. Then we will show that the analogue of Proposition \ref{2rhocvg} holds if we use the modified ages.

\begin{proof}[Proof of Proposition~\ref{2rhocvg}]
The rate of lightning striking the vertices $\rho^n_1$ and $\rho^n_2$ while they are singletons is $\lambda(n)\ll 1$ (by assumption \eqref{lightning_critical_regime}), thus we have
\begin{multline}\label{eq:no_lightning_singletons}
\Prob{\, \bar a^n_t(\rho_i^n)=a^n_t(\rho_i^n),\, i\in \{1,2\},\,   t\in[0,t_{\max}]\, } \ge \\ \exp\left(-2\lambda(n)t_{\max}\right) = 1-o_n(1).
\end{multline}

For $i \in \{1,2\}$ we have
\begin{align*}
\bar a^n_t(\rho^n_i)&= t - (-a^n_0(\rho^n_i)) \vee \sup\{ \, u\le t\,:\, C^n_{u-}(\rho^n_i)>1,\, C^n_u(\rho^n_i)=1 \, \},\\
 a^{(i)}_t &= t- (-a^{(i)}_0) \vee \sup\{ \, u\le t\,:\, C^{(i)}_{u-}>1,\, C^{(i)}_u=1 \, \}\,.
\end{align*}
The key point here is that in both cases the modified age process is a function of the initial age and the cluster size process.

We will improve the coupling that we gave in the proof of Theorem~\ref{T: asymptotic independence}, by augmenting it with initial ages.

We will show that for any $\epsilon, \delta > 0$ there exists $n_0$ such that for any $n \geq n_0$ the event that $ \sup_{t \in [0, t_{\max}]} \left| \bar a^n_t(\rho^n_i) /C^n_t(\rho^n_i) - a^{(i)}_t / C^{(i)}_t \right| \leq \epsilon$ occurs for both $i=1,2$ is greater than or equal to $1-8\delta$.

We use Proposition~\ref{2agecvgprop} to control the initial coupling at time $0$. We couple the pairs $(a_0^n(\rho_i^n),C_0^n(\rho_i^n))_{i \in \{1,2\}}$  to
  $(a_0^{(i)} ,C_0^{(i)})_{i \in \{1,2\}}$ so that for $n$ large enough we have $|a_0^n(\rho_i^n) - a_0^{(i)} | < \epsilon$ and $C_0^n(\rho_i^n) = C_0^{(i)}$ for both $i=1,2$ with probability at least $1-\delta$.

Choose $a_{\max}$ sufficiently large that $\pi_0([a_{\max},\infty)) < \delta$, so with probability at least $1-2\delta$ we have $a_0^{(i)} < a_{\max}$ for $i=1,2$, and for sufficiently large $n$, with probability at least $1 - 3\delta$ we have $a_0^n(\rho_i^n) < a_{\max}$ for $i=1,2$. Suppose that these events occur. Then all four ages cannot exceed $a_{\max} + t_{\max}$ during the time interval $[0,t_{\max}]$. Recall the threshold parameter $K$ in the coupling of Theorem \ref{T: asymptotic independence}, and set  $K > (a_{\max} + t_{\max})/\epsilon$.  Then whenever $C_t^{(i)} \ge K$, we have $a_t^{(i)}/C_t^{(i)} < \epsilon$, and similarly whenever $C_t^n(\rho_i^n) \ge K$, we have $\bar a_t^n(\rho_i^n)/C_t^n(\rho_i^n) < \epsilon$.
Take $n$ large enough that with this choice of $K$ the probability of coupling failure occurring at a strictly positive time $t\le t_{\max}$ is at most $\delta$.

Finally, we will bound the probability that some pair of  corresponding explosion and burning times in the coupling differ by at least $\epsilon$.
Because the inter-explosion times  of $C_t^{(i)}$ stochastically dominate a sequence of independent $\mathrm{Exp}(1)$ random variables, there exists $N$ depending on $t_{\max}$ and $\delta$ such that with probability at least $1-\delta$, $C_t^{(i)}$ explodes at most $N$ times during $[0,t_{\max}]$. If we take $K$ sufficiently large then (conditional on coupling success) the probability that the $\ell^{th}$ burning time of $\rho_i^n$ is at least $\epsilon$ away from the $\ell^{th}$ explosion time of $C^{(i)}$ is at most $\delta/(2N)$ for $n$ sufficiently large. This follows from \cite[Lemmas 4.1 and 4.2]{CraneFreemanToth}, which say that the watched cluster burns quickly with high probability once it is sufficiently large, and the corresponding statement for the cluster growth process. A union bound now shows that with probability at least $1-\delta$ all of the corresponding explosion and burning times in both $C^{(1)}$ and $C^{(2)}$ differ by less than $\epsilon$.

 Now the modified age processes can be read off as above from the initial ages and the cluster size processes.
 We find that with probability at least $1-8\delta$ we have for both $i=1$ and $i=2$ that $|a_t^{(i)} - \bar a_t^n(\rho_i^n)| < \epsilon$ throughout $[0,t_{\max}]$ except at times when exactly one of $C_t^{(i)}$ and $C_t^n(\rho_i^n)$ is at least $K$ and the other is equal to $1$. During those times we have $0 \le a_t^{(i)}/C_t^{(i)} < \epsilon$ and $0 \le \bar a_t^n(\rho_i^n)/C_t^n(\rho_i^n) < \epsilon$.

The proof of Proposition \ref{2rhocvg} follows using \eqref{eq:no_lightning_singletons} to replace $\bar a^n$ by $a^n$, with an asymptotically negligible probability of failure.
\end{proof}

\begin{proof}[Proof of Corollary~\ref{Cor: pointwise joint convergence}]
Let us fix  $0< \epsilon < t_{\max}^{-2}$.
Suppose that for $i=1,2$ we have  $$a_0^n(\rho_i^n) < \epsilon^{-1/2}, \quad a_0^{(i)} < \epsilon^{-1/2}\,,$$$$d_E\left(c_t^n(\rho_i^n),C_t^{(i)}\right) \le \epsilon, \quad \text{and} \quad \left|\frac{a_t^n(\rho_i^n)}{C_t^n(\rho_i^n)} - \frac{a_t^{(i)}}{C_t^{(i)}}\right| \le \epsilon\,.$$ (This happens with high probability if $\epsilon$ is small and then $n = n(\epsilon)$ is sufficiently large.) A calculation using the definition of the metric $d_E$ shows that one of the following three options must hold:
\begin{enumerate}\item $C_t^n(\rho_i^n) = C_t^{(i)} \le 1/\sqrt{\epsilon}$, in which case $\left|a_t^n(\rho_i^n) - a_t^{(i)}\right| \le \sqrt{\epsilon}$, or
\item $C_t^{(i)} \ge 1/\sqrt{\epsilon}$, or
\item $C_t^{(i)} = 1$ and $C_t^n(\rho_i^n) > 1/\epsilon$, in which case $$\frac{a_t^n(\rho_i^n)}{C_t^n(\rho_i^n) } < \epsilon(a_0^n(\rho_i^n) + t) \le 2\epsilon^{1/2}\,,$$ so  $a_t^{(i)} \le \epsilon + 2\epsilon^{1/2}$.
\end{enumerate}
For any fixed time $t > 0$, the probability that $C_t^{(i)}> 1/\sqrt{\epsilon}$ tends to $0$ as $\epsilon \to 0$ by Propositions \ref{prop_cft_cluster_marginal} and \ref{prop_Smol_uniqueness_simplified}, thus the probability of option 2 above is small.

The probability that $C^{(i)}$ has an explosion in the time interval $[t - \epsilon - 2\epsilon^{1/2},t]$ also tends to $0$ as $\epsilon \to 0$,
 because the expected number of explosions of $C_t$ in any interval $[a,b]$ is $\int_a^b \varphi(s)\,ds$ (see \cite[\S 3.4]{CraneFreemanToth}) and $\varphi(s)$ is bounded on $[0,t_{\max}]$ (see Proposition \ref{prop_Smol_uniqueness_simplified}), thus the probability of option 3 above is also small.

 It follows that option 1 above occurs with asymptotically high probability, completing the proof of Corollary~\ref{Cor: pointwise joint convergence}.
\end{proof}

\subsection{Convergence of empirical age evolution}
\label{subsection_conc_emp_age_distr}

The goal of \S\ref{subsection_conc_emp_age_distr} is to prove Theorem \ref{thm_convergence}.
We will prove that for any fixed $t \in [0,\infty)$, we have
 \begin{equation}\label{fixed_t_pi_convergence}
 \pi^n_t \stackrel{\mathbb{P}}{\Rightarrow} \pi_t,
 \end{equation}
 where $\pi_t$ is the distribution of $a_t$ (c.f.\ Definition \ref{def:cluster_growth_decorated_with_ages}).
By Definition \ref{def_conv_of_measures_simple},  we only need to check that for any
 bounded continuous function $f: [0,\infty) \to \mathbb{R}$ we have
\begin{align} \label{first_moment_pi_n_conv_to_pi}
\lim_{n \to \infty} \mathbb{E}\left( \int f(s)\, \mathrm{d} \pi^n_t(s) \right) & =  \int f(s)\, \mathrm{d} \pi_t(s),\\
\label{var_pi_n_conv_to_zero}
\lim_{n \to \infty} \mathrm{Var}\left( \int f(s)\, \mathrm{d} \pi^n_t(s) \right) & =  0.
\end{align}
First we show \eqref{first_moment_pi_n_conv_to_pi}.
 Corollary \ref{Cor: pointwise joint convergence}  gives $a^n_t(\rho^n_1)\Rightarrow  \pi_t$. This immediately gives
\eqref{first_moment_pi_n_conv_to_pi}, since
$\mathbb{E}\left( \int f(s)\, \mathrm{d} \pi^n_t(s) \right)=\mathbb{E}\left(f(a^n_t(\rho^n_1))\right)$.

Corollary \ref{Cor: pointwise joint convergence} also implies
\begin{multline}\label{second_mom_inged}
\lim_{n \to \infty } \mathbb{E}\left[ \left(\int f(s)\, \mathrm{d} \pi^n_t(s)\right)^2 \right]=
\lim_{n \to \infty}
\mathbb{E}\left[ f(a^n_t(\rho^n_1))f(a^n_t(\rho^n_2))  \right]=
 \\
\mathbb{E}\left[ f(a^{(1)}_t)f(a^{(2)}_t)  \right]=  \left( \int f(s)\, \mathrm{d} \pi_t(s) \right)^2.
\end{multline}
The required variance statement \eqref{var_pi_n_conv_to_zero} follows from \eqref{first_moment_pi_n_conv_to_pi} and \eqref{second_mom_inged}. This completes the proof of \eqref{fixed_t_pi_convergence}. The proof of Theorem \ref{thm_convergence} is complete.

% We are no longer assuming $\mathbb{E}(\int x\,d\pi^n_0(x)) \to \int x\,d\pi_0(x)$ as $n \to \infty$ or trying to prove that $\mathbb{E}(\int x\,d\pi^n_t(x)) \to \int x\,d\pi_t(x)$ as $n \to \infty$, so the next paragraph had to be removed.
%It remains to prove $\mathbb{E}(\int x\,d\pi^n_t(x)) \to \int x\,d\pi_t(x)$ as $n \to \infty$. By Lemma \ref{lemma_unif_int_fact} and the assumptions of Theorem \ref{thm_convergence},  $(a^n_0(\rho^n_1))_{n \in \mathbb{N}}$ is uniformly integrable. Since the ages increase by at most $t$, $(a^n_t(\rho^n_1))_{n \in \mathbb{N}}$ is also uniformly integrable. Now \eqref{fixed_t_pi_convergence} and another application of Lemma \ref{lemma_unif_int_fact} gives $\mathbb{E}(\int x\,d\pi^n_t(x)) \to \int x\,d\pi_t(x)$. The proof of Theorem \ref{thm_convergence} is complete.

\subsection{The local weak limit of MFFFA at time $t$}
\label{subsection_proof_of_thm_local_limit_mfff}

The goal of \S\ref{subsection_proof_of_thm_local_limit_mfff} is to prove  Theorem \ref{thm_local_limit_mfff}.

We begin by proving \eqref{local_weak_limit_in_ffm}.
 We have
$\pi^n_t \stackrel{\mathbb{P}}{\Rightarrow} \pi_t$ by Theorem \ref{thm_convergence}.
 Moreover by Theorem \ref{thm_mfff_inhomog},
the $\mathrm{MFFFA}(n,\underline a^n_0,\lambda)$  graph $\mathcal{G}^n_t$ is an age-inhomogeneous random graph
$\Gage(n,\underline a^n_t)$.
 We may thus use Proposition \ref{prop:localweakconv} to conclude that the graph
 $\mathcal{G}^n_t$ converges in probability in the local weak sense to the age-driven multitype branching process tree $T^{\pi_t}$ as $n\rightarrow\infty$.
 This proves \eqref{local_weak_limit_in_ffm}.

Next we prove \eqref{v_k_t_using_pi_t}. Recall from \eqref{eq:defnvn} that
$v^n_k(t)$ is the proportion of vertices of $\mathcal{G}^n_t$ in size $k$ components at time $t$.
 On the one hand, \eqref{local_weak_limit_in_ffm} implies that  $v^n_k(t)\stackrel{\mathbb{P}}{\rightarrow} \mathbb{P}(|T^{\pi_t}|=k)$ as $n \to \infty$
  for any $k \geq 1$ and  $t \geq 0$. On the other hand,  Propositions \ref{prop_Smol_uniqueness_simplified} and \ref{prop_RathToth_conv}  imply that $v^n_k(t)\stackrel{\mathbb{P}}{\rightarrow} v_k(t)$ as $n \to \infty$
  where $\underline{v}(\cdot)$ is the unique solution to the critical forest fire equations (\eqref{smol_ff_eq}+\eqref{smol_bc}) with initial state $\underline v(0)$ given by $v_k(0)=\mathbb{P}(|T^{\pi_0}|=k)$. This proves  \eqref{v_k_t_using_pi_t}.

It remains to  prove \eqref{age_crit_after_gel}.  By Proposition \ref{prop_Smol_uniqueness_simplified} we have $\sum_{k=1}^\infty v_k(t)=1$,
$$\sum_{k\ge 1}kv_k(t) <\infty,\; t<t_{\mathrm{gel}},\quad \text{and} \quad \sum_{k\ge 1}kv_k(t)=\infty,\;t\ge t_{\mathrm{gel}}.$$
By \eqref{v_k_t_using_pi_t} we have $\mathbb{P}(|T^{\pi}|<\infty )= \sum_{k=1}^\infty v_k(t)$ and $\mathbb{E}(|T^{\pi}|)=\sum_{k=1}^\infty kv_k(t)$,
thus \eqref{age_crit_after_gel} follows by Proposition \ref{prop_multitype_trichotomy}.
The proof of Theorem \ref{thm_local_limit_mfff} is complete.

%\begin{remark}\label{R: explicit density} One can show that
% if $\pi$ has no atoms then for any finite rooted tree $T$ equipped with with (distinct) vertex ages $a_v: v \in V(T)$, the density for $T^\pi$ to be isomorphic as an age-labelled rooted tree to $T$  is
% \[ \prod_{(v,w) \in E(T)} (a_v \wedge a_w) \prod_{v \in V(T)} \exp\left(-\int(s \wedge a_v)\,d\pi(s)\right)\,d\pi(a_v) \,.\]
%Note that re-root invariance of $T^\pi$ follows from the fact that this expression does not involve the information of which vertex is the root.Also
% Note that by integrating this density over a suitable polytope, then summing over isomorphism classes of rooted trees having $k$ vertices, we may express $\mathbb{P}(\left|T^\pi\right| = k)$ in closed form in terms of $\pi$. In the general case where $\pi$ may have atoms and the ages are not necessarily distinct, one must divide the above density by a product of factorials expressing the number of ways to embed $T$ as a rooted plane tree so that the children of each vertex are in non-decreasing order of age from left to right.
%\end{remark}

\subsection{$(C_t)$ and $(a_t)$ are intertwined}
\label{subsection_intertwine_C_a_cgp}
In
\S \ref{subsection_intertwine_C_a_cgp} we prove Lemma \ref{ConditionalCt}.

Theorem \ref{thm_convergence} gives $\pi^n_t \stackrel{ \mathbb{P} }{\Rightarrow} \pi_t$, and so we may apply Proposition \ref{2agecvgprop} to $\mathcal{G}^n_t$, to obtain
$$\left(a_t^n(\rho^n), C^n_t(\rho^n)\right) \;\Rightarrow \; \left( a(\rho),\, |T^{\pi_t}|\right),$$
where $\rho^n$ is uniformly distributed on $[n]$ and $a(\rho)$ is the age of the root in the age-labelled tree $T^{\pi_t}$. However, Corollary \ref{Cor: pointwise joint convergence} gives
$$\left(a_t^n(\rho^n),C^n_t(\rho^n)\right)\;\Rightarrow\; (a_t,C_t),$$
and so we know that conditional on $a_t$, we have $C_t\stackrel{d}=|T^{\pi_t}_{a_t}|$.

It remains to show that conditional on $a_t$, $C_t$ is independent of $(a_s)_{0\le s<t}$.

Recall that $0 < \tau_1 < \tau_2 < \dots$ are the consecutive explosion times of $C$ and $\tau_0 = -a_0$.  Recall from \eqref{a_t_L_t_def} that we denote by
 $L_t:= \max\{\, \ell \geq 0 \, : \, \tau_\ell \leq t \, \}$  the index of the last explosion time $\tau_{L_t}$ before $t$.

Assume that we are given the value of $a_t$. If $a_t \geq t$ then we know that $C$ did not explode on $[0,t]$, thus for all $s \in [0,t]$ we have $a_s= a_t -(t-s)$.
 So in this case, conditional on the value of $a_t$, the cluster size $C_t$ is independent of $(a_s)_{0\le s\le t}$, since the latter is fully determined by the value of $a_t$.

Now assume that we are given the value $a_t$, and we have $ a_t \in [0,t)$.  This is equivalent to assuming that $\tau_{L_t}=t-a_t$.
Under this assumption, one has $a_s=s-\tau_{L_t}$ for any $s \in [\tau_{L_t},t] $, therefore we only need to show that conditional on $\tau_{L_t}$, $C_t$ is independent of $(a_s)_{0\le s<\tau_{L_t}}$.

Note that $(a_s)_{0\le s<\tau_{L_t}}$ is determined by $a_0$ and $(C_s)_{0\le s<\tau_{L_t}}$ via \eqref{a_t_L_t_def}, so it is enough to show that
that conditional on $\tau_{L_t}$, $(C_s)_{\tau_{L_t} \leq s } $ is independent of $ \left( a_0, (C_s)_{0\le s<\tau_{L_t}} \right)$.
This statement follows from the construction of the cluster growth process with age given in \S \ref{subsection_construction_of_cgp}:
$(C_s)_{\tau_{L_t} \leq s } $ is constructed via \eqref{constr_first}-\eqref{constr_third} using the value of $J_{L_t,0}$ (i.e., $\tau_{L_t}$)
and the arrays
$(\epsilon_{\ell,m})_{ L_t \leq \ell , 1 \leq m}$ and $(U_{\ell,m})_{ L_t \leq \ell , 1 \leq m}$, and the joint distribution of these
arrays does not depend on $ \left( a_0, (C_s)_{0\le s<\tau_{L_t}} \right)$ (in particular, it does not depend on the number $L_t$ of explosions in $[0,t]$).
 The proof of Lemma \ref{ConditionalCt} is complete.

\subsection{The age process is Markov}
\label{subsection_age_is}
The goal of \S \ref{subsection_age_is} is to prove Lemma \ref{lemma:age_indeed_evolves_as_Markov} and also the formula \eqref{eq:phiexpression}.

The processes $(a_t)$ and $(A_t)$ have the same initial distribution, and
 as long as $t \in [0, t_{\mathrm{gel}})$,   both processes simply increase deterministically at speed $1$. Indeed, for $(A_t)$, this follows from
Definition \ref{def:A_t_age_Markov_proc} and our assumption that $\varphi(t)=0$ if $t \in [0, t_{\mathrm{gel}})$.
On the other hand, the cluster growth process $(C_t)$ does not explode in $[0,t_{\mathrm{gel}})$
(see \cite[Lemma 3.11]{CraneFreemanToth}), thus by  Definition \ref{def:cluster_growth_decorated_with_ages} we have $a_t=a_0+t$ for any $t <  t_{\mathrm{gel}}$.

Thus w.l.o.g.\ we may assume $t_{\mathrm{gel}}=0$ for the rest of \S\ref{subsection_age_is} (c.f.\ Remark \ref{remark_about_v_k_t_pi_t}\eqref{remark_assumtion_holds_at_time_t}).
 Under this assumption $\pi_t$ is age-critical for
any $t \geq 0$, by Theorem \ref{thm_local_limit_mfff}\eqref{age_crit_after_gel}.
 Moreover, the eigenfunction $\theta_t(\cdot)$ of $\mathcal{L}_{\pi_t}$ corresponding to the eigenvalue $\lambda=1$ is well-defined for all $t \geq 0$ by Lemma \ref{L: operator properties}\eqref{L:eigen}.

Our next step in the proof of Lemma \ref{lemma:age_indeed_evolves_as_Markov} is to show \eqref{eq:phiexpression}, that $\varphi(t)=(\int \theta_t(s)^3\, \mathrm{d}\pi_t(s))^{-1}$ when $t\ge t_{\mathrm{gel}}$.

\begin{proof}[Proof of \eqref{eq:phiexpression}]
Equation \eqref{f_T_expansion} of Lemma \ref{lemma_tree_s_gen_fn_asymp} implies
\begin{equation}\label{gen_fn_phi_t_pi_t_asymp}
\lim_{\varepsilon \to 0_+} \frac{1}{\sqrt{\varepsilon}} \mathbb{E}\left( 1 - (1-\varepsilon)^{|T^{\pi_t}|} \right)=\sqrt{2\phi_t},
\end{equation}
where, following \eqref{phi_theta_cube_int_def}, we take $\phi_t=1/\int \theta_t(s)^3\, \mathrm{d}\pi_t(s)$.
Let us compare this with \eqref{v_k_polynomial_decay}, which is equivalent to
\begin{equation}\label{gen_fn_phi_t_asymp}
\lim_{\varepsilon \to 0_+} \frac{1}{\sqrt{\varepsilon}} \left( 1 - \sum_{k=1}^{\infty}v_k(t)(1-\varepsilon)^k \right)=\sqrt{2\varphi(t)}
\end{equation}
by Example (c) of Theorem 4, chapter XIII.5 of \cite{Feller}.

By Theorem \ref{thm_local_limit_mfff}\eqref{v_k_t_using_pi_t} we have
\begin{equation}
\sum_{k=1}^{\infty}v_k(t)(1-\varepsilon)^k=
\mathbb{E}\left( (1-\varepsilon)^{|T^{\pi_t}|} \right), \qquad \varepsilon \in (0,1),
\end{equation}
 whence $\phi_t=\varphi(t)$, which proves
\eqref{eq:phiexpression}.
\end{proof}

\begin{proof}[Proof of Lemma \ref{lemma:age_indeed_evolves_as_Markov}] We now assume $t_{\mathrm{gel}}=0$.
The initial distribution of $a_0$ agrees with the initial distribution of $A_0$.
Both $a$ and $A$ increase deterministically at speed $1$ and occasionally jump back to zero.
Recalling the definition of $L_t$ from \eqref{a_t_L_t_def}, it follows that
$\tau_{L_t + 1}$ is the time of the first explosion after $t$. We
 only need to show
\begin{equation}\label{want_jump_rate}
\lim_{ h \to 0_+ } \frac{1}{h}\mathbb{P}\left(  \tau_{L_t + 1} \in [t,t+h] \, \big| \, (a(s))_{0 \leq s \leq t}  \right)=
\varphi(t)\theta_t(a(t))
\end{equation}
in order to conclude the proof of Lemma \ref{lemma:age_indeed_evolves_as_Markov}.
Let us define
\begin{equation}\label{eq: psi is a survival probability}
\psi_{t+h}(t):=\mathbb{P}\left( \tau_{L_t + 1}>t+h  \, \big| \, C_t=1 \right).
\end{equation}
(We explain the history of this notation in Remark \ref{remark_char_curves}).

Let us make the following observation, which follows from the branching structure underlying Definition \ref{def:cluster_growth_decorated_with_ages}. The cluster growth process started at time $t$ from the state
$C_t=k$ evolves until the first explosion according to the same dynamics as the sum of $k$ i.i.d.\
copies of the cluster growth process started at time $t$ from the state
$C_t=1$, where the explosion time of the sum is the minimum of the $k$ i.i.d. explosion times for its summands. From this observation we conclude
\begin{equation}\label{k_indep_copies_of_C}
\mathbb{P}\left( \tau_{L_t + 1}>t+h  \, \big| \, C_t=k \right)=\psi_{t+h}(t)^k,
\end{equation}
therefore by Lemma \ref{ConditionalCt} we have
\begin{equation}\label{trick_blabla}
\mathbb{P}\left(  \tau_{L_t + 1} \in [t,t+h] \, \big| \, (a(s))_{0 \leq s \leq t}  \right)=
1- \mathbb{E}\left( \,  \psi_{t+h}(t)^{|T^{\pi_t}_{a_t}|} \big| \, a_t \right).
\end{equation}
Thus in order to prove \eqref{want_jump_rate}, we only need to show
\begin{equation}\label{need_s_theta_gen_fn_asymp}
\lim_{h \to 0_+}\frac{1}{h}\left(  1- \mathbb{E}\left( \,  \psi_{t+h}(t)^{|T^{\pi_t}_{s}|} \right) \right)=
\varphi(t)\theta_t(s), \qquad s \geq 0.
\end{equation}

First we observe that
\begin{multline}\label{phi_explode_pi_t_gen_fn}
\varphi(t)\stackrel{\eqref{expected_number_of_explosions_and_phi} }{=}\lim_{ h \to 0_+ } \frac{1}{h}\mathbb{P}\left(  \tau_{L_t + 1} \in [t,t+h] \right)
\stackrel{ \eqref{eq_clgp_marginal_v_k_t}, \eqref{k_indep_copies_of_C} }{=} \\
 \lim_{ h \to 0_+ } \frac{1}{h} \left(1- \sum_{k=1}^{\infty} \psi_{t+h}(t)^k v_k(t)  \right)
 \stackrel{ \eqref{eq_v_k_t_using_pi_t} }{=}  \lim_{h \to 0_+}\frac{1}{h}\left(  1- \mathbb{E}\left( \,  \psi_{t+h}(t)^{|T^{\pi_t}|} \right) \right).
\end{multline}
 If we denote $\varepsilon(t,h)=1-\psi_{t+h}(t)$, then
\begin{multline}\label{intermediate_double_asymp}
\varphi(t) h + o(h)
\stackrel{ \eqref{phi_explode_pi_t_gen_fn} }{=}
 1- \mathbb{E}\left( \,  (1-\varepsilon(t,h))^{|T^{\pi_t}|} \right) \stackrel{ \eqref{eq:phiexpression}, \eqref{gen_fn_phi_t_pi_t_asymp} }{=} \\
 \sqrt{2 \varphi(t)} \sqrt{\varepsilon(t,h)} +
 o(\sqrt{\varepsilon(t,h)}), \qquad h \to 0_+\, .
\end{multline}
Taking the squares of both sides of \eqref{intermediate_double_asymp} and rearranging the terms, we obtain\\
$\frac{\varphi(t)}{2}h^2 +o(h^2)=\varepsilon(t,h) + o(\varepsilon(t,h))$, i.e.
\begin{equation}\label{short_char_curve}
\psi_{t+h}(t) = 1-\frac{\varphi(t)}{2}h^2 +o(h^2), \qquad h \to 0_+ \, .
\end{equation}
Equation \eqref{f_s_expansion} of Lemma \ref{lemma_tree_s_gen_fn_asymp} together with $\phi_t=\varphi(t)$ (i.e.,
\eqref{eq:phiexpression}) imply
\begin{equation}\label{gen_fn_s_phi_theta_asymp}
  \mathbb{E}\left(  (1-\varepsilon)^{|T^{\pi_t}_s|} \right)=
  1-\sqrt{\varepsilon}\sqrt{2\varphi(t)} \theta_t(s)+o(\sqrt{\varepsilon}) \qquad \varepsilon \to 0_+.
\end{equation}
Putting the above formulas together we obtain \eqref{need_s_theta_gen_fn_asymp}:
\begin{multline*}
 1- \mathbb{E}\left( \,  \psi_{t+h}(t)^{|T^{\pi_t}_{s}|} \right)\stackrel{ \eqref{short_char_curve} }{=}
1-\mathbb{E}\left( \,  \left(1- (\frac{\varphi(t)}{2}+o(1))h^2 \right)^{|T^{\pi_t}_{s}|} \right)\stackrel{ \eqref{gen_fn_s_phi_theta_asymp} }{=}\\
  \sqrt{\left(\frac{\varphi(t)}{2}+o(1)\right)h^2}\, \sqrt{2\varphi(t)}\, \theta_t(s)+o(h)=\varphi(t) \theta_t(s) h +o(h), \quad h \to 0_+ \, .
\end{multline*}
This completes the proof of Lemma \ref{lemma:age_indeed_evolves_as_Markov}.
\end{proof}

\begin{remark}\label{remark_char_curves}
For any $y > t_{\mathrm{gel}}$, the characteristic curve $\psi_{y}(t), t \in [0, y]$ is defined in \cite[Lemma 3.5]{CraneFreemanToth} to be the unique solution of the ODE
\begin{equation}
\psi_{y}(y)=1, \quad  \frac{\mathrm{d}}{\mathrm{d}t} \psi_y(t)=\psi_y(t)(1-X_t(\psi_y(t))),
\quad X_t(z):=\sum_{k=1}^{\infty} z^k v_k(t)
\end{equation}
for which $\psi_{y}(t)<1$ for all $t \in [0, y)$.
\\This definition and our definition \eqref{eq: psi is a survival probability} are equivalent by \cite[(3.16)]{CraneFreemanToth}.
\end{remark}

\subsection{Age differential equations}
\label{subsection_proof_of_age_pde}
The goal of \S\ref{subsection_proof_of_age_pde} is to prove Theorem~\ref{thm_age_pde}.

Let us first note that for any $t \in [0, t_{\mathrm{gel}})$ we have
$\frac{\partial}{\partial t} \int f(s)\mathrm{d}\pi_t(s) = \int f'(s)\mathrm{d}\pi_t(s)$, since
 $\pi_t$ is the distribution of the age $a_t$ by Definition \ref{def:cluster_growth_decorated_with_ages},
 and we have already seen at the beginning of \S \ref{subsection_age_is} that $a_t=a_0+t$ for any
 $t \in [0, t_{\mathrm{gel}})$.
  This proves \eqref{eq:ageDE}
for $t < t_{\mathrm{gel}}$. We can thus assume $t > t_{\mathrm{gel}} $ w.l.o.g.\ for the rest of \S \ref{subsection_proof_of_age_pde}.

Denote by $\widetilde{\pi}_t$ the distribution of $A_t$, c.f.\ Definition \ref{def:A_t_age_Markov_proc}.
Integrating \eqref{inf_gen_A_t} with respect to $\widetilde{\pi}_t$  we obtain the weak formulation of the Kolmogorov forward equation of $A$:
\begin{multline*}%\label{weak_kolmogorov_for_age}
\frac{\partial}{\partial t} \int f(s)\mathrm{d}\widetilde{\pi}_t(s) = \\
 \int f'(s)\mathrm{d}\widetilde{\pi}_t(s) - \varphi(t)\int f(s)\theta_t(s)\mathrm{d}\widetilde{\pi}_t(s) + \varphi(t) f(0) \int  \theta_t(s)\, \mathrm{d}\widetilde{\pi}_t(s)\, .
\end{multline*}

By Lemma \ref{lemma:age_indeed_evolves_as_Markov}, the marginal distributions of $a_t$ and $A_t$ agree, i.e.,
we have $\widetilde{\pi}_t=\pi_t$. Plugging this identity in the above equation and using
$\int \theta_t(s)\, \mathrm{d} \pi_t(s)=1$
we obtain the desired \eqref{eq:ageDE} for $t> t_{\mathrm{gel}}$. Noting that we have already proved \eqref{eq:phiexpression} in Section \ref{subsection_age_is},
the proof of  Theorem \ref{thm_age_pde} is complete.

\section*{Acknowledgements}

The work of E.C.\ is supported by the Heilbronn Institute for Mathematical Research.

The work of B.R.\ is partially supported by the Postdoctoral Fellowship NKFI-PD-121165 and grant NKFI-FK-123962 of NKFI (National Research, Development and Innovation Office), the Bolyai Research Scholarship of the Hungarian Academy of Sciences, the \'UNKP-18-4-BME-124 New National Excellence Program of the Ministry of Human Capacities and the ERC Synergy under Grant No. 810115 - DYNASNET

The work of D.Y.\ is supported by EPSRC doctoral training grant \newline EP/K503113, the German-Israeli Foundation of Scientific Research and Development Grant I-2494-304.6/2017, by ISF grant 1325/14, and by the Joan and Reginald Coleman--Cohen Fund.

The authors would like to thank both anonymous referees for their numerous constructive comments that improved the
quality of this paper.

\end{document}